\documentclass[11pt]{article}

\usepackage{bcc24}
\usepackage[nobysame]{amsrefs}
\usepackage{tikz,tkz-graph,hyperref}
\DeclareMathAlphabet {\mathbfit}{OML}{cmm}{b}{it}
\DeclareMathOperator{\tw}{tw}
\DeclareMathOperator{\cw}{cw}
\DeclareMathOperator{\pw}{pw}
\DeclareMathOperator{\rw}{rw}
\DeclareMathOperator{\claw}{claw}
\DeclareMathOperator{\diamondgraph}{diamond}
\DeclareMathOperator{\gem}{gem}
\DeclareMathOperator{\sun}{sun}
\DeclareMathOperator{\banner}{banner}
\DeclareMathOperator{\paw}{paw}

\DeclareMathOperator{\bull}{bull}
\DeclareMathOperator{\chair}{chair}

\DeclareMathOperator{\hammer}{hammer}
\DeclareMathOperator{\house}{house}
\DeclareMathOperator{\domino}{domino}

\newcommand{\NP}{{\sf NP}}

\newcommand{\GI}{{\sf GI}}
\newcommand\figurenames{Figures}
\newcommand{\ssi}{\subseteq_i}
\newcommand{\si}{\supseteq_i}
\newcounter{ctrclaim}[example]
\newcommand{\clm}[1]{\medskip\phantomsection\refstepcounter{ctrclaim}\noindent{\em Claim \thectrclaim. }{\em #1}\\}
\newcommand{\clmnonewline}[1]{\medskip\phantomsection\refstepcounter{ctrclaim}\noindent{\em Claim \thectrclaim. }{\em #1}}

\newtheorem{oproblem}[example]{Open Problem}
\newtheorem{observation}[example]{Observation}

\bcctitle{Clique-Width for Hereditary Graph Classes}
\bccshorttitle{Clique-Width for Hereditary Graph Classes}
\bccname{Konrad K. Dabrowski, Matthew Johnson and Dani\"el Paulusma}
\bccshortname{K.K. Dabrowski, M. Johnson, D. Paulusma} 
\bccaddressa{Department of Computer Science\\ Durham University\\ Durham, UK}
\bccemaila{\{konrad.dabrowski,matthew.johnson2,daniel.paulusma\}@durham.ac.uk}

\begin{document}
\makebcctitle

\begin{abstract}
Clique-width is a well-studied graph parameter owing to its use in understanding algorithmic tractability: if the clique-width of a graph class~${\cal G}$ is bounded by a constant, a wide range of problems that are \NP-complete in general can be shown to be polynomial-time solvable on~${\cal G}$.
For this reason, the boundedness or unboundedness of clique-width has been investigated and determined for many graph classes.
We survey these results for hereditary graph classes, which are the graph classes closed under taking induced subgraphs.
We then discuss the algorithmic consequences of these results, in particular for the {\sc Colouring} and {\sc Graph Isomorphism} problems.
We also explain a possible strong connection between results on boundedness of clique-width and on well-quasi-orderability by the induced subgraph relation for hereditary graph classes.
\end{abstract}

\newpage
\tableofcontents
\newpage
\listoffigures
\newpage

\section{Introduction}\label{s-intro}

Many decision problems are known to be \NP-complete~\cite{GJ79}, and it is generally believed that such problems cannot be solved in time polynomial in the input size.
For many of these hard problems, placing restrictions on the input (that is, insisting that the input has certain stated properties) can lead to significant changes in the computational complexity of the problem.
This leads one to ask fundamental questions: under which input restrictions can an \NP-complete problem be solved in polynomial time, and under which input restrictions does the problem remain \NP-complete? For problems defined on graphs, we can restrict the input to some special class of graphs that have some commonality.
The ultimate goal is to obtain complexity dichotomies for large families of graph problems, which tell us exactly for which graph classes a certain problem is efficiently solvable and for which it stays computationally hard.
Such dichotomies may not always exist if P$\neq$ \NP~\cite{La75}, but rather than solving problems one by one, and graph class by graph class, we want to discover general properties of graph classes from which we can determine the tractability or hardness of families of problems.

\subsection{Width Parameters}

One way to define a graph class is to use a notion of ``width'' and consider the set of graphs for which the width is bounded by a constant.
Though it will not be our focus, let us briefly illustrate this idea with the most well-known width parameter, \emph{treewidth}.
A \defword{tree decomposition} of a graph~$G=(V,E)$ is a tree~$T$ whose nodes are subsets of~$V$ and has the properties that, for each~$v$ in~$V$, the tree nodes that contain~$v$ induce a non-empty connected subgraph, and, for each edge~$vw$ in~$E$, there is at least one tree node that contains~$v$ and~$w$.
See \figurename~\ref{fig:tw} for an illustration of a graph and one of its tree decompositions.
\begin{figure}[b]
\begin{center} 
\begin{tabular}{ccc}
\begin{minipage}{0.25\textwidth}
\centering
\scalebox{0.7}{
{\begin{tikzpicture}[scale=1,rotate=0]
\GraphInit[vstyle=Normal]
\Vertex[x=1,y=6]{A}
\Vertex[x=3,y=6]{B}
\Vertex[x=1,y=4]{C}
\Vertex[x=3,y=4]{D}
\Vertex[x=0,y=2]{E}
\Vertex[x=2,y=2]{F}
\Vertex[x=2,y=0]{G}

\Edges(A,B,D,F,E,C,A)
\Edges(C,F)
\Edges(F,G)
\end{tikzpicture}}}
\end{minipage}
&
\hspace{1cm}
&
\begin{minipage}{0.25\textwidth}
\centering
\scalebox{0.7}{
{\begin{tikzpicture}[scale=1,rotate=0]
\GraphInit[vstyle=Normal]
\SetVertexNormal[MinSize=33pt]
\Vertex[x=1,y=6]{ABD}
\Vertex[x=1,y=4]{ACD}
\Vertex[x=1,y=2]{CDF}
\Vertex[x=0,y=0]{CEF}
\Vertex[x=2,y=0]{FG}

\Edges(ABD,ACD,CDF,CEF)
\Edges(CDF,FG)
\end{tikzpicture}}}
\end{minipage}\\
\end{tabular}
\end{center}
\caption{A graph, and a tree decomposition of the graph.}
\label{fig:tw}
\end{figure}
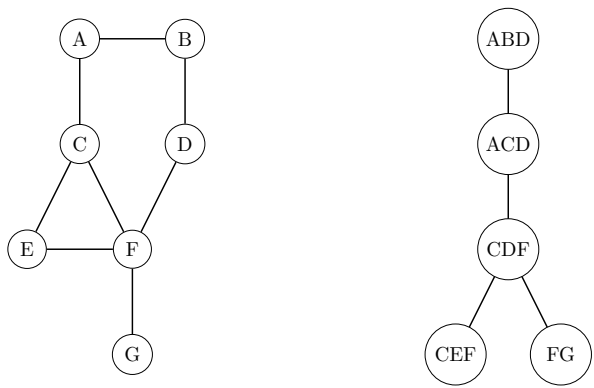
The sets of vertices that form the nodes of the tree are called \defword{bags} and the width of the decomposition is one less than the size of the largest bag.
The treewidth of~$G$ is the minimum width of its tree decompositions.
One can therefore define a class of graphs of bounded treewidth; that is, for some constant~$c$, the collection of graphs that each have treewidth at most~$c$.
The example in \figurename~\ref{fig:tw} has treewidth~$2$.
Moreover, it is easy to see that trees form exactly the class of graphs with treewidth~$1$.
Hence, the treewidth of a graph can be seen as a measure that indicates how close a graph is to being a tree.
Many graph problems can be solved in polynomial time on trees. For such problems it is natural to investigate whether restricting the problem to inputs that have bounded treewidth still yields algorithmic tractability.
An approach that often yields polynomial-time algorithms is to apply dynamic programming over the decomposition tree.
A disadvantage
of this approach
is that only sufficiently sparse graphs have bounded treewidth.

We further discuss reasons for focussing on width parameters in Section~\ref{s-mwp}, but let us first note that there are many alternative width parameters, each of which has led to progress in understanding the complexity of problems on graphs.

Clique-width, the central width parameter in our survey, is another well-known example, which has received significant attention since it was introduced by Courcelle, Engelfriet and Rozenberg~\cite{CER93} at the start of the 1990s.
Clique-width can be seen as a generalisation of treewidth that \emph{can} deal with dense graphs, such as complete graphs and complete bipartite graphs, provided these instances are sufficiently regular.
We will give explain this in Section~\ref{s-operations}, where we also give a formal definition, but, in outline, the idea is, given a graph~$G$, to determine how it can be built up vertex-by-vertex using four specific graph operations that involve assigning labels to the vertices.
The operations ensure that vertices labelled alike will keep the same label and thus, in some sense, behave identically.
The clique-width of~$G$ is the minimum number of different labels needed to construct~$G$ in this way.
Hence, if the clique-width of a graph~$G$ is small, we can decompose~$G$ into large sets of similarly behaving vertices, and these decompositions can be exploited to find polynomial-time algorithms (as we shall see later in this paper). 

We remark that many other width parameters have been defined including boolean-width, branch-width, MIM-width, MM-width, module-width, NLC-width, path-width and rank-width, to name just a few.
These parameters differ in strength, as we explain below; we refer to~\cite{Gu17,HOSG08,KLM09,Va12} for surveys on width parameters.

Given two width parameters~$p$ and~$q$, we say that~$p$ \defword{dominates}~$q$ if there is a function~$f$ such that $p(G)\leq f(q(G))$ for all graphs~$G$.
If~$p$ dominates~$q$ but not the reverse, then~$p$ is \defword{more general} than~$q$, as~$p$ is bounded for larger graph classes: 
whenever~$q$ is bounded for some graph class, then this is also the case for~$p$, but there exists an infinite family of graphs for which the reverse does not hold.
If~$p$ dominates~$q$ and~$q$ dominates~$p$, then~$p$ and~$q$ are \defword{equivalent}.
For instance, MIM-width is more general than boolean-width, clique-width, module-width, NLC-width and rank-width, all of which are equivalent~\cite{BTV11,Jo98,OS06,Ra08,Va12}.
The latter parameters are more general than the equivalent group of parameters branch-width, MM-width and treewidth, which are, in turn, more general than path-width~\cite{CO00,RS91,Va12}.
To give a concrete example, recall that the treewidth of the class of complete graphs is unbounded, in contrast to the clique-width.
More precisely, a complete graph on $n\geq 2$ vertices has treewidth~$n-1$ but clique-width~$2$.
As another example, the reason that rank-width and clique-width are equivalent is because the inequalities $\rw(G)\leq \cw(G)\leq 2^{\rw(G)+1}-1$ hold for every graph~$G$~\cite{OS06}.
These two inequalities are essentially tight~\cite{Ou17}, and, as such, the latter example also shows that two equivalent parameters may not necessarily be linearly, or even polynomially, related.

\subsection{Motivation for Width Parameters}\label{s-mwp}

The main computational reason for the large interest in width parameters is that many well-known \NP-complete graph problems become polynomial-time solvable if some width parameter is bounded.
There are a number of meta-theorems which prescribe general, sufficient conditions for a problem to be tractable on a graph class of bounded width.
For treewidth and equivalent parameters, such as branch-width and MM-width, one can use the celebrated theorem of Courcelle~\cite{Co90}.
This theorem, slightly extended from its original form, states that for every graph class of bounded treewidth, every problem definable in MSO$_2$ can be solved in time linear in the number of vertices of the graph.\footnote{MSO$_2$ refers to the fragment of second order logic where quantified relation symbols must have arity at most~$2$, 
which means that, with graphs, one can quantify over both sets of vertices and sets of edges. 
Many graph problems can be defined using MSO$_2$, such as deciding whether a graph has a $k$-colouring (for fixed~$k$) or a Hamiltonian path, but there are also problems that cannot be defined in this way.}
In order to use this theorem, one can use the linear-time algorithm of Bodlaender~\cite{Bo96} to verify whether a graph has treewidth at most~$c$ for any fixed constant~$c$ (that is, $c$ is not part of the input).
However, many natural graph classes, such as all those that contain graphs with arbitrarily large cliques, have unbounded treewidth.

We have noted that clique-width is more general than treewidth.
This means that if we have shown that a problem can be solved in polynomial time on graphs of bounded clique-width, then it can also be solved in polynomial time on graphs of bounded treewidth.
Similarly, if a problem is \NP-complete for graphs of bounded treewidth, then the same holds for graphs of bounded clique-width.
For graph classes of bounded clique-width, one can use several other meta-theorems.
The first such result is due to Courcelle, Makowsky and Rotics~\cite{CMR00}.
They proved that graph problems that can be defined in MSO$_1$ are linear-time solvable on graph classes of bounded clique-width.\footnote{MSO$_1$ is monadic second order logic with the use of quantifiers permitted on relations of arity~$1$ (such as vertices), but not of arity~$2$ (such as edges) or more.
Hence, MSO$_1$ is more restricted than MSO$_2$.
We refer to~\cite{CE12} for more information on MSO$_1$ and MSO$_2$.}
An example of such a problem is the well-known {\sc Dominating Set} problem.
This problem is to decide, for a graph~$G=(V,E)$ and integer~$k$, if~$G$ contains a set~$S\subseteq V$ of size at most~$k$ such that every vertex of $G-S$ has at least one neighbour in~$S$.\footnote{Several other problems, such as {\sc List Colouring} and {\sc Precolouring Extension} are polynomial-time solvable on graphs of bounded treewidth~\cite{JS97}, but stay \NP-complete on graph of bounded clique-width; the latter follows from results of~\cite{JS97} and~\cite{BDM09}, respectively; see also~\cite{GJPS17}.}

\subsection{Focus: Clique-Width}

As mentioned, in this survey we focus on clique-width.
Despite the usefulness of boundedness of clique-width, our understanding of clique-width itself is still very limited.
For example, although computing the clique-width of a graph is known to be \NP-hard in 
general~\cite{FRRS09},\footnote{It is also \NP-hard to compute treewidth~\cite{ACP87} and parameters equivalent to clique-width, such as NLC-width~\cite{GW05}, rank-width (see~\cite{HO08,Oum08}) and boolean-width~\cite{SV16}.} 
the complexity of computing the clique-width is open even on very restricted graph classes, such as unit interval graphs (see~\cite{HMR10} for some partial results).
To give another example, the complexity of determining whether a given graph has clique-width at most~$c$ is still open for every fixed constant $c\geq 4$.
On the positive side, see~\cite{CHLRR12} for a polynomial-time algorithm for $c=3$ and~\cite{EGW03} for a polynomial-time algorithm, for every fixed~$c$, on graphs of bounded treewidth.

To get a better handle on clique-width, many properties of clique-width, and relationships between clique-width and other graph parameters, have been determined over the years.
In particular, numerous graph classes of bounded and unbounded clique-width have been identified.
This has led to several dichotomies for various families of graph classes, which state exactly which graph classes of the family have
bounded or unbounded clique-width.
However, determining (un)boundedness of clique-width of a graph class is usually a highly non-trivial task, as it requires a thorough understanding of the structure of graphs in the class.
As such, there are still many gaps in our knowledge.

A number of results on clique-width are collected in the surveys on clique-width by Gurski~\cite{Gu17} and Kami\'nski, Lozin and Milani\v{c}~\cite{KLM09}.
Gurski focuses on the behaviour of clique-width (and NLC-width) under graph operations and graph transformations.
Kami\'nski, Lozin and Milani\v{c} also discuss results for special graph classes.
We refer to a recent survey of Oum~\cite{Ou17} for algorithmic and structural results on the equivalent width parameter rank-width.

\subsection{Aims and Outline}

In Section~\ref{s-pre} we introduce some basic terminology and notation that we use throughout the paper.
In Section~\ref{s-operations} we formally define clique-width.
In the same section we present a number of basic results on clique-width and explain two general techniques for showing that the clique-width of a graph class is bounded or unbounded.
For this purpose, in the same section we also list a number of graph operations that preserve (un)boundedness of clique-width for hereditary graph classes.

A graph class is \defword{hereditary} if it is closed under taking induced subgraphs, or equivalently, under vertex deletion.
Due to its natural definition, the framework of hereditary graph classes captures many well-known graph classes, such as bipartite, chordal, planar, interval and perfect graphs; we refer to the textbook of Brandst{\"a}dt, Le and Spinrad~\cite{BLS99} for a survey.
As we shall see, boundedness of clique-width has been particularly well studied for hereditary graph classes.
We discuss the state-of-the-art and other known results on boundedness of clique-width for hereditary graph classes in Section~\ref{s-hereditary}.
This is all related to our first aim: to update the paper of Kami\'nski, Lozin and Milani\v{c}~\cite{KLM09} from 2009 by surveying, in a systematic way, known results and open problems on boundedness of clique-width for hereditary graph classes.

Our second aim is to discuss algorithmic implications of the results from Section~\ref{s-hereditary}.
We do this in Section~\ref{s-algo} by focussing on two well-known problems.
We first discuss implications for the {\sc Colouring} problem, which is well known to be \NP-complete~\cite{Lo73}.
We focus on (hereditary) graph classes defined by two forbidden induced subgraphs.
Afterwards, we consider the algorithmic consequences for the {\sc Graph Isomorphism} problem.
This problem can be solved in quasi-polynomial time~\cite{Ba16}.
It is not known if {\sc Graph Isomorphism} can be solved in polynomial time, but it is not \NP-complete unless the polynomial hierarchy collapses~\cite{Sc88}.
As such, we define the complexity class \GI, which consists of all problems that can be polynomially reduced to {\sc Graph Isomorphism} and a problem in \GI\ is \GI-complete if {\sc Graph Isomorphism} can be polynomially reduced to it.
The {\sc Graph Isomorphism} problem is of particular interest, as there are similarities between proving unboundedness of clique-width of some graph class and proving that {\sc Graph Isomorphism} stays {\sc \GI}-complete on this class~\cite{Sc17}.

Our third aim is to discuss a conjectured relationship between boundedness of clique-width and well-quasi-orderability by the induced subgraph relation.
If it can be shown that a graph class is well-quasi-ordered, we can apply several powerful results to prove further properties of the class.
This is, for instance, illustrated by the Robertson-Seymour Theorem~\cite{RS04-Wagner}, which states that the set of all finite graphs is well-quasi-ordered by the minor relation.
This result makes it possible to test in cubic time whether a graph belongs to some given minor-closed graph class~\cite{RS95} (see~\cite{KKR12} for a quadratic algorithm).
For the induced subgraph relation, it is easy to construct examples of hereditary graph classes that are not well-quasi-ordered.
Take, for instance, the class of graphs of degree at most~$2$, which contains an infinite anti-chain, namely the set of all cycles.

If every hereditary graph class that is well-quasi-ordered by the induced subgraph relation also has bounded clique-width, then all algorithmic consequences of having bounded clique-width would also hold for being well-quasi-ordered by the induced subgraph relation.
However, Lozin, Razgon and Zamaraev~\cite{LRZ18} gave a negative answer to a question of Daligault, Rao and Thomass{\'e}~\cite{DRT10} about this implication, by presenting a hereditary graph class of unbounded clique-width that is nevertheless well-quasi-ordered by the induced subgraph relation.
Their graph class can be characterized only by infinitely many forbidden induced subgraphs.
This led the authors of~\cite{LRZ18} to conjecture that every finitely defined hereditary graph class that is well-quasi-ordered by the induced subgraph relation has bounded clique-width, which, if true, would still be very useful.
All known results agree with this conjecture, and we survey these results in Section~\ref{s-wqo}.
In the same section we explain that  the graph operations given in Section~\ref{s-operations} do not preserve well-quasi-orderability by the induced subgraph relation. 
However, we also explain that a number of these operations can be used for a stronger property, namely well-quasi-orderability by the labelled induced subgraph relation.

In Section~\ref{s-con} we conclude our survey with a list of other relevant open problems.
There, we also discuss some variants of clique-width, including linear clique-width and power-bounded clique-width.

\section{Preliminaries}\label{s-pre}

Throughout the paper we consider only finite, undirected graphs without multiple edges or self-loops.

Let $G=(V,E)$ be a graph.
The \defword{degree} of a vertex~$u\in V$ is the size of its neighbourhood $N(u)=\{v\in V\; |\; uv\in E\}$.
For a subset $S\subseteq V$, the graph~$G[S]$ denotes the subgraph of~$G$ \defword{induced by}~$S$, which is the graph with vertex set~$S$ and an edge between two vertices $u,v\in S$ if and only if $uv\in E$.
If~$F$ is an induced subgraph of~$G$, then we denote this by $F\ssi G$.
Note that~$G[S]$ can be obtained from~$G$ by deleting the vertices of $V\setminus S$.
The \defword{line graph} of~$G$ is the graph with vertex set~$E$ and an edge between two vertices~$e_1$ and~$e_2$ if and only if~$e_1$ and~$e_2$ share a common end-vertex in~$G$.

An \defword{isomorphism} from a graph~$G$ to a graph~$H$ is a bijective mapping $f:V(G)\to V(H)$ such that there is an edge between two vertices~$u$ and~$v$ in~$G$ if and only if there is an edge between~$f(u)$ and~$f(v)$ in~$H$.
If such an isomorphism exists then~$G$ and~$H$ are said to be \defword{isomorphic}.
We say that~$G$ is \defword{$H$-free} if~$G$ contains no induced subgraph isomorphic to~$H$.

Let $G=(V,E)$ be a graph.
A set $K\subseteq V$ is a \defword{clique} of~$G$ and~$G[K]$ is \defword{complete} if there is an edge between every pair of vertices in~$K$.
If~$G$ is connected, then a vertex $v\in V$ is a \defword{cut-vertex} of~$G$ if $G[V\setminus \{v\}]$ is disconnected, and a clique $K\subset V$ is a \defword{clique cut-set} of~$G$ if $G[V\setminus K]$ is disconnected.
If~$G$ is connected and has at least three vertices but no cut-vertices, then~$G$ is \defword{$2$-connected}.
A maximal induced subgraph of~$G$ that has no cut-vertices is a \defword{block} of~$G$.
If~$G$ is connected and has no clique cut-set, then~$G$ is an \defword{atom}.

The graphs~$C_n$, $P_n$ and~$K_n$ denote the cycle, path and complete graph on~$n$ vertices, respectively.
The \defword{length} of a path or a cycle is the number of its edges.
The \defword{distance} between two vertices~$u$ and~$v$ in a graph~$G$ is the length of a shortest path between them.
For an integer~$r\geq 1$, the \defword{$r$-th power} of~$G$ is the graph with vertex set~$V(G)$ and an edge between two vertices~$u$ and~$v$ if and only if~$u$ and~$v$ are at distance at most~$r$ from each other in~$G$.

If~$F$ and~$G$ are graphs with disjoint vertex sets, then the \defword{disjoint union} of~$F$ and~$G$ is the graph
$G+\nobreak F=(V(F)\cup V(G),E(F)\cup E(G))$.
The disjoint union of~$s$ copies of a graph~$G$ is denoted~$sG$.
A \defword{forest} is a graph with no cycles, that is, every connected component is a \defword{tree}.
A forest is \defword{linear} if it has no vertices of degree at least~$3$, or equivalently, if it is the disjoint union of paths.
A \defword{leaf} in a tree is a vertex of degree~$1$.
In a \defword{complete binary} tree all non-leaf vertices have degree~$3$.

Let~$S$ and~$T$ be disjoint vertex subsets of a graph~$G=(V,E)$.
A vertex~$v$ is \defword{(anti-)complete} to~$T$ if it is (non-)adjacent to every vertex in~$T$.
Similarly, $S$ is \defword{(anti-)complete} to~$T$ if every vertex in~$S$ is (non-)adjacent to every vertex in~$T$.
A set of vertices~$M$ is a \defword{module} of~$G$ if every vertex of~$G$ that is not in~$M$ is either complete or anti-complete to~$M$.
A module of~$G$ is \defword{trivial} if it contains zero, one or all vertices of~$G$, otherwise it is \defword{non-trivial}.
We say that~$G$ is \defword{prime} if every module of~$G$ is trivial.

A graph~$G$ is \defword{bipartite} if its vertex set can be partitioned into two (possibly empty) subsets~$X$ and~$Y$ such that every edge of~$G$ has one end-vertex in~$X$ and the other one in~$Y$.
If~$X$ is complete to~$Y$, then~$G$ is \defword{complete bipartite}.
For two non-negative integers~$s$ and~$t$, we
denote the complete bipartite graph with partition classes of size~$s$ and~$t$, respectively, by~$K_{s,t}$. 
The graph~$K_{1,t}$ is also known as the $(t+1)$-vertex \defword{star}.
The \defword{subdivision} of an edge~$uv$ in a graph replaces~$uv$ by a new vertex~$w$ and edges~$uw$ and~$vw$.
We let~$K_{1,t}^+$ and~$K_{1,t}^{++}$ be the graphs obtained from~$K_{1,t}$ by subdividing one of its edges once or twice, respectively (see \figurename~\ref{fig:K1t+} for examples).
\begin{figure}
\begin{center}
\begin{tabular}{cc}
\begin{minipage}{0.29\textwidth}
\begin{center}
\begin{tikzpicture}[xscale=-1]
\GraphInit[vstyle=Simple]
\SetVertexSimple[MinSize=6pt]
\Vertex[x=0,y=0]{a0}
\Vertex[x=0,y=0.5]{a1}
\Vertex[x=0,y=1]{a2}
\Vertex[x=0,y=1.5]{a3}
\Vertex[x=0,y=2]{a4}
\Vertex[x=-1,y=1]{x}
\Vertex[x=1,y=0]{b0}
\Edges(a1,x,a0,b0)
\Edges(a2,x,a3)
\Edges(x,a4)
\end{tikzpicture}
\end{center}
\end{minipage}
&
\begin{minipage}{0.29\textwidth}
\begin{center}
\begin{tikzpicture}[scale=1]
\GraphInit[vstyle=Simple]
\SetVertexSimple[MinSize=6pt]
\Vertex[x=0,y=0]{a0}
\Vertex[x=0,y=0.5]{a1}
\Vertex[x=0,y=1]{a2}
\Vertex[x=0,y=1.5]{a3}
\Vertex[x=0,y=2]{a4}
\Vertex[x=-1,y=1]{x}
\Vertex[x=1,y=0]{b0}
\Vertex[x=2,y=0]{c0}
\Edges(a1,x,a0,b0,c0)
\Edges(a2,x,a3)
\Edges(x,a4)
\end{tikzpicture}
\end{center}
\end{minipage}\\
\\
$K_{1,t}^+$~($t=\nobreak 5$~shown)&
$K_{1,t}^{++}$~($t=\nobreak 5$~shown)
\end{tabular}
\end{center}
\caption{\label{fig:K1t+}The graphs~$K_{1,t}^+$ and~$K_{1,t}^{++}$.}
\end{figure}
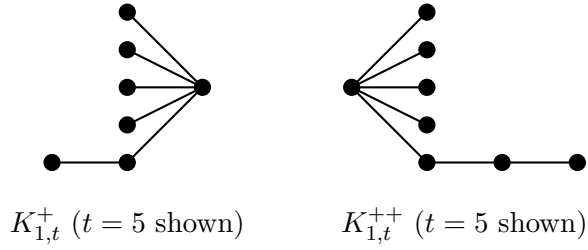

A graph is \defword{complete $r$-partite}, for some $r\geq 1$, if its vertex set can be partitioned into~$r$ independent sets $V_1,\ldots,V_r$ such that there exists an edge between two vertices~$u$ and~$v$ if and only if~$u$ and~$v$ do not belong to the same set~$V_i$.
Note that a non-empty graph is complete $r$-partite for some $r\geq 1$ if and only if it is $(P_1+\nobreak P_2)$-free.

Let $G=(V,E)$ be a graph.
Its \defword{complement}~$\overline{G}$ is the graph with vertex set~$V$ and an edge between two vertices~$u$ and~$v$ if and only if~$uv$ is not an edge of~$G$.
We say that~$G$ is \defword{self-complementary} if~$G$ is isomorphic to~$\overline{G}$.
The complement of a bipartite graph is a \defword{co-bipartite} graph.

\begin{figure}
\begin{center}
\begin{tabular}{ccccc}
\begin{tikzpicture}[scale=1]
\GraphInit[vstyle=Simple]
\SetVertexSimple[MinSize=6pt]
\Vertex[x=0,y=0.5]{a}
\Vertex[x=0.866025403784438646763723,y=0]{b}
\Vertex[x=0,y=-0.5]{c}
\Vertex[x=-0.866025403784438646763723,y=0]{d}
\Edges(a,b,c,d,a,c)
\end{tikzpicture}
&
\begin{tikzpicture}[scale=1.5,rotate=90]
\GraphInit[vstyle=Simple]
\SetVertexSimple[MinSize=6pt]
\Vertex[x=0,y=0]{x}
\Vertex[a=-45,d=1]{a}
\Vertex[a=-15,d=1]{b}
\Vertex[a=15,d=1]{c}
\Vertex[a=45,d=1]{d}
\Edges(a,x,b)
\Edges(c,x,d)
\Edges(a,b,c,d)
\end{tikzpicture}
&
\begin{tikzpicture}[scale=1]
\GraphInit[vstyle=Simple]
\SetVertexSimple[MinSize=6pt]
\Vertex[x=0.5,y=0.70710678118]{x}
\Vertex[x=0,y=0]{a}
\Vertex[x=1,y=0]{b}
\Vertex[x=0,y=-1]{c}
\Vertex[x=1,y=-1]{d}
\Edges(a,x,b)
\Edges(a,b,d,c,a)
\end{tikzpicture}
&
{\begin{tikzpicture}[scale=1,rotate=90]
\GraphInit[vstyle=Simple]
\SetVertexSimple[MinSize=6pt]
\Vertex[x=0,y=0]{x00}
\Vertex[x=1,y=0]{x10}
\Vertex[x=2,y=0]{x20}
\Vertex[x=0,y=1]{x01}
\Vertex[x=1,y=1]{x11}
\Vertex[x=2,y=1]{x21}
\Edges(x10,x00,x01,x11,x21,x20,x10,x11)
\end{tikzpicture}}\\\\
$\diamondgraph=\overline{2P_1+P_2}$ & $\gem = \overline{P_1+P_4}$ & $\house=\overline{P_5} $ & $\domino$\\\\
\begin{tikzpicture}[scale=1]
\GraphInit[vstyle=Simple]
\SetVertexSimple[MinSize=6pt]
\Vertex[x=0,y=0]{x}
\Vertex[a=0,d=1]{a1}
\Vertex[a=0,d=2]{a2}
\Vertex[a=0,d=3]{a3}
\Vertex[a=240,d=1]{b1}
\Vertex[a=240,d=2]{b2}
\Vertex[a=120,d=1]{c1}
\Edges(c1,x,a1,a2,a3)
\Edges(x,b1,b2)
\end{tikzpicture}
&
\begin{tikzpicture}[scale=1,rotate=90]
\GraphInit[vstyle=Simple]
\SetVertexSimple[MinSize=6pt]
\Vertices{circle}{a,b,c,d,e}
\Edges(a,b,c,d,e,a)
\end{tikzpicture}
&
\begin{tikzpicture}[scale=1,rotate=90]
\GraphInit[vstyle=Simple]
\SetVertexSimple[MinSize=6pt]
\Vertices{circle}{a,b,c,d,e}
\Edges(a,b,c,d,e,a,c,e,b,d,a)
\end{tikzpicture}
&
\begin{tikzpicture}[scale=1,rotate=90]
\GraphInit[vstyle=Simple]
\SetVertexSimple[MinSize=6pt]
\Vertices{circle}{a,b,c,d,e}
\Edges(d,e,a,b,c)
\end{tikzpicture}\\\\ 
$S_{h,i,j}$ & $C_n$ & $K_n$ & $P_n$ \\
($(h,i,j)=(1,2,3)$ shown) & ($n=5$ shown) & ($n=5$ shown) & ($n=5$ shown)\\\\
\end{tabular}
\end{center}
\caption{\label{fig:particular-graphs}Some common graphs used throughout the paper.}
\end{figure}
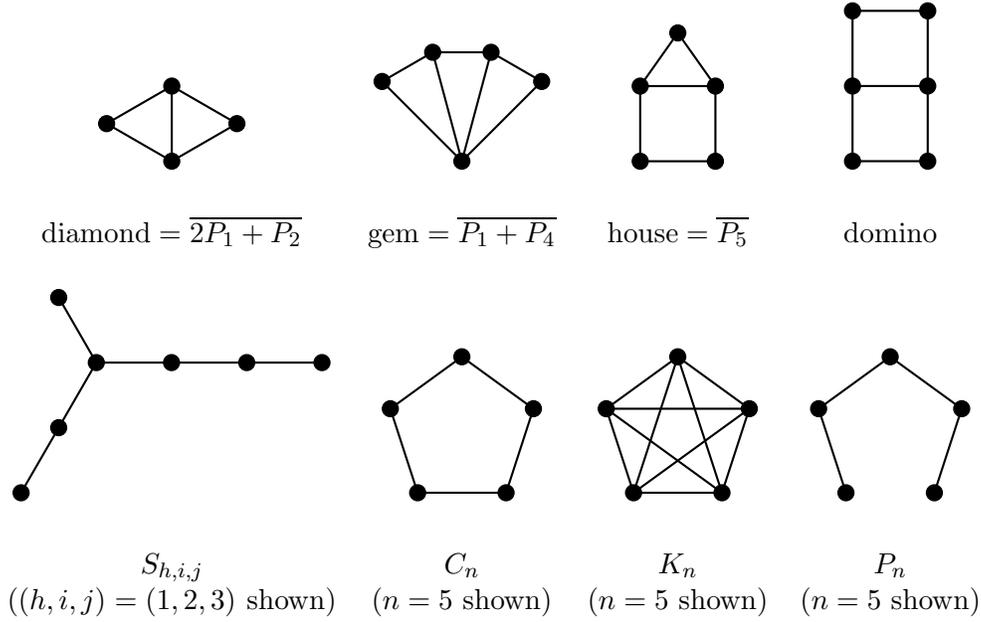

The graphs $K_{1,3}$, $\overline{2P_1+P_2}$, $\overline{P_1+P_4}$, and~$\overline{P_5}$ are also known as the \defword{claw}, \defword{diamond}, \defword{gem}, and \defword{house}, respectively.
The latter three graphs are shown in \figurename~\ref{fig:particular-graphs}, along with the \defword{domino}.
The graph~$S_{h,i,j}$, for $1\leq h\leq i\leq j$, denotes the \defword{subdivided claw}, which is the tree with one vertex~$x$ of degree~$3$ and exactly three leaves, which are of distance~$h$,~$i$ and~$j$ from~$x$, respectively.
Note that $S_{1,1,1}=K_{1,3}$, $S_{1,1,2}=K_{1,3}^+$ and $S_{1,1,3}=K_{1,3}^{++}$.
See \figurename~\ref{fig:particular-graphs} for an example.
We let~${\cal S}$ be the class of graphs every connected component of which is either a subdivided claw or a path on at least one vertex.
The graph~$T_{h,i,j}$ with $0\leq h\leq i\leq j$ denotes the triangle with pendant paths of length~$h$, $i$ and $j$, respectively.
That is, $T_{h,i,j}$ is the graph with vertices $a_0,\ldots,a_h$, $b_0,\ldots,b_i$ and $c_0,\ldots,c_j$ and edges $a_0b_0$, $b_0c_0$, $c_0a_0$, $a_pa_{p+1}$ for $p \in \{0,\ldots, h-\nobreak1\}$, $b_pb_{p+1}$ for $p\in\{0,\ldots,i-\nobreak 1\}$ and $c_pc_{p+1}$ for $p\in\{0,\ldots,j-\nobreak 1\}$.
Note that $T_{0,0,0}=C_3=K_3$.
The graphs $T_{0,0,1}=\overline{P_1+P_3}$, $T_{0,1,1}$, $T_{1,1,1}$ and~$T_{0,0,2}$ are also known as the \defword{paw}, \defword{bull}, \defword{net} and \defword{hammer}, respectively; see also \figurename~\ref{fig:Thij-examples}.
Also note that~$T_{h,i,j}$ is the line graph of~$S_{h+1,i+1,j+1}$.
We let~${\cal T}$ be the class of graphs that are the line graphs of graphs in~${\cal S}$.
Note that~${\cal T}$ contains every graph~$T_{h,i,j}$ and every path (as the line graph of~$P_t$ is~$P_{t-1}$ for $t\geq 2$).

\begin{figure}
\begin{center}
\begin{tabular}{ccccc}
\begin{minipage}{0.215\textwidth}
\begin{center}
\begin{tikzpicture}
\coordinate (a0) at (0.5,0) ;
\coordinate (b0) at (-0.5,0) ;
\coordinate (c0) at (0,0.86602540378) ;
\coordinate (c1) at (0,1.86602540378) ;
\draw [fill=black] (a0) circle (3pt) ;
\draw [fill=black] (b0) circle (3pt) ;
\draw [fill=black] (c0) circle (3pt) ;
\draw [fill=black] (c1) circle (3pt) ;
\draw [thick] (c1) -- (c0) -- (b0) -- (a0) -- (c0);
\end{tikzpicture}
\end{center}
\end{minipage}
&
\begin{minipage}{0.215\textwidth}
\begin{center}
\begin{tikzpicture}
\coordinate (a0) at (0,0) ;
\coordinate (b0) at (60:1) ;
\coordinate (c0) at (120:1) ;
\coordinate (b1) at (60:2) ;
\coordinate (c1) at (120:2) ;
\draw [fill=black] (a0) circle (3pt) ;
\draw [fill=black] (b0) circle (3pt) ;
\draw [fill=black] (c0) circle (3pt) ;
\draw [fill=black] (b1) circle (3pt) ;
\draw [fill=black] (c1) circle (3pt) ;
\draw [thick] (c1) -- (c0) -- (b0) -- (a0) -- (c0);
\draw [thick] (b1) -- (b0);
\end{tikzpicture}
\end{center}
\end{minipage}
&
\begin{minipage}{0.215\textwidth}
\begin{center}
\begin{tikzpicture}
\coordinate (a0) at (90:0.57735026919) ;
\coordinate (b0) at (90+120:0.57735026919) ;
\coordinate (c0) at (90+240:0.57735026919) ;
\coordinate (a1) at (90:1.57735026919) ;
\coordinate (b1) at (90+120:1.57735026919) ;
\coordinate (c1) at (90+240:1.57735026919) ;
\draw [fill=black] (a0) circle (3pt) ;
\draw [fill=black] (b0) circle (3pt) ;
\draw [fill=black] (c0) circle (3pt) ;
\draw [fill=black] (a1) circle (3pt) ;
\draw [fill=black] (b1) circle (3pt) ;
\draw [fill=black] (c1) circle (3pt) ;
\draw [thick] (c1) -- (c0) -- (b0) -- (a0) -- (c0);
\draw [thick] (a1) -- (a0);
\draw [thick] (b1) -- (b0);
\end{tikzpicture}
\end{center}
\end{minipage}
&
\begin{minipage}{0.215\textwidth}
\begin{center}
\begin{tikzpicture}
\coordinate (a0) at (0.5,0) ;
\coordinate (b0) at (-0.5,0) ;
\coordinate (c0) at (0,0.86602540378) ;
\coordinate (c1) at (0,1.86602540378) ;
\coordinate (c2) at (0,2.86602540378) ;
\draw [fill=black] (a0) circle (3pt) ;
\draw [fill=black] (b0) circle (3pt) ;
\draw [fill=black] (c0) circle (3pt) ;
\draw [fill=black] (c1) circle (3pt) ;
\draw [fill=black] (c2) circle (3pt) ;
\draw [thick] (c2) -- (c1) -- (c0) -- (b0) -- (a0) -- (c0);
\end{tikzpicture}
\end{center}
\end{minipage}
\\
\\
$T_{0,0,1}=\mbox{paw}$ & $T_{0,1,1}=\mbox{bull}$ & $T_{1,1,1}=\mbox{net}$ & $T_{0,0,2}=\mbox{hammer}$
\end{tabular}
\end{center}
\caption{\label{fig:Thij-examples}Examples of graphs~$T_{h,i,j}$.}
\end{figure}
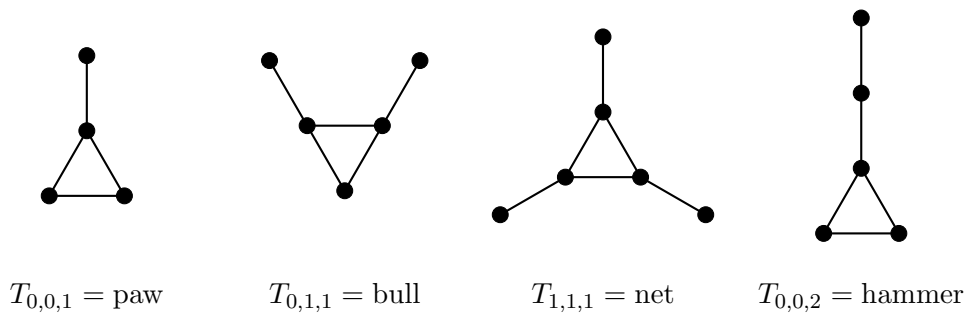

Let $G=(V,E)$ be a graph.
For an induced subgraph $F\ssi G$, the \defword{subgraph complementation} operation, which acts on~$G$ with respect to~$F$, replaces every edge in~$F$ by a non-edge, and vice versa.
If we apply this operation on~$G$ with respect to~$G$ itself, then we obtain the complement~$\overline{G}$ of~$G$.
For two disjoint vertex subsets~$S$ and~$T$ in~$G$, the \defword{bipartite complementation} operation, which acts on~$G$ with respect to~$S$ and~$T$, replaces every edge with one end-vertex in~$S$ and the other one in~$T$ by a non-edge and vice versa.
We note that applying a bipartite complementation is equivalent to applying a sequence of three consecutive subgraph complementations, namely on $G[S\cup T]$, $G[S]$ and~$G[T]$.

Let~${\cal G}$ be a graph class.
Denote the number of labelled graphs on~$n$ vertices in~${\cal G}$ by~$g_n$.
Then~${\cal G}$ is \defword{superfactorial} if there does not exist a constant~$c$ such that $g_n\leq n^{cn}$ for every~$n$.

Recall that a graph class is hereditary if it is closed under taking induced subgraphs.
It is not difficult to see that a graph class~${\cal G}$ is hereditary if and only if~${\cal G}$ can be characterized by a unique set~${\cal F}_{\cal G}$ of minimal forbidden induced subgraphs.
A hereditary graph class~${\cal G}$ is \defword{finitely defined} if~${\cal F}_{\cal G}$ is finite.
We note, however, that the set~${\cal F}_{\cal G}$ may have infinite size.
For example, if~${\cal G}$ is the class of bipartite graphs, then ${\cal F}_{\cal G}=\{C_3,C_5,C_7,\ldots\}$.
If~${\cal F}$ is a set of graphs, we say that a graph~$G$ is ${\cal F}$-free if~$G$ does not contain any graph in~${\cal F}$ as an induced subgraph.
In particular, this means that if a graph class~${\cal G}$ is hereditary, then~${\cal G}$ is exactly the class of ${\cal F}_{\cal G}$-free graphs.
If ${\cal F}=\{H_1,H_2,\ldots\}$ or $\{H_1,H_2,\ldots,H_p\}$ for some $p \geq 0$, we may also describe a graph~$G$ as being $(H_1,H_2,\ldots)$-free or $(H_1,H_2,\ldots,H_p)$-free, respectively, rather than ${\cal F}$-free;
recall that if~${\cal F}=\{H_1\}$ we may write $H_1$-free instead.

\begin{observation}\label{l-obs}
Let~${\cal H}$ and~${\cal H^*}$ be sets of graphs.
The class of ${\cal H}$-free graphs is contained in the class of ${\cal H^*}$-free graphs if and only if for every graph $H^* \in {\cal H^*}$, the set~${\cal H}$ contains an induced subgraph of~$H^*$.
\end{observation}

Suppose~${\cal H}$ and~${\cal H^*}$ are sets of graphs such that for every graph $H^* \in {\cal H^*}$, the set~${\cal H}$ contains an induced subgraph of~$H^*$.
Observation~\ref{l-obs} implies that any graph problem that is polynomial-time solvable for ${\cal H^*}$-free graphs is also polynomial-time solvable for ${\cal H}$-free graphs, and any graph problem that is \NP-complete for ${\cal H}$-free graphs is also \NP-complete for ${\cal H^*}$-free graphs.

We define the \defword{complement} of a hereditary graph class~${\cal G}$ as $\overline{{\cal G}}=\{\overline{G}\; |\; G\in {\cal G}\}$.
Then~${\cal G}$ is \defword{closed under complementation} if ${\cal G}=\overline{{\cal G}}$.
As~${\cal F}_{\cal G}$ is the unique minimal set of forbidden induced subgraphs for~${\cal G}$, we can make the following observation.

\begin{observation}\label{o-comp}
A hereditary graph class~${\cal G}$ is closed under complementation if and only if~${\cal F}_{\cal G}$ is closed under complementation.
\end{observation}

Let~$G$ be a graph.
The \defword{contraction} of an edge~$uv$ replaces~$u$ and~$v$ and their incident edges by a new vertex~$w$ and edges~$wy$ if and only if either~$uy$ or~$vy$ was an edge in~$G$ (without creating multiple edges or self-loops).
Let~$u$ be a vertex with exactly two neighbours $v,w$, which in addition are non-adjacent.
The \defword{vertex dissolution} of~$u$ removes~$u$, $uv$ and $uw$, and adds the edge~$vw$.
Note that vertex dissolution is a special type of edge contraction, and it is the reverse operation of an edge subdivision (recall that the latter operation replaces an edge~$uv$ by a new vertex~$w$ with edges~$uw$ and~$vw$).

Let~$G$ and~$H$ be graphs.
The graph~$H$ is a \defword{subgraph} of~$G$ if~$G$ can be modified into~$H$ by a sequence of vertex deletions and edge deletions.
We can define other containment relations using the graph operations defined above.
We say that~$G$ contains~$H$ as a \defword{minor} if~$G$ can be modified into~$H$ by a sequence of edge contractions, edge deletions and vertex deletions, as a \defword{topological minor} if~$G$ can be modified into~$H$ by a sequence of vertex dissolutions, edge deletions and vertex deletions, as an \defword{induced minor} if~$G$ can be modified into~$H$ by a sequence of edge contractions and vertex deletions, and as an
 \defword{induced topological minor} if~$G$ can be modified into~$H$ by a sequence of vertex dissolutions and vertex deletions.
Let $\{H_1,\ldots,H_p\}$ be a set of graphs.
If~$G$ does not contain any of the graphs $H_1,\ldots,H_p$ as a subgraph, then~$G$ is \defword{$(H_1,\ldots,H_p)$-subgraph-free}.
We define the terms \defword{$(H_1,\ldots,H_p)$-minor-free}, \defword{$(H_1,\ldots,H_p)$-topological-minor-free}, \defword{$(H_1,\ldots,H_p)$-induced-minor-free}, and \defword{$(H_1,\ldots,H_p)$-induced-topological-minor-free} analogously.
Note that graph classes defined by some set of forbidden subgraphs, minors, topological minors, induced minors, or induced topological minors are hereditary, as they are all closed under vertex deletion.

\begin{example}
A graph is \defword{planar} if it can be embedded in the plane in such a way that any two edges only intersect with each other at their end-vertices.
It is well known that the class of planar graphs can be characterized by a set of forbidden minors: Wagner's Theorem~\cite{Wa37} states that a graph is planar if and only if it is $(K_{3,3},K_5)$-minor-free.
\end{example}

We will also need the following folklore observation (see, for example,~\cite{GPR15}).

\begin{observation}\label{o-freefree}
For every $F\in {\cal S}$, a graph is $F$-subgraph-free if and only if it is $F$-minor-free.
\end{observation}

A \defword{$k$-colouring} of a graph~$G$ is a mapping $c: V\to \{1,\ldots,k\}$ such that $c(u)\neq c(v)$ whenever~$u$ and~$v$ are adjacent vertices.
The \defword{chromatic number} of~$G$ is the smallest~$k$ such that~$G$ has a $k$-colouring.
The \defword{clique number} of~$G$ is the size of a largest clique of~$G$.

A graph~$G$ is \defword{perfect} if, for every $H\ssi G$, the chromatic number of~$H$ is equal to the clique number of~$H$.
The Strong Perfect Graph Theorem~\cite{CRST06} states that~$G$ is perfect if and only if~$G$ is $(C_5,C_7,C_9,\ldots)$-free and $(\overline{C_7},\overline{C_9},\ldots)$-free.
A graph~$G$ is \defword{chordal} if it is $(C_4,C_5,C_6,\ldots)$-free and \defword{weakly chordal} if it is $(C_5,C_6,C_7,\ldots)$-free and $(\overline{C_6},\overline{C_7},\ldots)$-free.
A graph~$G$ is a \defword{split graph} if it has a \defword{split partition}, that is, a partition of its vertex set into two (possibly empty) sets~$K$ and~$I$, where~$K$ is a clique and~$I$ is an independent set.
It is known that a graph is split if and only if it is $(C_4,C_5,2P_2)$-free~\cite{FH77}.
A graph~$G$ is a \defword{permutation graph} if line segments connecting two parallel lines can be associated to its vertices in such a way that two vertices of~$G$ are adjacent if and only if the two corresponding line segments intersect.
A graph~$G$ is a \defword{permutation split graph} if it is both permutation and split, and~$G$ is a \defword{permutation bipartite graph} if it is both permutation and bipartite.
A graph~$G$ is \defword{chordal bipartite} if it is $(C_3,C_5,C_6,C_7,\ldots)$-free.
A graph~$G$ is \defword{distance-hereditary} if the distance between any two vertices~$u$ and~$v$ in any connected induced subgraph of~$G$ is the same as the distance of~$u$ and~$v$ in~$G$.
Equivalently, a graph is distance-hereditary if and only if it is $(\domino,\gem,\house,C_5,C_6,C_7,\ldots)$-free~\cite{BM86}.
A graph is \defword{(unit) interval} if it has a representation in which each vertex~$u$ corresponds to an interval~$I_u$ (of unit length) of the line such that two vertices~$u$ and~$v$ are adjacent if and only if $I_u\cap I_v\neq \emptyset$.

We make the following observation.
A number of inclusions in Observation~\ref{o-inclusion} follow immediately from the definitions and the Strong Perfect Graph Theorem.
For the remaining inclusions we refer to~\cite{BLS99}.

\begin{observation}\label{o-inclusion}
The following statements hold:\\[-2em]
\begin{enumerate}
\item every split graph is chordal\\[-2em]
\item every (unit) interval graph is chordal\\[-2em]
\item every chordal graph is weakly chordal\\[-2em]
\item every (bipartite or split) permutation graph is weakly chordal\\[-2em]
\item every distance-hereditary graph is weakly chordal\\[-2em]
\item every weakly chordal graph is perfect\\[-2em]
\item every bipartite permutation graph is chordal bipartite, and\\[-2em]
\item every (chordal) bipartite graph is perfect.
\end{enumerate}
\end{observation}
The containments listed in Observation~\ref{o-inclusion} (and those that follow from them by transitivity) are also displayed \figurename~\ref{fig:class-containment}.
It is not difficult to construct counterexamples for the other containments.
Indeed, for pairs of classes above for which we have listed the minimal forbidden induced subgraph characterizations, these characterizations immediately provide such counterexamples.

\begin{figure}
\begin{center}
\begin{tikzpicture}
\node at (0,10) (p) {perfect};
\node at (-3,8.5) (b) {bipartite};
\node at (0,8.5) (wc) {weakly-chordal};
\node at (-3,7) (pe) {permutation};
\node at (-6,7) (cb) {chordal bipartite};
\node at (0,7) (c) {chordal};
\node at (3,7) (dh) {distance-hereditary};
\node at (-6,5.5) (bp) {bipartite permutation};
\node at (-1.5,5.5) (s) {split};
\node at (0,5.5) (i) {interval};
\node at (0,4) (ui) {unit interval};
\node at (-3,4) (spe) {split permutation};
\draw [->] (pe) -- (spe);
\draw [->] (s) -- (spe);
\draw [->] (p) -- (b);
\draw [->] (b) -- (cb);
\draw [->] (cb) -- (bp);
\draw [->] (pe) -- (bp);
\draw [->] (wc) -- (pe);
\draw [->] (p) -- (wc);
\draw [->] (wc) -- (c);
\draw [->] (wc) -- (dh);
\draw [->] (c) -- (s);
\draw [->] (c) -- (i);
\draw [->] (i) -- (ui);
\end{tikzpicture}
\end{center}
\caption[The inclusion relations between well-known classes mentioned in the paper.]{\label{fig:class-containment}The inclusion relations between well-known classes mentioned in the paper.
An arrow from one class to another indicates that the first class contains the second.}
\end{figure}
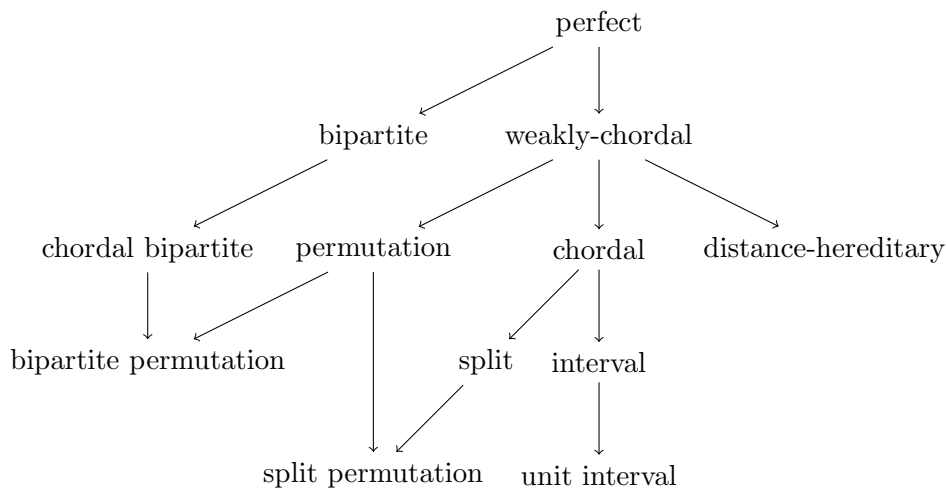

\medskip
We now introduce the notion of treewidth formally.
Recall from Section~\ref{s-intro} that treewidth expresses to what extent a graph is ``tree-like''.
A \defword{tree decomposition} of a graph~$G$ is a pair $(T,{\cal X})$ where~$T$ is a tree and ${\cal X}=\{X_{i} \mid i\in V(T)\}$ is a collection of subsets of~$V(G)$, such that the following three conditions hold: 
\begin{enumerate}[(i)]
\item $\bigcup_{i \in V(T)} X_{i} = V(G)$\\[-1.5em]
\item for every edge $xy \in E(G)$, there is an $i \in V(T)$ such that $x,y\in X_i$ and\\[-1.5em]
\item for every $x\in V(G)$, the set $\{ i\in V(T) \mid x \in X_{i} \}$ induces a connected subtree of~$T$.
\end{enumerate}
The \defword{width} of the tree decomposition $(T,{\cal X})$ is $\max\{|X_{i}| - 1 \;|\; i \in V(T)\}$, and the \defword{treewidth}~$\tw(G)$ of~$G$ is the minimum width over all tree decompositions of~$G$.
If~$T$ is a path, then $(X,T)$ is a \defword{path decomposition} of~$G$.
The \defword{path-width}~$\pw(G)$ of~$G$ is the minimum width over all path decompositions of~$G$.

A \defword{quasi order}~$\leq$ on a set~$X$ is a reflexive, transitive binary relation.
Two elements $x,y \in X$ in~$\leq$ are \defword{comparable} if $x \leq y$ or $y \leq x$; otherwise they are \defword{incomparable}.
A set of pairwise (in)comparable elements in~$\leq$ is called an \defword{(anti)-chain}.
A quasi-order~$\leq$ is a \defword{well-quasi-order} if every infinite sequence of elements $x_1,x_2,x_3,\ldots$ in~$X$ contains a pair $(x_i,x_j)$ with $x_i \leq x_j$ and $i<j$, or equivalently, if~$\leq$ has no infinite strictly decreasing sequence and no infinite anti-chain.
A \defword{partial order}~$\leq$ is a quasi-order which is anti-symmetric, that is, if $x \leq y$ and $y \leq x$ then $x=y$.
If we consider two graphs to be ``equal'' when they are isomorphic, then all quasi orders considered in this paper are in fact partial orders.
As such, throughout this paper ``quasi order'' can be interpreted as ``partial order''.

For an arbitrary set~$M$, we let~$M^*$ denote the set of finite sequences of elements of~$M$.
A quasi-order~$\leq$ on~$M$ defines a quasi-order~$\leq^*$ on~$M^*$ as follows: $(a_1,\ldots,a_m) \leq^* (b_1,\ldots,b_n)$ if and only if there is a sequence of integers $i_1,\ldots,i_m$ with $1 \leq i_1<\cdots<i_m \leq n$ such that $a_j \leq b_{i_j}$ for $j \in \{1,\ldots,m\}$.
We call~$\leq^*$ the {\em subsequence relation}.

The following lemma is well known and very useful when dealing with quasi-orders.

\begin{lemma}[Higman's Lemma~\cite{Higman52}]\label{lem:higman}
Let $(M,\leq)$ be a well-quasi-order.
Then \mbox{${(M^*, \leq^*)}$} is a well-quasi-order.
\end{lemma}

\section{Clique-Width}\label{s-operations}

In this section we give a number of basic results on clique-width.
We begin by giving a formal definition.\footnote{The term \emph{clique-width} and the definition in essentially same form we give here were introduced by Courcelle and Olariu~\cite{CO00} based on operations and related decompositions from Courcelle, Engelfriet and Rozenberg~\cite{CER93}; see also~\cite{CE12}.
Although we consider only undirected graphs, the definitions of~\cite{CO00} also covered the case of directed graphs.
Other equivalent width parameters have also been studied for directed graphs.
For example, Kant\'e and Rao~\cite{KR13} considered the rank-width of directed graphs.}
The \defword{clique-width} of a graph~$G$, denoted by~$\cw(G)$, is the minimum number of labels needed to construct~$G$ using the following four operations:
\begin{enumerate}
\item Create a new graph with a single vertex~$v$ with label~$i$.
(This operation is written~$i(v)$.)\\[-2em]
\item\label{operation:disjoint-union}Take the disjoint union of two labelled graphs~$G_1$ and~$G_2$ (written~$G_1\oplus\nobreak G_2$).\\[-2em]
\item Add an edge between every vertex with label~$i$ and every vertex with label~$j$, $i\neq j$ (written~$\eta_{i,j}$).\\[-2em]
\item Relabel every vertex with label~$i$ to have label~$j$ (written~$\rho_{i\rightarrow j}$).
\end{enumerate}

We say that a construction of a graph~$G$ with the four operations is a \defword{$k$-expression} if it uses at most~$k$ labels.
Thus the clique-width of~$G$ is the minimum~$k$ for which~$G$ has a $k$-expression.
We refer to~\cite{CHMPR15,HMPR16,HS15} for a number of characterizations of clique-width and to~\cite{Ka18} for a compact representation of graphs of clique-width~$k$.

\begin{example}\label{e-p4}
We first note that $\cw(P_1)=1$ and $\cw(P_2)=\cw(P_3)=2$.
Now consider a path on four vertices $v_1, v_2, v_3, v_4$, in that order.
Then this path can be constructed using the four operations (using only three labels) as follows:
$$\eta_{3,2}(3(v_4)\oplus \rho_{3\rightarrow 2}(\rho_{2\rightarrow 1}(\eta_{3,2}(3(v_3)\oplus \eta_{2,1}(2(v_2)\oplus 1(v_1)))))).$$
Note that at the end of this construction, only~$v_4$ has label~$3$.
It is easy to see that a construction using only two labels is not possible.
Hence, we deduce that $\cw(P_4)=3$.
This construction can readily be generalized to longer paths: for $n \geq 5$ let~$E$ be a $3$-expression for the path~$P_{n-1}$ on vertices $v_1,\ldots,v_{n-1}$, with only the vertex~$v_{n-1}$ having label~$3$, then $\eta_{3,2}(3(v_n)\oplus \rho_{3\rightarrow 2}(\rho_{2\rightarrow 1}(E)))$ is a $3$-expression for the path~$P_n$ on vertices $v_1,\ldots,v_n$, with only the vertex~$v_n$ having label~$3$.
Therefore $\cw(P_n)=3$ for all $n\geq 4$. 
Moreover, by changing the construction to give the first vertex~$v_1$ on a path~$P_n$ $(n\geq 3)$ a unique fourth label, we can connect it to the last constructed vertex~$v_n$ of~$P_n$ (the only vertex with label~$3$) via an edge-adding operation to obtain~$C_n$.
Hence, we find that $\cw(C_n)\leq 4$ for every $n\geq 3$.
In fact $\cw(C_n)=4$ holds for every $n\geq 7$~\cite{MR99}.
\end{example}

A class of graphs~${\cal G}$ has \defword{bounded} clique-width if there is a constant~$c$ such that the clique-width of every graph in~${\cal G}$ is at most~$c$.
If such a constant~$c$ does not exist, we say that the clique-width of~${\cal G}$ is \defword{unbounded}.
A hereditary graph class~${\cal G}$ is a \defword{minimal} class of unbounded clique-width if it has unbounded clique-width and every proper hereditary subclass of~${\cal G}$ has bounded clique-width.

The following two observations, which are both well known and readily seen, give two graph classes of small clique-width. 
In particular, Proposition~\ref{p-atmost2} follows from Example~\ref{e-p4} after observing that a graph of maximum degree at most~$2$ is the disjoint union of paths and cycles.
For more examples of graph classes of small width, see, for instance,~\cite{BDLM05,BELL06}.

\begin{prop}\label{p-forest}
Every forest has clique-width at most~$3$.
\end{prop}

\begin{proof}
Let~$T$ be a tree with a root vertex~$v$.
We claim that there is a $3$-expression which creates~$T$ such that, in the resulting labelled tree, only~$v$ has label~$3$.
We prove this by induction on~$|V(T)|$.
Clearly this holds when $|V(T)|=1$.
Otherwise, let $v_1,\ldots,v_k$ be the children of~$v$ and let $T_1,\ldots,T_k$ be the subtrees of~$T$ rooted at $v_1,\ldots,v_k$, respectively.
By the induction hypothesis, for each~$i$ there is a $3$-expression which creates~$T_i$ such that, in the resulting labelled tree, only~$v_i$ has label~$3$.
We take the disjoint union~$\oplus$ of these expressions and let~$E$ be the resulting $3$-expression.
Then $\eta_{3,2}(3(v)\oplus \rho_{3\rightarrow 2}(\rho_{2\rightarrow 1}(E)))$ is a $3$-expression which creates~$T$ such that, in the resulting labelled tree, only~$v$ has label~$3$.
Therefore for every tree~$T$, there is a $3$-expression that constructs~$T$.
Since a forest is a disjoint union of trees, we can then use the~$\oplus$ operation to extend this to a $3$-expression for any forest.
The proposition follows.
\end{proof}

\begin{prop}\label{p-atmost2}
Every graph of maximum degree at most~$2$ has clique-width at most~$4$.
\end{prop}

Recall that for general graphs, the complexity of computing the clique-width of a graph was open for a number of years, until Fellows, Rosamund, Rotics and Szeider~\cite{FRRS09} proved that this is \NP-hard.
However, Proposition~\ref{p-forest} implies that we can determine the clique-width of a forest~$F$ in polynomial time: 
if~$F$ contains an induced~$P_4$, then $\cw(F)=3$; if~$F$ is $P_4$-free but has an edge, then $\cw(F)=2$; and if $F=sP_1$ for some $s\geq 1$, then $\cw(F)=1$.

In contrast to Proposition~\ref{p-atmost2}, graphs of maximum degree at most~$3$ may have arbitrarily large clique-width.
An example of this is a \defword{wall} of arbitrary height, which can be thought of as a hexagonal grid.
We do not formally define the wall, but instead we refer to \figurename~\ref{fig:walls}, in which three examples of walls of different heights are depicted; see, for example,~\cite{Ch15} for a formal definition.
Note that walls of height at least~$2$ have maximum degree~$3$.
The following result is well known; see for example~\cite{KLM09}.

\begin{theorem}\label{t-unbounded}
The class of walls has unbounded clique-width.
\end{theorem}

\begin{figure}
\begin{center}
\begin{minipage}{0.2\textwidth}
\centering
\begin{tikzpicture}[scale=0.45, every node/.style={scale=0.4}]
\GraphInit[vstyle=Simple]
\SetVertexSimple[MinSize=6pt]
\Vertex[x=1,y=0]{v10}
\Vertex[x=2,y=0]{v20}
\Vertex[x=3,y=0]{v30}
\Vertex[x=4,y=0]{v40}
\Vertex[x=5,y=0]{v50}

\Vertex[x=0,y=1]{v01}
\Vertex[x=1,y=1]{v11}
\Vertex[x=2,y=1]{v21}
\Vertex[x=3,y=1]{v31}
\Vertex[x=4,y=1]{v41}
\Vertex[x=5,y=1]{v51}

\Vertex[x=0,y=2]{v02}
\Vertex[x=1,y=2]{v12}
\Vertex[x=2,y=2]{v22}
\Vertex[x=3,y=2]{v32}
\Vertex[x=4,y=2]{v42}

\Edges(    v10,v20,v30,v40,v50)
\Edges(v01,v11,v21,v31,v41,v51)
\Edges(v02,v12,v22,v32,v42)

\Edge(v01)(v02)

\Edge(v10)(v11)

\Edge(v21)(v22)

\Edge(v30)(v31)

\Edge(v41)(v42)

\Edge(v50)(v51)

\end{tikzpicture}
\end{minipage}
\begin{minipage}{0.3\textwidth}
\centering
\begin{tikzpicture}[scale=0.45, every node/.style={scale=0.4}]
\GraphInit[vstyle=Simple]
\SetVertexSimple[MinSize=6pt]
\Vertex[x=1,y=0]{v10}
\Vertex[x=2,y=0]{v20}
\Vertex[x=3,y=0]{v30}
\Vertex[x=4,y=0]{v40}
\Vertex[x=5,y=0]{v50}
\Vertex[x=6,y=0]{v60}
\Vertex[x=7,y=0]{v70}

\Vertex[x=0,y=1]{v01}
\Vertex[x=1,y=1]{v11}
\Vertex[x=2,y=1]{v21}
\Vertex[x=3,y=1]{v31}
\Vertex[x=4,y=1]{v41}
\Vertex[x=5,y=1]{v51}
\Vertex[x=6,y=1]{v61}
\Vertex[x=7,y=1]{v71}

\Vertex[x=0,y=2]{v02}
\Vertex[x=1,y=2]{v12}
\Vertex[x=2,y=2]{v22}
\Vertex[x=3,y=2]{v32}
\Vertex[x=4,y=2]{v42}
\Vertex[x=5,y=2]{v52}
\Vertex[x=6,y=2]{v62}
\Vertex[x=7,y=2]{v72}

\Vertex[x=1,y=3]{v13}
\Vertex[x=2,y=3]{v23}
\Vertex[x=3,y=3]{v33}
\Vertex[x=4,y=3]{v43}
\Vertex[x=5,y=3]{v53}
\Vertex[x=6,y=3]{v63}
\Vertex[x=7,y=3]{v73}

\Edges(    v10,v20,v30,v40,v50,v60,v70)
\Edges(v01,v11,v21,v31,v41,v51,v61,v71)
\Edges(v02,v12,v22,v32,v42,v52,v62,v72)
\Edges(    v13,v23,v33,v43,v53,v63,v73)

\Edge(v01)(v02)

\Edge(v10)(v11)
\Edge(v12)(v13)

\Edge(v21)(v22)

\Edge(v30)(v31)
\Edge(v32)(v33)

\Edge(v41)(v42)

\Edge(v50)(v51)
\Edge(v52)(v53)

\Edge(v61)(v62)

\Edge(v70)(v71)
\Edge(v72)(v73)
\end{tikzpicture}
\end{minipage}
\begin{minipage}{0.35\textwidth}
\centering
\begin{tikzpicture}[scale=0.45, every node/.style={scale=0.4}]
\GraphInit[vstyle=Simple]
\SetVertexSimple[MinSize=6pt]
\Vertex[x=1,y=0]{v10}
\Vertex[x=2,y=0]{v20}
\Vertex[x=3,y=0]{v30}
\Vertex[x=4,y=0]{v40}
\Vertex[x=5,y=0]{v50}
\Vertex[x=6,y=0]{v60}
\Vertex[x=7,y=0]{v70}
\Vertex[x=8,y=0]{v80}
\Vertex[x=9,y=0]{v90}

\Vertex[x=0,y=1]{v01}
\Vertex[x=1,y=1]{v11}
\Vertex[x=2,y=1]{v21}
\Vertex[x=3,y=1]{v31}
\Vertex[x=4,y=1]{v41}
\Vertex[x=5,y=1]{v51}
\Vertex[x=6,y=1]{v61}
\Vertex[x=7,y=1]{v71}
\Vertex[x=8,y=1]{v81}
\Vertex[x=9,y=1]{v91}

\Vertex[x=0,y=2]{v02}
\Vertex[x=1,y=2]{v12}
\Vertex[x=2,y=2]{v22}
\Vertex[x=3,y=2]{v32}
\Vertex[x=4,y=2]{v42}
\Vertex[x=5,y=2]{v52}
\Vertex[x=6,y=2]{v62}
\Vertex[x=7,y=2]{v72}
\Vertex[x=8,y=2]{v82}
\Vertex[x=9,y=2]{v92}

\Vertex[x=0,y=3]{v03}
\Vertex[x=1,y=3]{v13}
\Vertex[x=2,y=3]{v23}
\Vertex[x=3,y=3]{v33}
\Vertex[x=4,y=3]{v43}
\Vertex[x=5,y=3]{v53}
\Vertex[x=6,y=3]{v63}
\Vertex[x=7,y=3]{v73}
\Vertex[x=8,y=3]{v83}
\Vertex[x=9,y=3]{v93}

\Vertex[x=0,y=4]{v04}
\Vertex[x=1,y=4]{v14}
\Vertex[x=2,y=4]{v24}
\Vertex[x=3,y=4]{v34}
\Vertex[x=4,y=4]{v44}
\Vertex[x=5,y=4]{v54}
\Vertex[x=6,y=4]{v64}
\Vertex[x=7,y=4]{v74}
\Vertex[x=8,y=4]{v84}

\Edges(    v10,v20,v30,v40,v50,v60,v70,v80,v90)
\Edges(v01,v11,v21,v31,v41,v51,v61,v71,v81,v91)
\Edges(v02,v12,v22,v32,v42,v52,v62,v72,v82,v92)
\Edges(v03,v13,v23,v33,v43,v53,v63,v73,v83,v93)
\Edges(v04,v14,v24,v34,v44,v54,v64,v74,v84)

\Edge(v01)(v02)
\Edge(v03)(v04)

\Edge(v10)(v11)
\Edge(v12)(v13)

\Edge(v21)(v22)
\Edge(v23)(v24)

\Edge(v30)(v31)
\Edge(v32)(v33)

\Edge(v41)(v42)
\Edge(v43)(v44)

\Edge(v50)(v51)
\Edge(v52)(v53)

\Edge(v61)(v62)
\Edge(v63)(v64)

\Edge(v70)(v71)
\Edge(v72)(v73)

\Edge(v81)(v82)
\Edge(v83)(v84)

\Edge(v90)(v91)
\Edge(v92)(v93)
\end{tikzpicture}
\end{minipage}
\caption{Walls of height $2$, $3$ and~$4$, respectively.}\label{fig:walls}
\end{center}
\end{figure}
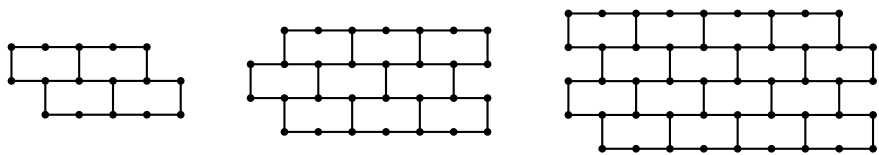

As mentioned, clique-width is more general than treewidth.
Courcelle and Olariu~\cite{CO00} proved that $\cw(G)\leq 4\cdot 2^{\tw(G)-1}+1$ for every graph~$G$ (see \cite{GM03} for an alternative proof).
Corneil and Rotics~\cite{CR05} improved this bound by showing that $\cw(G)\leq 3\cdot 2^{\tw(G)-1}$ for every graph~$G$.
They also proved that for every~$k$, there is a graph $G$ with $\tw(G)= k$ and $\cw(G)\geq 2^{\left\lfloor\frac{\tw(G)}{2}\right\rfloor-1}$.
The following result shows that for restricted graph classes the two parameters may be equivalent (see~\cite{Co18,Co18b} for graph classes for which treewidth and clique-width are even linearly related).

\begin{sloppypar}
\begin{theorem}[\cite{GW00}]\label{t-cwtw}
For $t\geq 1$, every class of $K_{t,t}$-subgraph-free graphs of bounded clique-width has bounded treewidth.
\end{theorem}
\end{sloppypar}

\begin{cor}\label{c-cwtw}
A class of graphs of bounded maximum degree has bounded clique-width if and only if it has bounded treewidth.
\end{cor}

Gurski and Wanke gave another connection between treewidth and clique-width.

\begin{theorem}[\cite{GW07}]\label{t-gw1}
A class of graphs~${\cal G}$ has bounded treewidth if and only if the class of line graphs of graphs in~${\cal G}$ has bounded clique-width.
\end{theorem}

As mentioned in Section~\ref{s-intro}, boundedness of clique-width has been determined for many hereditary graph classes.
However, using the definition of clique-width directly to prove that a certain hereditary graph class~${\cal G}$ has bounded clique-width is often difficult.
An alternative way to show that a hereditary graph class~${\cal G}$ has bounded clique-width is to prove that for infinitely many values of~$n$, the number of labelled graphs in~${\cal G}$ on~$n$ vertices is at most the Bell number~$B_n$~\cite{ALR09}, but this has limited applicability.
The following {\tt BCW Method} is more commonly used:

\medskip
\noindent \rule{\textwidth}{0.1mm}\\
{\tt Bounding Clique-Width (BCW Method)}\\[-1.5em]
\begin{enumerate}
\item If possible, consider only graphs in~${\cal G}$ that have some suitable property~$\pi$.\\[-1.5em]
\item Take a graph class~${\cal G}'$ for which it is known that its clique-width is bounded.\\[-1.5em]
\item For every graph~$G\in {\cal G}$ (possibly with property~$\pi$), reduce~$G$ to a graph in~${\cal G}'$ by using a constant number of graph operations that do not change the clique-width of~$G$ by ``too much''.\\[-2em]
\end{enumerate}
\noindent \rule{\textwidth}{0.1mm}

\medskip
\noindent
Note that the subclass of graphs in~${\cal G}$ that have some property~$\pi$ in Step~1 need not be hereditary.
For example, it is known~\cite{BL02,LR04} that we may choose the property~$\pi$ to be that of being $2$-connected and that we can delete some constant number~$k$ of vertices from a graph without affecting the clique-width by more than some bounded amount.
Then we could try to prove that~${\cal G}$ has bounded clique-width by showing that for every $2$-connected graph in~${\cal G}$, we can delete no more than~$k$ vertices to obtain a graph in some class~${\cal G'}$ that we know to have bounded clique-width.
We give some concrete examples of this method in the next section.

The power of the method depends on both the graph property~$\pi$ in Step~1 and the graph operations that we are allowed to use in Step~3.
In particular we will use graph operations to modify a graph~$G$ of some class~${\cal G}$ into the disjoint union of some graphs that have a simpler structure than~$G$ itself.
As a result, we can then deal with these simpler graphs separately.
This approach is particularly useful if ${\cal G}$ is hereditary: if the simpler graphs are induced subgraphs of the original graph~$G$, then we can still make use of earlier deduced properties for~${\cal G}$ when dealing with the simpler induced subgraphs of $G$.
Before giving important examples of these operations and properties, we first formalize our approach.

Let $k\geq 0$ be a constant and let~$\gamma$ be some graph operation.
We say that a graph class~${\cal G'}$ is {\em $(k,\gamma)$-obtained} from a graph class~${\cal G}$ if the following two conditions hold:
\begin{enumerate}
\item\label{i-cond1}every graph in~${\cal G'}$ can be obtained from a graph in~${\cal G}$ by performing~$\gamma$ at most~$k$ times, and\\[-2em]
\item\label{i-cond2}for every graph $G\in {\cal G}$ there exists at least one graph in~${\cal G'}$ that can be obtained from~$G$ by performing~$\gamma$ at most~$k$ times (note that~${\cal G}$ is not necessarily a subclass of~${\cal G}'$).
\end{enumerate}

A graph operation~$\gamma$ \defword{preserves} boundedness of clique-width if, for every finite constant~$k$ and every graph class~${\cal G}$, every graph class~${\cal G}'$ that is $(k,\gamma)$-obtained from~${\cal G}$ has bounded clique-width if and only if~${\cal G}$ has bounded clique-width.
We note that Condition~\ref{i-cond1} is necessary for this definition to be meaningful; without this condition the class of all graphs (which has unbounded clique-width) would be $(k,\gamma)$-obtained from every other graph class.
Similarly, we also need Condition~\ref{i-cond2}, as otherwise every graph class would be $(k,\gamma)$-obtained from the class of all graphs.
If $k=\infty$ is allowed, then~$\gamma$ preserves boundedness of clique-width \defword{ad infinitum}.
Similarly, a graph property~$\pi$ \defword{preserves} boundedness of clique-width if, for every graph class~${\cal G}$, the subclass of~${\cal G}$ with property~$\pi$ has bounded clique-width if and only if~${\cal G}$ has bounded clique-width.
If necessary, we may restrict these definitions to only be valid for some specific types of graph classes.

We refer to the survey of Gurski~\cite{Gu17} for a detailed overview of graph operations that preserve boundedness of clique-width and for bounds that tell us more precisely by how much the clique-width can change when applying various operations.\footnote{We note that some of these graph operations may exponentially increase the upper bound of the clique-width.}
Here, we only state the most important graph operations, together with two well-known properties that preserve boundedness of clique-width.

\paragraph{Facts about clique-width:}
\begin{enumerate}[\bf F{a}ct 1.]

\item \label{fact:del-vert}Vertex deletion preserves boundedness of clique-width~\cite{LR04}.\\[-1em]

\item \label{fact:comp}Subgraph complementation preserves boundedness of clique-width~\cite{KLM09}.\\[-1em]

\item \label{fact:bip}Bipartite complementation preserves boundedness of clique-width~\cite{KLM09}.\\[-1em]

\item \label{fact:prime}Being prime preserves boundedness of clique-width for hereditary graph classes~\cite{CO00}.\\[-1em]

\item \label{fact:2-conn}Being $2$-connected preserves boundedness of clique-width for hereditary graph classes~\cite{BL02,LR04}.\\[-1em]

\item \label{fact:subdiv}Edge subdivision preserves boundedness of clique-width ad infinitum for graph classes of bounded maximum degree~\cite{KLM09}.

\end{enumerate}

We note that Fact~\ref{fact:bip} follows from Fact~\ref{fact:comp}, as bipartite complementations can be mimicked by three subgraph complementations.
Moreover, an edge deletion is a special case of subgraph complementation, whereas an edge contraction is a vertex deletion and a bipartite complementation.
Finally, recall that an edge subdivision is the reverse operation of a vertex dissolution, which can be seen as a type of edge contraction.
Hence, from Facts~\ref{fact:del-vert}--\ref{fact:bip} it follows that edge deletion, edge contraction and edge subdivision each preserve boundedness of clique-width.

Vertex deletions, edge deletions and edge contractions do not preserve boundedness of clique-width ad infinitum: one can take any graph class of unbounded clique-width and apply one of these operations until one obtains the empty graph or an edgeless graph. 
Hence, Facts~\ref{fact:del-vert}--\ref{fact:bip} do not preserve boundedness of clique-width ad infinitum.
This holds even for graphs of maximum degree at most~$3$, as the class of walls and their induced subgraphs has unbounded clique-width by Theorem~\ref{t-unbounded}.

In contrast, Fact~\ref{fact:subdiv} says that edge subdivisions applied on graphs of bounded maximum degree do preserve boundedness of clique-width ad infinitum.
We note that Fact~\ref{fact:subdiv} follows from Corollary~\ref{c-cwtw} and the fact that an edge subdivision does not change the treewidth of a graph (see, for example,~\cite{LR06}).
However, the condition on the maximum degree is necessary for the ``only if'' direction of Fact~\ref{fact:subdiv}.
Otherwise, as discussed in~\cite{DP16}, one could start with a clique~$K$ on at least two vertices (which has clique-width~$2$) and then apply an edge subdivision on an edge~$uv$ in~$K$ if and only if~$uv$ is not an edge in some graph~$G$ of arbitrarily large clique-width with $|V(G)|=|V(K)|$.
This yields a graph~$G'$ that contains~$G$ as an induced subgraph, implying that $\cw(G')\geq \cw(G)$, which is arbitrarily larger than $\cw(K)=2$.

As an aside, note that edge contractions do not increase the clique-width of graphs of bounded maximum degree either.
We can apply Corollary~\ref{c-cwtw} again after observing from the definition of treewidth that edge contractions do not increase treewidth.
However, the condition on the maximum degree is necessary here as well; a (non-trivial) counterexample is given by Courcelle~\cite{Co14}, who proved that the class of graphs that are obtained by edge contractions from the class of graphs of clique-width~$3$ has unbounded clique-width.

For the {\tt BCW Method}, operations that preserve boundedness of clique-width may be combined, but these operations may not always be used in combination with some property~$\pi$ that preserves boundedness of clique-width.
This is because applying a graph operation may result in a graph that does not have property~$\pi$.
Moreover, it is not always clear whether two or more properties that preserve boundedness of clique-width may be unified into one property.
For instance, every non-empty class of $2$-connected graphs is not hereditary and every class of prime graphs containing a graph on more than two vertices is not hereditary.
As such, it is unknown whether Facts~\ref{fact:prime} and~\ref{fact:2-conn}, which may only be applied on hereditary graph classes, can be combined.
That is, the following problem is open.

\begin{oproblem}
Let~${\cal G}$ be a hereditary class of graphs and let~${\cal F}$ be the class of $2$-connected prime graphs in~${\cal G}$.
If~${\cal F}$ has bounded clique-width, does this imply that~${\cal G}$ has bounded clique-width?
\end{oproblem}

To prove that a graph class~${\cal G}$ has unbounded clique-width, a similar method to the {\tt BCW Method} can be used.

\medskip
\noindent \rule{\textwidth}{0.1mm}\\
{\tt Unbounding Clique-Width (UCW Method)}\\[-1.5em]
\begin{enumerate}
\item Take a graph class~${\cal G}'$ known to have unbounded clique-width.\\[-1.5em]
\item For every graph~$G'\in {\cal G}'$, reduce~$G'$ to a graph in~${\cal G}$ by using a constant number of graph operations that do not change the clique-width of~$G'$ by ``too much''.\\[-2em]
\end{enumerate}
\noindent \rule{\textwidth}{0.1mm}

\medskip
By Theorem~\ref{t-unbounded}, we can consider the class of walls as a starting point for the graph class~${\cal G}'$.
A \defword{$k$-subdivided wall} is a graph obtained from a wall after subdividing each edge exactly~$k$ times for some constant $k\geq 0$.
Combining Fact~\ref{fact:subdiv} with Theorem~\ref{t-unbounded} and the observation that walls of height at least~$2$ have maximum degree~$3$ leads to the following result.

\begin{cor}[\cite{LR06}]\label{cor:walls}
For any constant $k\geq 0$, the class of $k$-subdivided walls has unbounded clique-width.
\end{cor}

Corollary~\ref{cor:walls} has proven to be very useful.
For instance, it can be used to obtain the following result (recall that~${\cal S}$ is the class of graphs each connected component of which is either a subdivided claw or a path).

\begin{cor}[\cite{LR06}]\label{cor:s}
Let $\{H_1,\ldots,H_p\}$ be a finite set of graphs.
If $H_i\notin {\cal S}$ for all $i \in \{1,\ldots,p\}$, then the class of $(H_1,\ldots,H_p)$-free graphs has unbounded~clique-width.
\end{cor}

As a side note, we remark that ``limit classes'' of hereditary graph classes of unbounded clique-width may have bounded clique-width.
For instance, the class of $(C_k,\ldots,C_\ell)$-subgraph-free graphs has unbounded clique-width for any two integers $k\geq 3$ and $\ell\geq k$ due to Corollary~\ref{cor:walls}.
However, for every $k\geq 3$, the class of $(C_k,C_{k+1},\ldots)$-subgraph-free graphs has bounded clique-width~\cite{LM13}.
We refer to~\cite{LM13} for more details on limit classes.

Corollary~\ref{cor:walls} is further generalized by the following theorem.

\begin{theorem}[\cite{DP16}]\label{thm:generalunbounded}
For $m\geq 0$ and $n >\nobreak m+\nobreak 1$ the clique-width of a graph~$G$ is at least $\lfloor\frac{n-1}{m+1}\rfloor+\nobreak 1$ if~$V(G)$ has a partition into sets $V_{i,j} \; (i,j \in \{0,\ldots,n\})$ with the following properties:
\begin{enumerate}
\item \label{prop:v_i0-small}$|V_{i,0}| \leq 1$ for all $i\geq 1$\\[-2em]
\item \label{prop:v_0j-small}$|V_{0,j}| \leq 1$ for~all~$j\geq 1$\\[-2em]
\item \label{prop:v_ij-nonempty}$|V_{i,j}|\geq 1$ for all $i,j\geq 1$\\[-2em]
\item \label{prop:row-connected}$G[\cup^n_{j=0}V_{i,j}]$ is connected for all $i\geq 1$\\[-2em]
\item \label{prop:column-connected}$G[\cup^n_{i=0}V_{i,j}]$ is connected for all $j\geq 1$\\[-2em]
\item \label{prop:v_k0-nbrs}for $i,j,k\geq 1$, if a vertex of~$V_{k,0}$ is adjacent to a vertex of~$V_{i,j}$ then $i \leq k$\\[-2em]
\item \label{prop:v_0k-nbrs}for $i,j,k\geq 1$, if a vertex of~$V_{0,k}$ is adjacent to a vertex of~$V_{i,j}$ then $j \leq k$, and\\[-2em]
\item \label{prop:v_ij-nbrs}for $i,j,k,\ell\geq 1$, if a vertex of~$V_{i,j}$ is adjacent to a vertex of~$V_{k,\ell}$ then $|k-i|\leq m$ and $|\ell-j| \leq m$.
\end{enumerate}
\end{theorem}

Many other constructions of graphs of large clique-width follow from Theorem~\ref{thm:generalunbounded} using the {\tt UCW Method} (possibly by applying Facts~\ref{fact:del-vert}--\ref{fact:bip}).
For instance, this is the case for square grids~\cite{MR99}, whose exact clique-width was determined by Golumbic and Rotics~\cite{GR00}.
This is also the case for the constructions of Brandst{\"a}dt, Engelfriet, Le and Lozin~\cite{BELL06}, Lozin and Volz~\cite{LV08}, Korpelainen, Lozin and Mayhill~\cite{KLM14} and Kwon, Pilipczuk and Siebertz~\cite{KPS17} for proving that the classes of $K_4$-free co-chordal graphs, $2P_3$-free bipartite graphs, split permutation graphs and twisted chain graphs, respectively, have unbounded clique-width.

Constructions of graphs of arbitrarily large clique-width not covered by Theorem~\ref{thm:generalunbounded} can be found in~\cite{GR00} and~\cite{BL03}, which prove that unit interval graphs and bipartite permutation graphs, respectively, have unbounded clique-width.
We discuss these results in more detail in the next section, but we note the following.

First, the classes of split permutation graphs (and the analogous bipartite class of bichain graphs)~\cite{ABLS}, unit interval graphs~\cite{Lo11} and bipartite permutation graphs~\cite{Lo11} are even minimal hereditary graph classes of unbounded clique-width.
Collins, Foniok, Korpelainen, Lozin and Zamaraev~\cite{CFKLZ} proved that the number of minimal hereditary graphs of unbounded clique-width is infinite.
Second, for classes, such as split graphs, bipartite graphs, co-bipartite graphs and $(K_{1,3},2K_2)$-free graphs, unboundedness of clique-width also follows from the fact that these classes are superfactorial~\cite{BL02} and an application of the following result.

\begin{theorem}[\cite{BL02}]\label{t-super}
Every superfactorial graph class has unbounded clique-width.
\end{theorem}

\section{Results on Clique-Width for Hereditary Graph Classes}\label{s-hereditary}

In this section we survey known results on (un)boundedness of clique-width for hereditary graph classes in a systematic way.\footnote{The Information System on Graph Classes and their Inclusions~\cite{isgci} also keeps a
record of many graph classes for which boundedness or unboundedness of clique-width is known.} The proofs of these results often use the {\tt BCW Method} or {\tt UCW Method}.
As mentioned earlier, many well-studied graph classes are hereditary.
From the point of view of clique-width, these are also natural classes to consider, as the definition of clique-width implies that if a graph~$G$ contains a graph~$H$ as an induced subgraph, then $\cw(H) \leq \cw(G)$.

Recall that a graph class~${\cal G}$ is hereditary if and only if it can be characterized by a (possibly infinite) set of forbidden induced subgraphs~${\cal F}_{\cal G}$.
We start by giving a dichotomy for the case when~${\cal F}_{\cal G}$ consists of a single graph~$H$.
This result is folklore: observe that~$P_4$ has clique-width~$3$ and see~\cite{CO00} for a proof that $P_4$-free graphs have clique-width at most~$2$ and~\cite{DP16} for a proof of the other claims of Theorem~\ref{thm:single}.

\begin{theorem}\label{thm:single}
Let~$H$ be a graph.
The class of $H$-free graphs has bounded clique-width if and only if~$H$ is an induced subgraph of~$P_4$.
Furthermore, a graph has clique-width at most~$2$ if and only if it is $P_4$-free. 
\end{theorem}

Note that by Theorem~\ref{thm:single} we can test whether a graph~$G$ has clique-width at most~$2$ in polynomial time by checking whether~$G$ is $P_4$-free.
We recall that deciding whether a graph has clique-width at most~$c$ is known to be polynomial-time solvable for $c=3$~\cite{CHLRR12}, but open for $c\geq 4$. 

As discussed in Section~\ref{s-intro}, an important reason for studying boundedness of clique-width for special graph classes is to obtain more classes of graphs for which a wide range of classical \NP-complete problems become polynomial-time solvable.
Theorem~\ref{thm:single} shows that this cannot be done for (most) classes of $H$-free graphs.
In order to find more graph classes of bounded clique-width, we can follow several approaches that try to extend Theorem~\ref{thm:single}.

To give an example, Vanherpe~\cite{Va04} considered the class of partner-limited graphs, which were introduced by Roussel, Rusu and Thuillier in~\cite{RRT99}.
A vertex~$u$ in a graph~$G$ is a \defword{partner} of an induced subgraph~$H$ isomorphic to~$P_4$ of~$G$ if $V(H)\cup \{u\}$ induces at least two~$P_4$s in~$G$.
A graph~$G$ is said to be \defword{partner-limited} if every induced~$P_4$ has at most two partners.
Vanherpe proved that the clique-width of partner-limited graphs is at most~$4$.
This result generalized a corresponding result of Courcelle, Makowsky and Rotics~\cite{CMR00} for \defword{$P_4$-tidy} graphs, which are graphs in which every induced~$P_4$ has at most one partner.

To give another example, Makowsky and Rotics~\cite{MR99} considered the classes of $(q,t)$-graphs, which were introduced by Babel and Olariu in~\cite{BO98}.
For two integers~$q$ and~$t$, a graph is a \defword{$(q,t)$-graph} if every subset of~$q$ vertices induces a subgraph that has at most~$t$ distinct induced~$P_4$s.
Note that $P_4$-free graphs are the $(4,0)$-graphs, whereas $(5,1)$-graphs are also known as 
\defword{$P_4$-sparse} graphs; note that the latter class of graphs is a subclass of the class of $P_4$-tidy graphs.
Makowsky and Rotics proved the following result.

\begin{theorem}[\cite{MR99}]\label{t-mr99}
Let $q\geq 4$ and $t\geq 0$. Then the class of $(q,t)$-graphs has bounded clique-width if 

\begin{itemize}
\item $q\leq 6$ and $t\leq q-4$, or \\[-2em]
\item $q\geq 7$ and $t\leq q-3$
\end{itemize}
and it has unbounded clique-width if
\begin{itemize}
\item $q\leq 6$ and $t\geq q-3$ \\[-2em]
\item $q=7$ and $t\geq q-2$, or \\[-2em]
\item $q\geq 8$ and $t\geq q-1$.
\end{itemize}
\end{theorem}

Theorem~\ref{t-mr99} covers all cases except where $q\geq 8$ and $t=q-2$.
Makowsky and Rotics~\cite{MR99} therefore posed the following open problem (see also~\cite{KLM09}).

\begin{oproblem}
Is the clique-width of $(q,q-2)$-graphs bounded if $q\geq 8$? 
\end{oproblem}

Below we list five other systematic approaches, which we discuss in detail in the remainder of this section.
First, we can try to replace ``$H$-free graphs'' by ``$H$-free graphs in some hereditary graph class~${\cal X}$'' in Theorem~\ref{thm:single}.
We discuss this line of research in Section~\ref{s-x}.

Second, we may try to determine boundedness of clique-width of hereditary graph classes~${\cal G}$ for which~${\cal F}_{\cal G}$ is small.
However, even the classification for $(H_1,H_2)$-free graphs is not straightforward and is still incomplete.
We discuss the state-of-the-art for $(H_1,H_2)$-free graphs in Section~\ref{s-2}.
There, we also explain how results in Section~\ref{s-x} are helpful for proving results for $(H_1,H_2)$-free graphs.\footnote{We emphasize that the underlying research goal is not to start classifying the case of three forbidden induced subgraphs~$H_1$, $H_2$ and~$H_3$ after the classification for two graphs~$H_1$ and~$H_2$ has been completed.
Instead the aim is to develop new techniques through a systematic study, by looking at hereditary graph classes from different angles in order to increase our understanding of clique-width.}

Third, we may try to determine boundedness of clique-width for hereditary graph classes~${\cal G}$ for which~${\cal F}_{\cal G}$ only contains graphs of small size.
For instance, Brandst\"adt, Dragan, Le and Mosca~\cite{BDLM05} classified boundedness of clique-width for those hereditary graph classes for which~${\cal F}_{\cal G}$ consists of $1$-vertex extensions of~$P_4$.
We discuss their result, together with other results in this direction, in Section~\ref{s-four}.

Fourth, we observe that~$P_4$ is self-complementary.
As such we can try to extend Theorem~\ref{thm:single} to graph classes closed under complementation.
Determining boundedness of clique-width for such graph classes is also natural to consider due to Fact~\ref{fact:comp}.
We present the current state-of-the-art in this direction in Section~\ref{s-complement}.

Fifth, we may consider hereditary graph classes that can be described not only in terms of forbidden induced subgraphs but also using some other forbidden subgraph containment.
For instance, we can consider hereditary graph classes characterized by some set~${\cal F}$ of forbidden minors.
We survey the known results in this direction in Section~\ref{s-containment}.

\subsection{Considering $\mathbfit{H}$-Free Graphs Contained in Some Hereditary Graph Class}\label{s-x}

Theorem~\ref{thm:single} shows that the class of $H$-free graphs has bounded clique-width only if~$H$ is an induced subgraph of~$P_4$.
In this section we survey the effect on boundedness of clique-width of restricting the class of $H$-free graphs to just those graphs that belong to some hereditary graph class~${\cal X}$.
Initially we do not want to make the hereditary graph class~${\cal X}$, in which we look for these $H$-free graphs, too narrow.
However, if we let~${\cal X}$ be too large, the classification might remain the same as the one for general $H$-free graphs in Theorem~\ref{thm:single}.
This is the case if we let~${\cal X}$ be the class of perfect graphs, or even the class of weakly chordal graphs, which form a proper subclass of perfect graphs by Observation~\ref{o-inclusion}.

\begin{theorem}[\cite{BDHP17}]\label{t-weakly-chordal}
Let~$H$ be a graph.
The class of $H$-free weakly chordal graphs has bounded clique-width if and only if~$H$ is an induced subgraph of~$P_4$.
\end{theorem}

If we restrict~${\cal X}$ further, then there are several potential classes of graphs to consider, such as chordal graphs, permutation graphs and distance-hereditary graphs (see also \figurename~\ref{fig:class-containment}).
However, distance-hereditary graphs are known to have clique-width at most~$3$~\cite{GR00} (and hence their clique-width can be computed in polynomial time using the algorithm of~\cite{CHLRR12}).
On the other hand, the classes of chordal graphs and permutation graphs have unbounded clique-width.
This follows from combining Observation~\ref{o-inclusion} with one of the following three theorems.

\begin{theorem}[\cite{GR00}]\label{t-unitinterval}
The class of unit interval graphs has unbounded clique-width.
\end{theorem}

\begin{theorem}[\cite{KLM14}]\label{t-splitpermutation}
The class of split permutation graphs has unbounded clique-width.
\end{theorem}

\begin{sloppypar}
\begin{theorem}[\cite{BL03}]\label{t-bipper}
The class of bipartite permutation graphs has unbounded clique-width.
\end{theorem}
\end{sloppypar}

The case when~${\cal X}$ is the class of chordal graphs has received particular attention, as we now discuss.
Brandst{\"a}dt, Engelfriet, Le and Lozin~\cite{BELL06} proved that the class of $4P_1$-free chordal graphs has unbounded clique-width.
However, there are many graphs~$H$ besides~$P_4$ for which the class of $H$-free chordal graphs has bounded clique-width.
A result of~\cite{CR05} implies that $K_r$-free chordal graphs have bounded clique-width for every integer~$r\geq 1$.
Brandst{\"a}dt, Le and Mosca~\cite{BLM04} showed that $(P_1+\nobreak P_4)$-free chordal graphs have clique-width at most~$8$ and that $\overline{P_1+P_4}$-free chordal graphs are distance-hereditary graphs and thus have clique-width at most~$3$.
Brandst{\"a}dt, Dabrowski, Huang and Paulusma~\cite{BDHP17} proved that bull-free chordal graphs have clique-width at most~$3$, improving a known bound of~$8$~\cite{Le03}.
The same authors also proved that $\overline{S_{1,1,2}}$-free chordal graphs have clique-width at most~$4$, and that the classes of $\overline{K_{1,3}+\nobreak 2P_1}$-free chordal graphs, $(P_1+\nobreak \overline{P_1+P_3})$-free chordal graphs and $(P_1+\nobreak \overline{2P_1+P_2})$-free chordal graphs each have bounded clique-width.

Combining all the above results~\cite{BDHP17,BELL06,BLM04,CR05,GR00,MR99} leads to the following summary for $H$-free chordal graphs; see \figurename~\ref{fig:open-chordal} for definitions of the graphs~$F_1$ and~$F_2$ and \figurename~\ref{fig:chordal-bounded} for pictures of all (maximal) graphs~$H$ for which the class of $H$-free chordal graphs is known to have bounded clique-width.

\begin{theorem}[\cite{BDHP17}]\label{thm:chordal-classification}
Let~$H$ be a graph with $H\notin \{F_1,F_2\}$.
The class of $H$-free chordal graphs has bounded clique-width if and only if:
\begin{enumerate}[(i)]
\item $H=K_r$ for some $r\geq 1$\\[-2em]
\item $H\ssi \bull$\\[-2em]
\item $H\ssi P_1+\nobreak P_4$\\[-2em]
\item $H\ssi \gem$\\[-2em]
\item $H\ssi \overline{K_{1,3}+\nobreak 2P_1}$\\[-2em]
\item $H\ssi P_1+\nobreak\overline{P_1+P_3}$\\[-2em]
\item $H\ssi P_1+\nobreak\overline{2P_1+P_2}$, or\\[-2em]
\item $H\ssi \overline{S_{1,1,2}}$.
\end{enumerate}
\end{theorem}
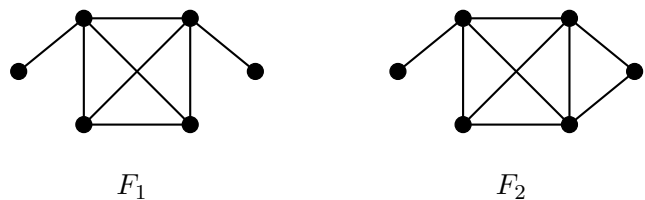
\begin{figure}
\begin{center}
\begin{tabular}{ccc}
\begin{minipage}{0.33\textwidth}
\centering
\scalebox{1}{
{\begin{tikzpicture}[scale=1,rotate=135]
\GraphInit[vstyle=Simple]
\SetVertexSimple[MinSize=6pt]
\Vertices{circle}{a,b,c,d}
\Vertex[a=45,d=1.57313218497]{e}
\Vertex[a=225,d=1.57313218497]{f}
\Edges(a,b,c,d,a,c)
\Edges(b,d)
\Edges(a,e)
\Edges(f,d)
\end{tikzpicture}}}
\end{minipage}
&
\begin{minipage}{0.33\textwidth}
\centering
\scalebox{1}{
{\begin{tikzpicture}[scale=1,rotate=135]
\GraphInit[vstyle=Simple]
\SetVertexSimple[MinSize=6pt]
\Vertices{circle}{a,b,c,d}
\Vertex[a=45,d=1.57313218497]{e}
\Vertex[a=225,d=1.57313218497]{f}
\Edges(a,b,c,d,a,c)
\Edges(b,d)
\Edges(a,e)
\Edges(c,f,d)
\end{tikzpicture}}}
\end{minipage}\\
&\\
$F_1$ & $F_2$ \\
\end{tabular}
\end{center}
\caption{The two graphs~$H$ for which the boundedness of clique-width of the class of $H$-free chordal graphs is open.}
\label{fig:open-chordal}
\end{figure}
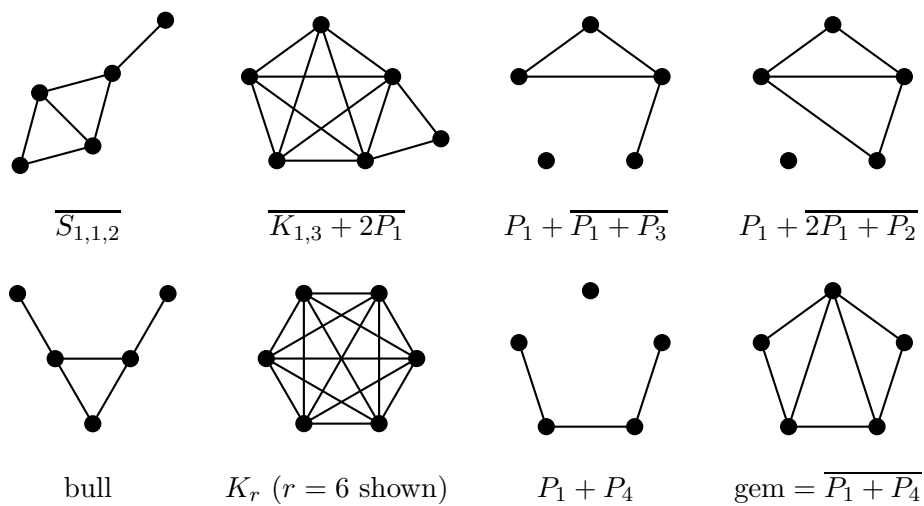
\begin{figure}
\begin{center}
\begin{tabular}{cccc}
\begin{minipage}{0.20\textwidth}
\centering
\scalebox{1}{
{\begin{tikzpicture}[scale=1,rotate=45]
\GraphInit[vstyle=Simple]
\SetVertexSimple[MinSize=6pt]
\Vertex[x=0,y=0]{a}
\Vertex[a=30,d=1]{b}
\Vertex[a=-30,d=1]{e}
\Vertex[x=1.73205080757,y=0]{d}
\Vertex[x=2.73205080757,y=0]{c}
\Edges(e,a,b,e,d,b)
\Edges(c,d)
\end{tikzpicture}}}
\end{minipage}
&
\begin{minipage}{0.20\textwidth}
\centering
\scalebox{1}{
{\begin{tikzpicture}[scale=1,rotate=162]
\GraphInit[vstyle=Simple]
\SetVertexSimple[MinSize=6pt]
\Vertices{circle}{a,b,c,d,e}
\Vertex[a=180,d=1.67504239778]{f}
\Edges(a,b,c,d,e,a)
\Edges(a,c,e,b,d,a)
\Edges(c,f,d)
\end{tikzpicture}}}
\end{minipage}
&
\begin{minipage}{0.20\textwidth}
\centering
\scalebox{1}{
{\begin{tikzpicture}[scale=1,rotate=90]
\GraphInit[vstyle=Simple]
\SetVertexSimple[MinSize=6pt]
\Vertices{circle}{a,b,c,d,e}
\Edges(e,a,b,e,d)
\end{tikzpicture}}}
\end{minipage}
&
\begin{minipage}{0.20\textwidth}
\centering
\scalebox{1}{
{\begin{tikzpicture}[scale=1,rotate=90]
\GraphInit[vstyle=Simple]
\SetVertexSimple[MinSize=6pt]
\Vertices{circle}{a,b,c,d,e}
\Edges(e,a,b,e,d,b)
\end{tikzpicture}}}
\end{minipage}\\
& & &\\
$\overline{S_{1,1,2}}$ & $\overline{K_{1,3}+\nobreak 2P_1}$ & $P_1+\nobreak\overline{P_1+P_3}$ & $P_1+\nobreak\overline{2P_1+P_2}$\\
& & &\\
\begin{minipage}{0.20\textwidth}
\centering
\scalebox{1}{
{\begin{tikzpicture}[scale=1,rotate=90]
\GraphInit[vstyle=Simple]
\SetVertexSimple[MinSize=6pt]
\Vertex[x=0,y=0]{a}
\Vertex[a=30,d=1]{b}
\Vertex[a=30,d=2]{c}
\Vertex[a=-30,d=1]{d}
\Vertex[a=-30,d=2]{e}
\Edges(c,b,a,d,e)
\Edges(b,d)
\end{tikzpicture}}}
\end{minipage}
&
\begin{minipage}{0.20\textwidth}
\centering
\scalebox{1}{
{\begin{tikzpicture}[scale=1,rotate=0]
\GraphInit[vstyle=Simple]
\SetVertexSimple[MinSize=6pt]
\Vertices{circle}{a,b,c,d,e,f}
\Edges(a,b,c,d,e,a)
\Edges(a,c,e,b,d,a)
\Edges(a,f,b)
\Edges(c,f,d)
\Edges(f,e)
\end{tikzpicture}}}
\end{minipage}
&
\begin{minipage}{0.20\textwidth}
\centering
\scalebox{1}{
{\begin{tikzpicture}[scale=1,rotate=90]
\GraphInit[vstyle=Simple]
\SetVertexSimple[MinSize=6pt]
\Vertices{circle}{a,b,c,d,e}
\Edges(b,c,d,e)
\end{tikzpicture}}}
\end{minipage}
&
\begin{minipage}{0.20\textwidth}
\centering
\scalebox{1}{
{\begin{tikzpicture}[scale=1,rotate=90]
\GraphInit[vstyle=Simple]
\SetVertexSimple[MinSize=6pt]
\Vertices{circle}{a,b,c,d,e}
\Edges(a,b,c,d,e,a)
\Edges(c,a,d)
\end{tikzpicture}}}
\end{minipage}\\
& & &\\
bull & $K_r$~($r=\nobreak 6$~shown) & $P_1+\nobreak P_4$ & $\gem=\overline{P_1+P_4}$\\
\end{tabular}
\end{center}
\caption{The graphs~$H$ listed in Theorem~\ref{thm:chordal-classification}, for which the class of $H$-free chordal graphs has bounded clique-width.}\label{fig:chordal-bounded}
\end{figure}

As can be seen from its statement, Theorem~\ref{thm:chordal-classification} leaves only two cases open, namely~$F_1$ and~$F_2$; see also~\cite{BDHP17}.

\begin{oproblem}\label{oprob:chordal}
Determine whether the class of $H$-free chordal graphs has bounded or unbounded clique-width when $H=F_1$ or $H=F_2$.
\end{oproblem}

Recall that split graphs are chordal by Observation~\ref{o-inclusion} and have been shown to have unbounded clique-width~\cite{MR99} (this also follows from Theorem~\ref{t-splitpermutation}).
We now let~${\cal X}$ be the class of split graphs, that is, we consider classes of $H$-free split graphs, and find graphs~$H$ for which the class of $H$-free split graphs has bounded clique-width.
We first note that as the class of split graphs is the class of $(C_4,C_5,2P_2)$-free graphs~\cite{FH77}, the complement of a split graph is also a split graph by Observation~\ref{o-comp}.
By Fact~\ref{fact:comp} this implies the following observation, which we discuss in more depth in Section~\ref{s-complement}.

\begin{observation}\label{o-splito}
For a graph~$H$, the class of $H$-free split graphs has bounded clique-width if and only the class of $\overline{H}$-free split graphs has bounded clique-width.
\end{observation}

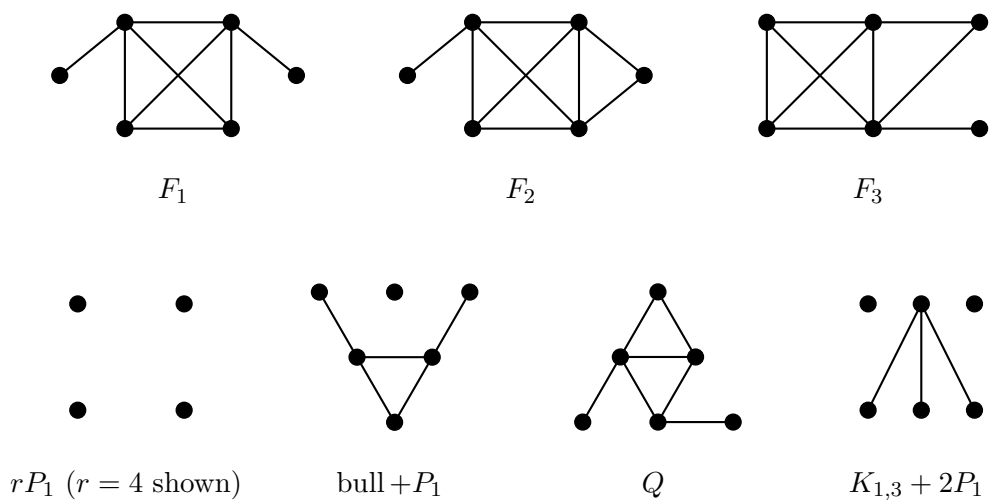
\begin{figure}
\begin{center}
\begin{tabular}{ccc}
\begin{minipage}{0.3\textwidth}
\centering
\scalebox{1}{
{\begin{tikzpicture}[scale=1,rotate=135]
\GraphInit[vstyle=Simple]
\SetVertexSimple[MinSize=6pt]
\Vertices{circle}{a,b,c,d}
\Vertex[a=45,d=1.57313218497]{e}
\Vertex[a=225,d=1.57313218497]{f}
\Edges(a,b,c,d,a,c)
\Edges(b,d)
\Edges(a,e)
\Edges(f,d)
\end{tikzpicture}}}
\end{minipage}
&
\begin{minipage}{0.3\textwidth}
\centering
\scalebox{1}{
{\begin{tikzpicture}[scale=1,rotate=135]
\GraphInit[vstyle=Simple]
\SetVertexSimple[MinSize=6pt]
\Vertices{circle}{a,b,c,d}
\Vertex[a=45,d=1.57313218497]{e}
\Vertex[a=225,d=1.57313218497]{f}
\Edges(a,b,c,d,a,c)
\Edges(b,d)
\Edges(a,e)
\Edges(c,f,d)
\end{tikzpicture}}}
\end{minipage}
&
\begin{minipage}{0.3\textwidth}
\centering
\scalebox{1}{
{\begin{tikzpicture}[xscale=-1,rotate=135]
\GraphInit[vstyle=Simple]
\SetVertexSimple[MinSize=6pt]
\Vertices{circle}{a,b,c,d}
\Vertex[x=1,y=2]{y}
\Vertex[x=2,y=1]{z}
\Edges(a,b,c,d,a,c)
\Edges(b,d)
\Edges(a,z,b,y)
\end{tikzpicture}}}
\end{minipage}
\\
& & \\
$F_1$ & $F_2$ & $F_3$\\
\\
\\
\end{tabular}
\begin{tabular}{cccc}
\begin{minipage}{0.22\textwidth}
\centering
\scalebox{1}{
{\begin{tikzpicture}[scale=1,rotate=135]
\GraphInit[vstyle=Simple]
\SetVertexSimple[MinSize=6pt]
\Vertices{circle}{a,b,c,d}
\end{tikzpicture}}}
\end{minipage}
&
\begin{minipage}{0.22\textwidth}
\centering
\scalebox{1}{
{\begin{tikzpicture}[scale=1,rotate=90]
\GraphInit[vstyle=Simple]
\SetVertexSimple[MinSize=6pt]
\Vertex[x=0,y=0]{a}
\Vertex[a=30,d=1]{b}
\Vertex[a=30,d=2]{c}
\Vertex[a=-30,d=1]{d}
\Vertex[a=-30,d=2]{e}
\Vertex[x=1.73205080757,y=0]{x}
\Edges(c,b,a,d,e)
\Edges(b,d)
\end{tikzpicture}}}
\end{minipage}
&
\begin{minipage}{0.22\textwidth}
\centering
\scalebox{1}{
{\begin{tikzpicture}[scale=1,rotate=30]
\GraphInit[vstyle=Simple]
\SetVertexSimple[MinSize=6pt]
\Vertex[a=0,d=0.57735026919]{a}
\Vertex[a=120,d=0.57735026919]{b}
\Vertex[a=240,d=0.57735026919]{c}
\Vertex[a=60,d=1.15470053838]{d}
\Vertex[a=180,d=1.15470053838]{e}
\Vertex[a=300,d=1.15470053838]{f}
\Edges(a,b,c,a,d,b,e)
\Edge(c)(f)
\end{tikzpicture}}}
\end{minipage}
&
\begin{minipage}{0.22\textwidth}
\centering
\scalebox{1}{
{\begin{tikzpicture}[scale=1]
\GraphInit[vstyle=Simple]
\SetVertexSimple[MinSize=6pt]
\Vertex[x=0,y=1.4142135623]{a}
\Vertex[x=-0.70710678118,y=0]{b}
\Vertex[x=0,y=0]{c}
\Vertex[x=0.70710678118,y=0]{d}
\Vertex[x=-0.70710678118,y=1.4142135623]{e}
\Vertex[x=0.70710678118,y=1.4142135623]{f}
\Edges(c,a,b)
\Edges(a,d)
\end{tikzpicture}}}
\end{minipage}\\
\\
$rP_1$~($r=\nobreak 4$~shown) & $\bull+\nobreak P_1$ & $Q$ & $K_{1,3}+\nobreak 2P_1$
\end{tabular}
\end{center}
\caption{The graphs~$H$ from Theorem~\ref{thm:split-classification} for which the classes of $H$-free split graphs and $\overline{H}$-free split graphs have bounded clique-width.}
\label{fig:bounded-split}
\end{figure}

Brandst{\"a}dt, Dabrowski, Huang and Paulusma considered $H$-free split graphs in~\cite{BDHP16}.
They considered the two cases $H=F_1$ and $H=F_2$ that are open for $H$-free chordal graphs (Open Problem~\ref{oprob:chordal}) and proved that the classes of $F_1$-free split graphs and $F_2$-free split graphs have bounded clique-width.
They showed the same result for $(\bull+\nobreak P_1)$-free split graphs, $Q$-free split graphs, $(K_{1,3}+\nobreak 2P_1)$-free split graphs and $F_3$-free split graphs; see \figurename~\ref{fig:bounded-split} for a description of each of these graphs.
They also proved that for every integer~$r \geq 1$, the clique-width of $rP_1$-free split graphs is at most~$r+\nobreak 1$.
Moreover, they showed the following: if~$H$ is a graph with at least one edge and at least one non-edge that is not an induced subgraph of a graph in $\{F_4,\overline{F_4},F_5,\overline{F_5}\}$ (see \figurename~\ref{fig:open-split}), then the class of $H$-free split graphs has unbounded clique-width.
Note that both~$F_4$ and~$F_5$ have seven vertices.
The $6$-vertex induced subgraphs of~$F_4$ are: $\bull +\nobreak P_1, \overline{F_1}, \overline{F_3}$ and~$K_{1,3}+\nobreak 2P_1$.
The $6$-vertex induced subgraphs of~$F_5$ are: $\bull +\nobreak P_1, F_1,F_2,\overline{F_2},F_3,\overline{F_3}$ and~$Q$.
The above results lead to the following theorem.

\begin{theorem}[\cite{BDHP16}]\label{thm:split-classification}
Let~$H$ be a graph not in $\{F_4,\overline{F_4},F_5,\overline{F_5}\}$.
The class of $H$-free split graphs has bounded clique-width if and only if:
\begin{enumerate}[(i)]
\item $H=rP_1$ for some $r\geq 1$\\[-2em]
\item $H=K_r$ for some $r\ge 1$, or\\[-2em]
\item $H$ is an induced subgraph of a graph in $\{F_4,\overline{F_4},F_5,\overline{F_5}\}$.
\end{enumerate}
\end{theorem} 

Theorem~\ref{thm:split-classification}, combined with Observation~\ref{o-splito}, leaves two open cases: $F_4$ (or equivalently~$\overline{F_4}$) and~$F_5$ (or equivalently~$\overline{F_5}$); see also~\cite{BDHP16}.

\begin{sloppypar}
\begin{oproblem}\label{oprob:split}
Determine whether the class of $H$-free split graphs has bounded or unbounded clique-width when $H=F_4$ or $H=F_5$.
\end{oproblem}
\end{sloppypar}

\begin{figure}
\begin{center}
\begin{tabular}{cc}
\begin{minipage}{0.33\textwidth}
\centering
\scalebox{1}{
{\begin{tikzpicture}[scale=1]
\GraphInit[vstyle=Simple]
\SetVertexSimple[MinSize=6pt]
\Vertex[x=0,y=0]{a}
\Vertex[a=60,d=1]{b}
\Vertex[a=-60,d=1]{c}
\Vertex[a=0,d=1]{d}
\Vertex[a=0,d=2]{x}
\Vertex[a=0,d=-1]{y}
\Vertex[a=-25,d=1.25]{z}
\Edges(y,a,b,d,c,a,d,x)
\end{tikzpicture}}}
\end{minipage}
&
\begin{minipage}{0.33\textwidth}
\centering
\scalebox{1}{
{\begin{tikzpicture}[scale=1,rotate=135]
\GraphInit[vstyle=Simple]
\SetVertexSimple[MinSize=6pt]
\Vertices{circle}{a,b,c,d}
\Vertex[a=225,d=1.57313218497]{f}
\Vertex[x=1,y=2]{y}
\Vertex[x=2,y=1]{z}
\Edges(a,b,c,d,a,c)
\Edges(b,d)
\Edges(a,z,b,y)
\Edges(f,d)
\end{tikzpicture}}}
\end{minipage}\\
& \\
$F_4$ & $F_5$
\end{tabular}
\end{center}
\caption{\label{fig:open-split} The (only) two graphs for which it is not known whether or not the classes of $H$-free split graphs and $\overline{H}$-free split graphs have bounded clique-width.}
\end{figure}
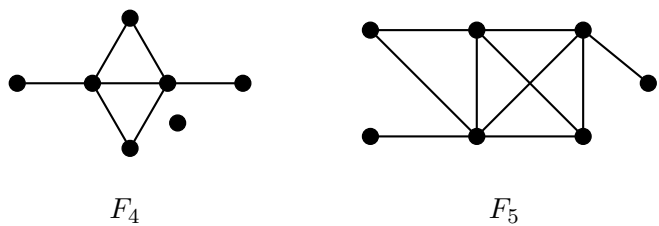

Note that a split graph with split partition $(K,I)$ can be changed into a bipartite graph with bipartition classes~$K$ and~$I$ by applying a subgraph complementation on~$K$.
Hence, due to Fact~\ref{fact:comp}, there is a close relationship between boundedness of clique-width for subclasses of split graphs and for subclasses of bipartite graphs.
As such, it is natural to also consider the class of bipartite graphs as our class~${\cal X}$.
We note that the relationship between split graphs and bipartite graphs involves some subtleties as a split graph can have two non-isomorphic split partitions and a (disconnected) bipartite graph may have more than one bipartition (see~\cite{BDHP16} for a precise explanation).
Nevertheless, results on boundedness of clique-width for $H$-free bipartite graphs, which we discuss below, have proved useful in proving Theorem~\ref{thm:split-classification}.

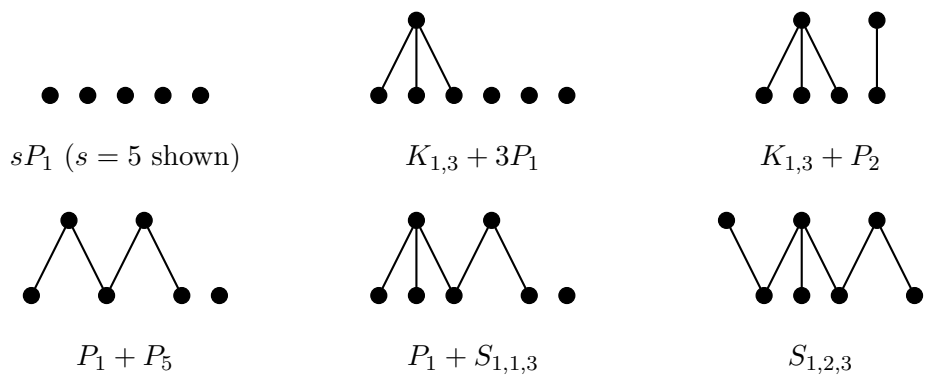
\begin{figure}
\begin{center}
\begin{tabular}{ccc}
\begin{minipage}{0.3\textwidth}
\centering
\begin{tikzpicture}[scale=1]
\GraphInit[vstyle=Simple]
\SetVertexSimple[MinSize=6pt]
\Vertex[x=0,y=0]{x0}
\Vertex[x=0.5,y=0]{x1}
\Vertex[x=1,y=0]{x2}
\Vertex[x=1.5,y=0]{x3}
\Vertex[x=2,y=0]{x4}
\SetVertexSimple[MinSize=6pt,LineColor=white,FillColor=white]
\Vertex[x=0,y=1]{blank}
\end{tikzpicture}
\end{minipage}
&
\begin{minipage}{0.3\textwidth}
\centering
\begin{tikzpicture}[scale=1]
\GraphInit[vstyle=Simple]
\SetVertexSimple[MinSize=6pt]
\Vertex[x=0.5,y=0]{x0}
\Vertex[x=1,y=0]{x1}
\Vertex[x=1.5,y=0]{x2}
\Vertex[x=2,y=0]{x3}
\Vertex[x=2.5,y=0]{x4}
\Vertex[x=3,y=0]{x5}
\Vertex[x=1,y=1]{z}
\Edge(x0)(z)
\Edge(x1)(z)
\Edge(x2)(z)
\end{tikzpicture}
\end{minipage}
&
\begin{minipage}{0.3\textwidth}
\centering
\begin{tikzpicture}[scale=1]
\GraphInit[vstyle=Simple]
\SetVertexSimple[MinSize=6pt]
\Vertex[x=0.5,y=0]{x0}
\Vertex[x=1,y=0]{x1}
\Vertex[x=1.5,y=0]{x2}
\Vertex[x=2,y=0]{x3}
\Vertex[x=1,y=1]{z}
\Vertex[x=2,y=1]{y}
\Edge(x0)(z)
\Edge(x1)(z)
\Edge(x2)(z)
\Edge(x3)(y)
\end{tikzpicture}
\end{minipage}\\\\
$sP_1$ ($s=5$ shown) & $K_{1,3}+\nobreak 3P_1$ & $K_{1,3}+\nobreak P_2$\\
\\
\begin{minipage}{0.3\textwidth}
\centering
\begin{tikzpicture}[scale=1]
\GraphInit[vstyle=Simple]
\SetVertexSimple[MinSize=6pt]
\Vertex[x=0,y=0]{x0}
\Vertex[x=1,y=0]{x2}
\Vertex[x=2,y=0]{x4}
\Vertex[x=2.5,y=0]{x5}
\Vertex[x=0.5,y=1]{x1}
\Vertex[x=1.5,y=1]{x3}
\Edges(x0,x1,x2,x3,x4)
\end{tikzpicture}
\end{minipage}
&
\begin{minipage}{0.3\textwidth}
\centering
\begin{tikzpicture}[scale=1]
\GraphInit[vstyle=Simple]
\SetVertexSimple[MinSize=6pt]
\Vertex[x=0.5,y=0]{x0}
\Vertex[x=1,y=0]{x1}
\Vertex[x=1.5,y=0]{x2}
\Vertex[x=2.5,y=0]{x4}
\Vertex[x=3,y=0]{x5}
\Vertex[x=1,y=1]{z}
\Vertex[x=2,y=1]{x3}
\Edge(x0)(z)
\Edge(x1)(z)
\Edge(x2)(z)
\Edges(x2,x3,x4)
\end{tikzpicture}
\end{minipage}
&
\begin{minipage}{0.3\textwidth}
\centering
\begin{tikzpicture}[scale=1]
\GraphInit[vstyle=Simple]
\SetVertexSimple[MinSize=6pt]
\Vertex[x=0.5,y=0]{x0}
\Vertex[x=1,y=0]{x1}
\Vertex[x=1.5,y=0]{x2}
\Vertex[x=2.5,y=0]{x4}
\Vertex[x=1,y=1]{z}
\Vertex[x=2,y=1]{x3}
\Vertex[x=0,y=1]{x5}
\Edges(x5,x0,z)
\Edge(x1)(z)
\Edge(x2)(z)
\Edges(x2,x3,x4)
\end{tikzpicture}
\end{minipage}\\\\
$P_1+\nobreak P_5$ & $P_1+\nobreak S_{1,1,3}$ & $S_{1,2,3}$
\end{tabular}
\caption{\label{fig:bip}The graphs~$H$ for which the class of $H$-free bipartite graphs has bounded clique-width.}
\end{center}
\end{figure}

Lozin~\cite{Lo02} proved that the clique-width of $S_{1,2,3}$-free bipartite graphs is at most~$5$.
He previously proved this bound in~\cite{Lo00} for $(\sun_4,S_{1,2,3})$-free bipartite graphs where $\sun_4$ is the graph obtained from a $4$-vertex cycle on vertices $u_1,\ldots,u_4$ by adding four new vertices $v_1,\ldots,v_4$ with edges $u_iv_i$ for $i\in\{1,\ldots,4\}$.
Fouquet, Giakoumakis and Vanherpe~\cite{FGV99} proved that $(P_7,S_{1,2,3})$-free bipartite graphs have clique-width at most~$4$.

Lozin and Volz~\cite{LV08} used the above results to continue the study of~\cite{LR06} into boundedness of clique-width of $H$-free bipartite graphs.
They fully classified the boundedness of clique-width for a variant of $H$-free bipartite graphs called strongly $H^\ell$-free graphs, where~$H$ is forbidden with respect to a specified bipartition given by some labelling~$\ell$ (which is unique if~$H$ is connected).
Dabrowski and Paulusma~\cite{DP14} proved a similar (but different) dichotomy for a relaxation of this variant called weakly $H^\ell$-free graphs, which is the variant used for proving some of the cases in Theorem~\ref{thm:split-classification}.
We refer to~\cite{DP14} for an explanation of strongly and weakly $H^\ell$-free bipartite graphs.
Using the above results Dabrowski and Paulusma~\cite{DP14} also gave a full classification for $H$-free bipartite graphs,
that is, with~$H$ forbidden as an induced subgraph, as before; see also \figurename~\ref{fig:bip}.

\begin{theorem}[\cite{DP14}]\label{thm:bipartite}
Let~$H$ be a graph.
The class of $H$-free bipartite graphs has bounded clique-width if and only~if:
\begin{enumerate}[(i)]
\item $H=sP_1$ for some $s\geq 1$\\[-2em]
\item $H \ssi K_{1,3}+\nobreak 3P_1$\\[-2em]
\item $H \ssi K_{1,3}+\nobreak P_2$\\[-2em]
\item $H \ssi P_1+\nobreak S_{1,1,3}$, or\\[-2em]
\item $H \ssi S_{1,2,3}$.
\end{enumerate}
\end{theorem}

We refer to~\cite{BGM18} for some specific bounds on the clique-width of subclasses of $H$-free split graphs, bipartite graphs and co-bipartite graphs obtained from a decomposition property of $1$-Sperner hypergraphs.

We continue our discussion on finding suitable graph classes~${\cal X}$ for which the classification of boundedness of the clique-width of its $H$-free subclasses differs from the (general) classification for $H$-free graphs in Theorem~\ref{thm:single}. 
Theorem~\ref{t-unitinterval} states that the class of unit interval graphs has unbounded clique-width.
Unit interval graphs are contained in the class of interval graphs, which are contained in the class of chordal graphs by Observation~\ref{o-inclusion}.
Hence, as well as narrowing the class of chordal graphs to split graphs, it is also natural to consider unit interval graphs and interval graphs to be the class~${\cal X}$.
We recall that the class of unit interval graphs is a minimal hereditary graph class of unbounded clique-width~\cite{Lo11}.
Hence the clique-width of $H$-free unit interval graphs is bounded if and only if~$H$ is a unit interval graph.
We refer to~\cite{MR15} for bounds on the clique-width of certain subclasses of unit interval graphs and pose the following open problem.

\begin{oproblem}\label{o-interval}
Determine for which graphs~$H$ the class of $H$-free interval graphs has bounded clique-width.
\end{oproblem}

As mentioned earlier, instead of chordal graphs we can consider other subclasses of weakly chordal graphs as our class~${\cal X}$, such as permutation graphs (the containment follows from Observation~\ref{o-inclusion}).
Recall that even the classes of split permutation graphs and bipartite permutation graphs have unbounded clique-width, as stated in Theorems~\ref{t-splitpermutation} and~\ref{t-bipper}, respectively.
Hence, we could also take each of these three graph classes as the class~${\cal X}$.
However, we recall that the classes of split permutation graphs~\cite{ABLS} and bipartite permutation graphs~\cite{Lo11} are minimal hereditary graph classes of unbounded clique-width.
Hence, the clique-width of $H$-free split permutation graphs is bounded if and only if~$H$ is a split permutation graph, and similarly, the clique-width of $H$-free bipartite permutation graphs is bounded if and only if~$H$ is a bipartite permutation graph.
Recall that Theorem~\ref{thm:single} states that the class of $H$-free graphs has bounded clique-width if and only if~$H$ is an induced subgraph of~$P_4$ and that Theorem~\ref{t-weakly-chordal} states that the same classification holds if we restrict to $H$-free weakly chordal graphs.
Brignall and Vatter proved that the same classification also holds if we further restrict to $H$-free permutation graphs.

\begin{theorem}[\cite{BV19}]\label{t-hfreepermutation}
Let~$H$ be a graph.
The class of $H$-free permutation graphs has bounded clique-width if and only if~$H$ is an induced subgraph of~$P_4$.
\end{theorem}

\begin{proof}
Let~$H$ be a graph and note that if~$H$ is not a permutation graph, then the class of $H$-free permutation graphs equals the class of permutation graphs, which has unbounded clique-width by Theorem~\ref{thm:single}.
We may therefore assume that~$H$ is a permutation graph.
If~$H$ is an induced subgraph of~$P_4$ then the class of $H$-free permutation graphs is a subclass of the class of $P_4$-free graphs and in this case Theorem~\ref{thm:single} completes the proof.

The class of $C_3$-free permutation graphs is equal to the class of bipartite permutation graphs, which has unbounded clique-width by Theorem~\ref{t-bipper}.
Since the class of permutation graphs is closed under complementation (in the definition of permutation graphs, reverse the order of intersections of the line segments with one of the parallel lines), Fact~\ref{fact:comp} implies that $3P_1$-free permutation graphs also have unbounded clique-width.
It therefore remains to consider the case when~$H$ is a $(C_3,3P_1)$-free graph that is not an induced subgraph of~$P_4$.

It is easy to verify that the only $(C_3,3P_1)$-free graph on more than four vertices is~$C_5$.
Since~$C_5$ is not a permutation graph, we may assume that~$H$ has at most four vertices.
By inspection, the only $(C_3,3P_1)$-free graphs~$H$ on at most four vertices that are not induced subgraphs of~$P_4$ are~$C_4$ and~$2P_2$.
As~$C_5$ is not a permutation graph, the class of $(C_4,2P_2)$-free permutation graphs is equal to the class of split permutation graphs, which has unbounded clique-width by Theorem~\ref{t-splitpermutation}.
Therefore the class of $H$-free permutation graphs has unbounded clique-width if $H \in \{C_4,2P_2\}$.
This completes the proof.
\end{proof}

Recall from Observation~\ref{o-inclusion} that bipartite permutation graphs are chordal bipartite,
and that by Theorem~\ref{t-bipper} the class of bipartite permutation graphs has unbounded clique-width.
From these two facts it follows that the class of chordal bipartite graphs has unbounded clique-width.
In contrast, Lozin and Rautenbach~\cite{LR04b} proved that $K_{1,t}^+$-free chordal bipartite graphs have bounded clique-width (recall that~$K_{1,t}^+$ is the graph obtained from the star~$K_{1,t}$ by subdividing one of its edges).
Subdividing all three edges of the claw~$K_{1,3}$ yields the graph~$S_{2,2,2}$.
As every bipartite permutation graph is $S_{2,2,2}$-free chordal bipartite, the class of $S_{2,2,2}$-free chordal bipartite graphs has unbounded clique-width, again due to Theorem~\ref{t-bipper}. 

The above discussion leads to the following open problems.
Let~$E_t$ denote the graph obtained from the star~$K_{1,t+1}$ after subdividing exactly two of its edges.
Kami\'nski, Lozin and Milani\v{c}~\cite{KLM09} asked the question: for which~$t$, does the class of $E_t$-free chordal bipartite graphs have bounded clique-width?
For $t\leq 2$, the class of $E_t$-free graphs has bounded clique-width by Theorem~\ref{thm:bipartite}, as $E_2=S_{1,2,2}$.
Hence $t=3$ is the first open case.
By taking the class of chordal bipartite graphs as the class~${\cal X}$, we can pose a more general open problem.

\begin{oproblem}\label{o-chordalbipartite}
Determine for which graphs~$H$ the class of $H$-free chordal bipartite graphs has bounded clique-width.
\end{oproblem}

Boliac and Lozin~\cite{BL02} proved that for a graph~$H$, the class of $H$-free claw-free graphs has bounded clique-width if and only if $H\ssi P_4$, $H\ssi \paw$ or $H\ssi K_3+\nobreak P_1$ (see also the more general Theorem~\ref{thm:classification2} in Section~\ref{s-2}).
Line graphs form a subclass of the class of claw-free graphs.
Gurski and Wanke~\cite{GW07b} proved that if a line graph has a vertex whose non-neighbours induce a subgraph of clique-width~$k$, then it has clique-width at most $8k+4$, which  would imply, for instance, that $(P_1+\nobreak P_4)$-free line graphs have clique-width at most~$18$ (they then improved this bound to~$14$).
In fact we can show the following classification for the boundedness of clique-width of $(H_1,\ldots,H_p)$-free line graphs.
Recall that~${\cal S}$ is the class of graphs every connected component of which is either a subdivided claw or a path on at least one vertex, whereas~${\cal T}$ consists of all line graphs of graphs in~${\cal S}$.

\begin{sloppypar}
\begin{theorem}\label{t-line}
Let $\{H_1,\ldots,H_p\}$ be a finite set of graphs.
Then the class of $(H_1,\ldots,H_p)$-free line graphs has bounded clique-width if and only if $H_i\in {\cal T}$ for some $i \in \{1,\ldots,p\}$.
\end{theorem}
\end{sloppypar}

\begin{proof}
First suppose that $H_i\in {\cal T}$ for some $i \in \{1,\ldots,p\}$.
By definition of~${\cal T}$, it follows that~$H_i$ is the line graph of some graph $F \in {\cal S}$.
Because~$F$ is in~${\cal S}$, forbidding~$F$ as a (not necessarily induced) subgraph of~$G$ is the same as forbidding~$F$ as a minor by Observation~\ref{o-freefree}.
Moreover, $F$ is planar.
By a result of Bienstock, Robertson, Seymour and Thomas~\cite{BRST91}, every graph that does not contain some fixed planar graph as a minor has bounded path-width.
Hence, the class of $F$-subgraph-free graphs has bounded path-width and consequently, bounded treewidth.
Then, by Theorem~\ref{t-gw1}, the class of $H_i$-free graphs, and thus the class of $(H_1,\ldots,H_p)$-free graphs, has bounded clique-width.

Now suppose that $H_i\notin {\cal T}$ for every $i \in \{1,\ldots,p\}$.
Then every~$H_i$ has a connected component $H'_i \notin {\cal T}$.
We may assume without loss of generality that each~$H_i$ is a line graph (otherwise forbidding it does not affect the class defined; 
if no~$H_i$ is a line graph, then the class of $(H_1,\ldots,H_p)$-free line graphs is the class of all line graphs, which has unbounded clique-width~\cite{BL02}).
Since every $H_i' \notin {\cal T}$, every~$H_i'$ is not isomorphic to $K_3$.
Hence, for every~$H_i'$ there exists a unique graph~$F_i'$ such that~$H_i'$ is the line graph of~$F_i'$ (see, for example,~\cite{Ha69}).
Since $H'_i \notin {\cal T}$, it follows that $F'_i\notin {\cal S}$, which means that there exists a positive integer~$k_i$, such that the class of $F_i'$-subgraph-free graphs contains the class of $k_i$-subdivided walls.
We let $k=\max\{k_i\; |\; 1\leq i\leq p\}$.
Then the class of $(F_1',\ldots,F_p')$-subgraph-free graphs contains the class of $k$-subdivided walls.
As the class of $k$-subdivided walls has unbounded clique-width by Corollary~\ref{cor:walls}, it follows that the class of $(F_1',\ldots,F_p')$-subgraph-free graphs has unbounded clique-width
and hence unbounded treewidth~\cite{CO00}.
Then, by Theorem~\ref{t-gw1}, the class of $(H_1',\ldots,H_p')$-free line graphs has unbounded clique-width.
Since the class of $(H_1,\ldots,H_p)$-free line graphs contains the class of $(H_1',\ldots,H_p')$-free line graphs, it follows that the class of $(H_1,\ldots,H_p)$-free line graphs also has unbounded clique-width.
\end{proof}

\subsection{Forbidding A Small Number of Graphs}\label{s-2}

As discussed, even the case when only two induced subgraphs~$H_1$ and~$H_2$ are forbidden has not yet been fully classified, and there are only partial results for the cases where three or four induced subgraphs are forbidden.
Besides the class of $(C_4,C_5,2P_2)$-free graphs (split graphs)~\cite{MR99}, it is, for example, known that the classes of $(C_4,\allowbreak K_{1,3},\allowbreak K_4,\allowbreak \diamondgraph)$-free graphs~\cite{BL02,BELL06} and $(3P_2,\allowbreak P_2+\nobreak P_4,\allowbreak P_6,\allowbreak\gem)$-free graphs have unbounded clique-width~\cite{DP16}.
Recall that the gem is the graph~$\overline{P_1+P_4}$ (see \figurename~\ref{fig:particular-graphs}) and that the hammer is the graph~$T_{0,0,2}$ (see \figurename~\ref{fig:Thij-examples}).
It is known that the clique-width of $(\hammer,\gem,S_{1,1,2})$-free graphs is at most~$7$~\cite{BLV03}.
However, unlike the case for two forbidden induced subgraphs, no large-scale systematic study has been initiated for finitely defined hereditary graphs classes with more than two forbidden induced subgraphs;
in Sections~\ref{s-four} and~\ref{s-complement}, respectively, we discuss two studies~\cite{BDLM05,BDJLPZ17} with partial results in this direction.
In this section, we focus only on $(H_1,H_2)$-free graphs.

Despite the classification for $H$-free graphs (Theorem~\ref{thm:single}) and many existing results for (un)boundedness of clique-width for $(H_1,H_2)$-free graphs~\cite{BDHP16,BELL06,BKM06,BL02,BLM04,BLM04b,BM02,DGP14,DHP0,DLRR12} over the years, the number of open cases $(H_1,H_2)$ was only recently proven to be finite, in~\cite{DP16}.
This was done by combining the existing known results together with a number of new results for $(H_1,H_2)$-free graphs, and led to a classification that left~13 non-equivalent open cases.\footnote{\label{footnote:equivalence}Given four graphs $H_1,H_2,H_3,H_4$, the classes of $(H_1,H_2)$-free graphs and $(H_3,H_4)$-free graphs are said to be equivalent if the unordered pair $H_3,H_4$ can be obtained from the unordered pair $H_1,H_2$ by some combination of the operations: (i) complementing both graphs in the pair, and (ii) if one of the graphs in the pair is~$3P_1$, replacing it with $P_1+\nobreak P_3$ or vice versa.
If two classes are equivalent, then one of them has bounded clique-width if and only if the other one does~\cite{DP16}.}
This number has been reduced to five non-equivalent open cases by four later papers~\cite{BDJLPZ17,BDJP18,DDP17,DLP17}, and the current state-of-the-art is as follows (recall that~${\cal S}$ is the class of graphs each connected component of which is either a subdivided claw or a path and see also \figurenames~\ref{fig:particular-graphs}, \ref{fig:Thij-examples} and~\ref{fig:bip} in which a number of the graphs mentioned below are displayed).

\begin{theorem}[\cite{BDJP18}]\label{thm:classification2}
Let~${\cal G}$ be a class of graphs defined by two forbidden induced subgraphs.
Then:
\begin{enumerate}
\item ${\cal G}$ has bounded clique-width if it is equivalent to a class of $(H_1,H_2)$-free graphs such that one of the following holds:
\begin{enumerate}[(i)]
\item \label{thm:classification2:bdd:P4} $H_1$ or $H_2 \ssi P_4$\\[-1.5em]
\item \label{thm:classification2:bdd:ramsey} $H_1=K_s$ and $H_2=tP_1$ for some $s,t\geq 1$\\[-1.5em]
\item \label{thm:classification2:bdd:P_1+P_3} $H_1 \ssi \paw$ and $H_2 \ssi K_{1,3}+\nobreak 3P_1,\; K_{1,3}+\nobreak P_2,\;\allowbreak P_1+\nobreak P_2+\nobreak P_3,\;\allowbreak P_1+\nobreak P_5,\;\allowbreak P_1+\nobreak S_{1,1,2},\;\allowbreak P_2+\nobreak P_4,\;\allowbreak P_6,\; \allowbreak S_{1,1,3}$ or~$S_{1,2,2}$\\[-1.5em]
\item \label{thm:classification2:bdd:2P_1+P_2} $H_1 \ssi \diamondgraph$ and $H_2\ssi P_1+\nobreak 2P_2,\; 3P_1+\nobreak P_2$ or~$P_2+\nobreak P_3$\\[-1.5em]
\item \label{thm:classification2:bdd:P_1+P_4} $H_1 \ssi \gem$ and $H_2 \ssi P_1+\nobreak P_4$ or~$P_5$\\[-1.5em]
\item \label{thm:classification2:bdd:K_13} $H_1\ssi K_3+\nobreak P_1$ and $H_2 \ssi K_{1,3}$, or\\[-1.5em]
\item \label{thm:classification2:bdd:2P1_P3} $H_1\ssi \overline{2P_1+\nobreak P_3}$ and $H_2\ssi 2P_1+\nobreak P_3$.
\end{enumerate}
\item ${\cal G}$ has unbounded clique-width if it is equivalent to a class of $(H_1,H_2)$-free graphs such that one of the following holds:
\begin{enumerate}[(i)]
\item \label{thm:classification2:unbdd:not-in-S} $H_1\not\in {\cal S}$ and $H_2 \not \in {\cal S}$\\[-1.5em]
\item \label{thm:classification2:unbdd:not-in-co-S} $H_1\notin \overline{{\cal S}}$ and $H_2 \not \in \overline{{\cal S}}$\\[-1.5em]
\item \label{thm:classification2:unbdd:K_13or2P_2} $H_1 \si K_3+\nobreak P_1$ or~$C_4$ and $H_2 \si 4P_1$ or~$2P_2$\\[-1.5em]
\item \label{thm:classification2:unbdd:2P_1+P_2} $H_1 \si \diamondgraph$ and $H_2 \si K_{1,3},\; 5P_1,\; P_2+\nobreak P_4$ or~$P_6$\\[-1.5em]
\item \label{thm:classification2:unbdd:3P_1} $H_1 \si K_3$ and $H_2 \si 2P_1+\nobreak 2P_2,\; 2P_1+\nobreak P_4,\; 4P_1+\nobreak P_2,\; 3P_2$ or~$2P_3$\\[-1.5em]
\item \label{thm:classification2:unbdd:4P_1} $H_1 \si K_4$ and $H_2 \si P_1 +\nobreak P_4$ or~$3P_1+\nobreak P_2$, or\\[-1.5em]
\item \label{thm:classification2:unbdd:gem} $H_1 \si \gem$ and $H_2 \si P_1+\nobreak 2P_2$.
\end{enumerate}
\end{enumerate}
\end{theorem}

\begin{example}
As an example of how results from Section~\ref{s-x} were useful in proving Theorem~\ref{thm:classification2}, consider the case when $(H_1,H_2)=(K_4,2P_1+\nobreak P_3)$.
In~\cite{BDHP17}, it was shown that $(K_4,2P_1+\nobreak P_3)$-free graphs have bounded clique-width.
This was proven as follows.
First, Theorem~\ref{thm:chordal-classification} was applied to solve the case when the given $(K_4,2P_1+\nobreak P_3)$-free graph~$G$ is chordal.
If~$G$ is not chordal, then~$G$ must contain a cycle~$C$ of length at least~$4$.
As~$G$ is $(2P_1+\nobreak P_3)$-free, $C$ can have length at most~$7$.
This leads to a case distinction depending on the length of~$C$.
In each case, the set of vertices of~$G$ not on~$C$ is partitioned according to the intersection of their set of neighbours with~$C$.
This partition is then analysed and the facts from Section~\ref{s-operations} are used to modify~$G$ into a graph belonging to a class known to have bounded clique-width.
\end{example}

As mentioned earlier, Theorem~\ref{thm:classification2} does not cover five (non-equivalent) cases; see also~\cite{BDJP18}.

\begin{oproblem}\label{oprob:twographs}
Does the class of $(H_1,H_2)$-free graphs have bounded or unbounded clique-width when:
\begin{enumerate}[(i)]
\renewcommand{\theenumi}{(\roman{enumi})}
\renewcommand{\labelenumi}{(\roman{enumi})}
\item\label{oprob:twographs:3P_1}$H_1=K_3$ and $H_2 \in \{P_1+\nobreak S_{1,1,3},\allowbreak S_{1,2,3}\}$\\[-1em]
\item\label{oprob:twographs:2P_1+P_2}$H_1=\diamondgraph$ and $H_2 \in \{P_1+\nobreak P_2+\nobreak P_3,\allowbreak P_1+\nobreak P_5\}$
\item\label{oprob:twographs:P_1+P_4}$H_1=\gem$ and $H_2=P_2+\nobreak P_3$.
\end{enumerate}
\end{oproblem}

As discussed in~\cite{DLP17}, it would be interesting to find out if $H$-free bipartite graphs and $H$-free triangle-free graphs have the same classification with respect to the boundedness of their clique-width.
It follows from Theorems~\ref{thm:bipartite} and~\ref{thm:classification2} that the evidence so far is affirmative.
Nevertheless, Open Problem~\ref{oprob:twographs}.\ref{oprob:twographs:3P_1} shows that two remaining cases still need to be solved, namely 
$H=P_1+\nobreak S_{1,1,2}$ and $H=S_{1,2,3}$.

We will prove two partial results for the two cases in Open Problem~\ref{oprob:twographs}.\ref{oprob:twographs:3P_1}.
These results also illustrate some of the previously discussed techniques.
Namely, we show that the class of prime $(K_3,C_5,S_{1,2,3})$-free graphs has bounded clique-width (Proposition~\ref{p-noC5-partial-b}) and that the class of $(K_3,C_5,P_1+\nobreak S_{1,1,3})$-free graphs has bounded clique-width (Proposition~\ref{p-noC5-partial2}).
Combining Propositions~\ref{p-noC5-partial-b} and~\ref{p-noC5-partial2} with Fact~\ref{fact:prime} implies that in both cases of Open Problem~\ref{oprob:twographs}.\ref{oprob:twographs:3P_1} we need only consider prime graphs that contain~$C_5$ as an induced subgraph.

For Proposition~\ref{p-noC5-partial-b} we need the following lemma, which follows from~\cite[Lemma~8]{DDP17}.\footnote{\cite[Lemma~8]{DDP17} is about $(K_3,C_5,S_{1,2,3})$-free graphs without \defword{false twins}, that is, without pairs of non-adjacent vertices which have the same set of neighbours. Prime graphs have no false twins by definition.} Proposition~\ref{p-noC5-partial2} is a new result.

\begin{lemma}[\cite{DDP17}]\label{lem:noC5-partial}
If~$G$ is a prime $(K_3,C_5,S_{1,2,3})$-free graph, then~$G$ is either bipartite or a cycle.
\end{lemma}

\begin{sloppypar}
\begin{prop}\label{p-noC5-partial-b}
The class of prime $(K_3,C_5,S_{1,2,3})$-free graphs has bounded clique-width.
\end{prop}
\end{sloppypar}
\begin{proof}
If a $(K_3,C_5,S_{1,2,3})$-free graph is bipartite, then it is an $S_{1,2,3}$-free bipartite graph and we are done by Theorem~\ref{thm:bipartite}.
If it is a cycle then it has maximum degree~$2$, and we are done by Proposition~\ref{p-atmost2}.
By Lemma~\ref{lem:noC5-partial} this completes the proof.
\end{proof}

\begin{prop}\label{p-noC5-partial2}
The class of $(K_3,C_5,P_1+\nobreak S_{1,1,3})$-free graphs has bounded clique-width.
\end{prop}

\begin{proof}
Let~$G$ be a $(K_3,C_5,P_1+\nobreak S_{1,1,3})$-free graph.
Since the clique-width of a graph equals the maximum of the clique-width of its components, we may assume that~$G$ is connected.
We may assume that~$G$ is not bipartite, otherwise it is a $(P_1+\nobreak S_{1,1,3})$-free bipartite graph, in which case it has bounded clique-width by Theorem~\ref{thm:bipartite}.
As~$G$ is $(C_3,C_5)$-free (since $C_3=K_3$), it contains an induced odd cycle~$C$ on~$k$ vertices, say $v_1,v_2, \ldots, v_k$ in that order, where $k \geq 7$.
We may assume without loss of generality that~$C$ is an odd cycle of minimum length in~$G$.

If $V(G)=V(C)$, then~$G$ has maximum degree~$2$ and we can use Proposition~\ref{p-atmost2}.
From now on we assume that~$G$ contains at least one vertex not on~$C$.
Suppose that there is a vertex~$v$ that is adjacent to at least two vertices of~$C$.
As~$C$ has minimal length and~$G$ is $(C_3,C_5)$-free, $v$ must be adjacent to precisely two vertices of~$C$, which must be at distance~$2$ from each other on~$C$.

For $i\in \{1,\ldots,k\}$, let~$V_i$ be the set of vertices outside~$C$ that are adjacent to~$v_{i-1}$ and~$v_{i+1}$ (subscripts on vertices and vertex sets are interpreted modulo~$k$ throughout the proof), and let~$W_i$ be the set of vertices whose unique neighbour in~$C$ is~$v_i$.
Finally, let~$U$ be the set of vertices that have no neighbour in~$C$.
Thus every vertex in~$G$ is in~$C$, $U$ or in some set~$V_i$ or~$W_i$ for some $i \in \{1,\ldots,k\}$.
Moreover, as~$G$ is connected, there must be at least one set of the form~$V_i$ or~$W_i$ that is non-empty.
We may assume without loss of generality that there is a vertex $v \in V_1 \cup W_2$.
If $k \geq 9$ then $G[v_7,v_2,v,v_1,v_3,v_4,v_5]$ is a $P_1+\nobreak S_{1,1,3}$, a contradiction.
We conclude that $k=7$.

We now prove five claims, the first of which follows immediately from the fact that~$G$ is $K_3$-free.

\clmnonewline{\label{ViWi-indep}For $i \in \{1,\ldots,7\}$, $V_i$ and~$W_i$ are independent sets.}

\clm{\label{clm:U-small}For every $i \in \{1,\ldots,7\}$, $V_i$ and~$W_i$ are complete to~$U$, and $|U| \leq 1$.}
Suppose, for contradiction, that a vertex $x \in V_1 \cup W_2$ is non-adjacent to~$y \in U$.
Then $G[y,v_2,x,v_1,v_3,v_4,v_5]$ is a $P_1+\nobreak S_{1,1,3}$, a contradiction.
By symmetry, this proves the first part of the claim.
Now suppose that~$U$ contains at least two vertices~$y$ and~$y'$.
Then $v \in V_1 \cup W_2$ is adjacent to both~$y$ and~$y'$.
Since~$G$ is $K_3$-free, it follows that~$y$ and~$y'$ are not adjacent.
Then $G[v_6,v,y,y',v_2,v_3,v_4]$ is a $P_1+\nobreak S_{1,1,3}$, a contradiction.
This proves the second part of the claim.

\clm{\label{clm:Wi-small}For $i \in \{1,\ldots,7\}$, $|W_i| \leq 1$.}
Suppose that $x,y \in W_1$.
By Claim~\ref{ViWi-indep}, we find that~$x$ is non-adjacent to~$y$.
Then $G[v_6,v_1,x,y,v_2,v_3,v_4]$ is a $P_1+\nobreak S_{1,1,3}$, a contradiction.
The claim follows by symmetry.

\medskip
\noindent
A set of vertices is \defword{large} if it contains at least two vertices and \defword{small} otherwise.

\begin{sloppypar}
\clm{\label{clm:Vi-Vi+2-comp}For $i,j \in \{1,\ldots,7\}$, if~$v_i$ is adjacent to~$v_j$ and at least one of~$V_i$ and~$V_j$ is large, then~$V_i$ is complete to~$V_j$.}
Suppose that there are vertices $x,x' \in V_2$ and $y \in V_3$ such that~$y$ is non-adjacent to~$x'$.
By Claim~\ref{ViWi-indep}, $x$ is non-adjacent to~$x'$.
Then $G[x',y,x,v_2,v_4,v_5,v_6]$ or $G[y,v_1,x,x',v_7,v_6,v_5]$ is a $P_1+\nobreak S_{1,1,3}$ if~$y$ is adjacent or non-adjacent to~$x$, respectively, a contradiction.
The claim follows by symmetry.
\end{sloppypar}

\clm{\label{clm:Vi-Vi+k-anti}For distinct $i,j \in \{1,\ldots,7\}$, if a vertex of~$V_i$ has a neighbour in~$V_j$, then~$v_i$ is adjacent to~$v_j$.}
Since~$G$ is $K_3$-free, for every~$i$ the set~$V_i$ is anti-complete to the set~$V_{i+2}$.
Moreover, if~$i$ and~$j$ are such that the vertices~$v_i$ and~$v_j$ are at distance more than~$2$ on the cycle, then~$V_i$ and~$V_j$ must be anti-complete, as otherwise there would be a smaller odd cycle than~$C$ in~$G$, contradicting the minimality of~$k$.
This proves Claim~\ref{clm:Vi-Vi+k-anti}.

\medskip
\noindent
Let~$G'$ be the graph obtained from~$G$ by deleting all vertices in small sets~$V_i$, $W_i$ or~$U$ (note that in doing this we delete at most $7+\nobreak 7+\nobreak 1=15$ vertices).
By Fact~\ref{fact:del-vert}, it is sufficient to show that~$G'$ has bounded clique-width.
Let~$V_i'$ be~$V_i$ if~$V_i$ is large and~$\emptyset$ otherwise.
By Claims~\ref{clm:U-small} and~\ref{clm:Wi-small}, $G'$ only contains vertices in~$C$ and the sets~$V_i'$.
By Claim~\ref{ViWi-indep}, each set~$V_i'$ is independent.
Furthermore, by Claim~\ref{clm:Vi-Vi+2-comp}, if~$v_i$ and~$v_j$ are adjacent vertices of~$C$ then~$V_i'$ is complete to~$V_j'$.
By Claim~\ref{clm:Vi-Vi+k-anti}, for all other choices of~$i$ and~$j$, the set~$V_i'$ is anti-complete to~$V_j'$.
This implies that for every $i \in \{1,\ldots,7\}$, the set $V_i' \cup \{v_i\}$ is a module that is an independent set.
We apply seven bipartite complementations, namely between~$V_i$ and~$V_{i+1}$ for $i\in\{1,\ldots,7\}$.
This yields an edgeless graph, which has clique-width~$1$.
By Fact~\ref{fact:bip}, it follows that~$G'$ has bounded clique-width.
Hence~$G$ has bounded clique-width.
This completes the proof.
\end{proof}

\subsection{Forbidding Small Induced Subgraphs}\label{s-four}

Theorem~\ref{thm:single} states that a class of $H$-free graphs has bounded clique-width if and only if~$H$ is an induced subgraph of~$P_4$.
As discussed, one way to obtain more graph classes of bounded clique-width is to extend~$P_4$ by one extra vertex, but then we need to forbid at least one other graph as an induced subgraph besides this \mbox{$1$-vertex} extension of~$P_4$.
In this context, Brandst\"adt and Mosca~\cite{BM03} classified the boundedness of clique-width for ${\cal H}$-free graphs, where~${\cal H}$ is a subset of the set of $P_4$-sparse graphs with five vertices.
Brandst{\"a}dt, Ho\`ang and Le~\cite{BHL03} proved that $(\bull,S_{1,1,2},\overline{S_{1,1,2}}$)-free graphs have bounded clique-width.
Brandst\"adt, Dragan, Le and Mosca proved the following more general dichotomy containing the results of~\cite{BHL03,BM03}; see also~\figurename~\ref{fig:P4-ext}.

\begin{theorem}[\cite{BDLM05}]\label{thm:1-vert-ext-of-P4}
Let~${\cal H}$ be a set of $1$-vertex extensions of~$P_4$.
The class of ${\cal H}$-free graphs has bounded clique-width if and only if~${\cal H}$ is not a subset of any of the following sets:
\begin{enumerate}[(i)]
\item $\{P_1+\nobreak P_4, P_5, S_{1,1,2}, \overline{\banner},C_5,\overline{S_{1,1,2}}\}$, \\[-1.9em]

\item $\{\overline{P_1+\nobreak P_4}, \overline{P_5}, S_{1,1,2}, \banner,C_5,\overline{S_{1,1,2}}\}$, \\[-1.9em]

\item $\{P_1+\nobreak P_4, P_5, S_{1,1,2}, \overline{\banner}, \banner, C_5,\bull\}$, \\[-1.9em]

\item $\{\overline{P_1+\nobreak P_4}, \overline{P_5}, \overline{S_{1,1,2}}, \overline{\banner}, \banner, C_5,\bull\}$ or \\[-1.9em]

\item $\{P_5, \overline{\banner}, \banner, C_5, \overline{P_5}\}$.
\end{enumerate}
\end{theorem}

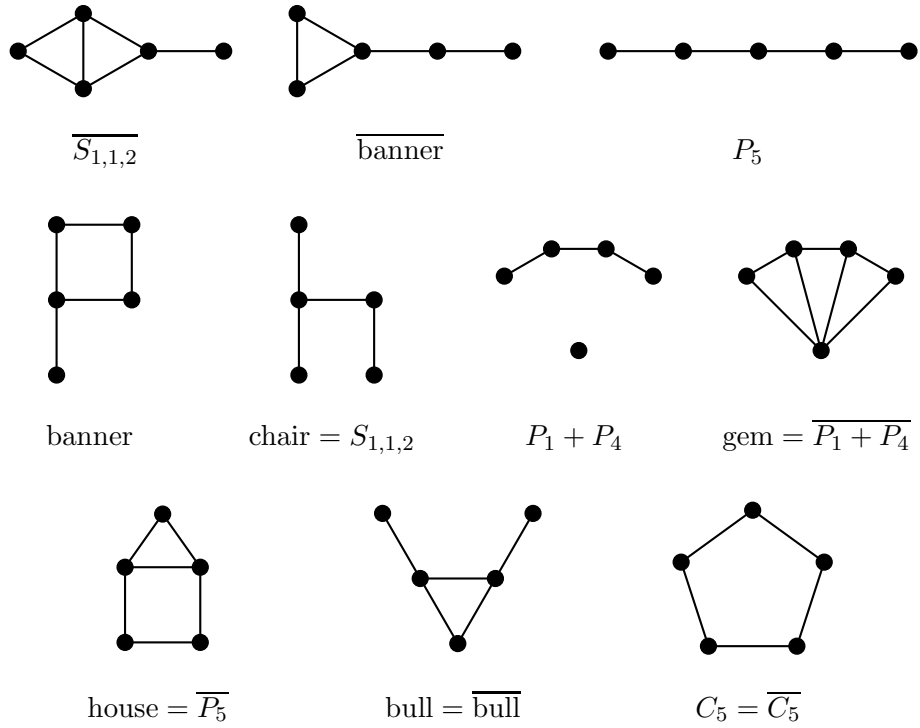
\begin{figure}
\centering
\begin{tabular}{ccc}
\begin{minipage}{0.2\textwidth}
\centering
\scalebox{1}{
{\begin{tikzpicture}[scale=1]
\GraphInit[vstyle=Simple]
\SetVertexSimple[MinSize=6pt]
\Vertex[x=0,y=0.5]{a}
\Vertex[x=0,y=-0.5]{b}
\Vertex[x=-0.86602540378,y=0]{c}
\Vertex[x=0.86602540378,y=0]{d}
\Vertex[x=1.86602540378,y=0]{e}
\Edges(a,b,c,a,d,b)
\Edges(d,e)
\end{tikzpicture}}}
\end{minipage}
&
\begin{minipage}{0.3\textwidth}
\centering
\scalebox{1}{
{\begin{tikzpicture}[scale=1]
\GraphInit[vstyle=Simple]
\SetVertexSimple[MinSize=6pt]
\Vertex[x=0,y=0]{a0}
\Vertex[a=0,d=1]{a1}
\Vertex[a=0,d=2]{a2}
\Vertex[a=-150,d=1]{b}
\Vertex[a=150,d=1]{c}
\Edges(a2,a1,a0,b,c,a0)
\end{tikzpicture}}}
\end{minipage}
&
\begin{minipage}{0.3\textwidth}
\centering
\scalebox{1}{
{\begin{tikzpicture}[scale=1]
\GraphInit[vstyle=Simple]
\SetVertexSimple[MinSize=6pt]
\Vertex[x=0,y=0]{a}
\Vertex[x=1,y=0]{b}
\Vertex[x=2,y=0]{c}
\Vertex[x=3,y=0]{d}
\Vertex[x=4,y=0]{e}
\Edges(a,b,c,d,e)
\end{tikzpicture}}}
\end{minipage}
\\\\
$\overline{S_{1,1,2}}$ & $\overline{\banner}$ & $P_5$\\
\end{tabular}

\vspace*{1.5em}
\begin{tabular}{cccc}
\begin{minipage}{0.2\textwidth}
\centering
\scalebox{1}{
{\begin{tikzpicture}[scale=1]
\GraphInit[vstyle=Simple]
\SetVertexSimple[MinSize=6pt]
\Vertex[x=0,y=0]{x}
\Vertex[x=0,y=1]{y}
\Vertex[x=0,y=2]{z}
\Vertex[x=1,y=1]{a}
\Vertex[x=1,y=2]{b}
\Edges(x,y,z)
\Edges(y,a,b,z)
\end{tikzpicture}}}
\end{minipage}
&
\begin{minipage}{0.2\textwidth}
\centering
\scalebox{1}{
{\begin{tikzpicture}[scale=1]
\GraphInit[vstyle=Simple]
\SetVertexSimple[MinSize=6pt]
\Vertex[x=0,y=0]{x}
\Vertex[x=0,y=1]{y}
\Vertex[x=0,y=2]{z}
\Vertex[x=1,y=0]{a}
\Vertex[x=1,y=1]{b}
\Edges(x,y,z)
\Edges(a,b,y)
\end{tikzpicture}}}
\end{minipage}
&
\begin{minipage}{0.2\textwidth}
\centering
\scalebox{1}{
{\begin{tikzpicture}[scale=1.4,rotate=90]
\GraphInit[vstyle=Simple]
\SetVertexSimple[MinSize=6pt]
\Vertex[x=0,y=0]{x}
\Vertex[a=-45,d=1]{a}
\Vertex[a=-15,d=1]{b}
\Vertex[a=15,d=1]{c}
\Vertex[a=45,d=1]{d}
\Edges(a,b,c,d)
\end{tikzpicture}}}
\end{minipage}
&
\begin{minipage}{0.2\textwidth}
\centering
\scalebox{1}{
{\begin{tikzpicture}[scale=1.4,rotate=90]
\GraphInit[vstyle=Simple]
\SetVertexSimple[MinSize=6pt]
\Vertex[x=0,y=0]{x}
\Vertex[a=-45,d=1]{a}
\Vertex[a=-15,d=1]{b}
\Vertex[a=15,d=1]{c}
\Vertex[a=45,d=1]{d}
\Edges(a,x,b)
\Edges(c,x,d)
\Edges(a,b,c,d)
\end{tikzpicture}}}
\end{minipage}\\\\
$\banner$ & $\chair=S_{1,1,2}$ & $P_1+\nobreak P_4$ & $\gem=\overline{P_1+P_4}$
\end{tabular}

\vspace*{1.5em}
\begin{tabular}{ccc}
\begin{minipage}{0.25\textwidth}
\centering
\scalebox{1}{
{\begin{tikzpicture}[scale=1]
\GraphInit[vstyle=Simple]
\SetVertexSimple[MinSize=6pt]
\Vertex[x=0.5,y=0.70710678118]{x}
\Vertex[x=0,y=0]{a}
\Vertex[x=1,y=0]{b}
\Vertex[x=0,y=-1]{c}
\Vertex[x=1,y=-1]{d}
\Edges(a,x,b)
\Edges(a,b,d,c,a)
\end{tikzpicture}}}
\end{minipage}
&
\begin{minipage}{0.25\textwidth}
\centering
\scalebox{1}{
{\begin{tikzpicture}[scale=1,rotate=90]
\GraphInit[vstyle=Simple]
\SetVertexSimple[MinSize=6pt]
\Vertex[x=0,y=0]{a}
\Vertex[a=-30,d=1]{b}
\Vertex[a=-30,d=2]{c}
\Vertex[a=30,d=1]{d}
\Vertex[a=30,d=2]{e}
\Edges(e,d,a,b,c)
\Edges(b,d)
\end{tikzpicture}}}
\end{minipage}
&
\begin{minipage}{0.25\textwidth}
\centering
\scalebox{1}{
{\begin{tikzpicture}[scale=1,rotate=90]
\GraphInit[vstyle=Simple]
\SetVertexSimple[MinSize=6pt]
\Vertices{circle}{a,b,c,d,e}
\Edges(a,b,c,d,e,a)
\end{tikzpicture}}}
\end{minipage}
\\\\
$\house=\overline{P_5}$ & $\bull=\overline{\bull}$ & $C_5=\overline{C_5}$
\end{tabular}
\caption{\label{fig:P4-ext}The $1$-vertex extensions of~$P_4$.}
\end{figure}

\begin{figure}
\begin{center}
\begin{tabular}{ccccc}
\begin{minipage}{0.16\textwidth}
\centering
\scalebox{1.0}{
{\begin{tikzpicture}[scale=1,rotate=45]
\GraphInit[vstyle=Simple]
\SetVertexSimple[MinSize=6pt]
\Vertices{circle}{a,b,c,d}
\Edges(a,b,c,d,a,c)
\Edges(d,b)
\end{tikzpicture}}}
\end{minipage}
&
\begin{minipage}{0.16\textwidth}
\centering
\scalebox{1.0}{
{\begin{tikzpicture}[scale=1,rotate=45]
\GraphInit[vstyle=Simple]
\SetVertexSimple[MinSize=6pt]
\Vertices{circle}{a,b,c,d}
\Edges(a,b,c,d,a,c)
\end{tikzpicture}}}
\end{minipage}
&
\begin{minipage}{0.16\textwidth}
\centering
\scalebox{1.0}{
{\begin{tikzpicture}[scale=1,rotate=45]
\GraphInit[vstyle=Simple]
\SetVertexSimple[MinSize=6pt]
\Vertices{circle}{a,b,c,d}
\Edges(a,b,c,d,a)
\end{tikzpicture}}}
\end{minipage}
&
\begin{minipage}{0.16\textwidth}
\centering
\scalebox{1.0}{
{\begin{tikzpicture}[scale=1,rotate=45]
\GraphInit[vstyle=Simple]
\SetVertexSimple[MinSize=6pt]
\Vertices{circle}{a,b,c,d}
\Edges(a,b,c,a)
\Edges(d,a)
\end{tikzpicture}}}
\end{minipage}
&
\begin{minipage}{0.16\textwidth}
\centering
\scalebox{1.0}{
{\begin{tikzpicture}[scale=1,rotate=45]
\GraphInit[vstyle=Simple]
\SetVertexSimple[MinSize=6pt]
\Vertices{circle}{a,b,c,d}
\Edges(c,b,d)
\Edges(a,b)
\end{tikzpicture}}}
\end{minipage}\\
\\
$K_4=$ & $\diamondgraph=$ & $C_4=$ & $\paw=$ & $\claw=K_{1,3}=$\\
$\overline{4P_1}$ & $\overline{2P_1+P_2}$ & $\overline{2P_2}$ & $\overline{P_1+P_3}$ & $\overline{K_3+P_1}$\\
\\
\begin{minipage}{0.16\textwidth}
\centering
\scalebox{1.0}{
{\begin{tikzpicture}[scale=1,rotate=45]
\GraphInit[vstyle=Simple]
\SetVertexSimple[MinSize=6pt]
\Vertices{circle}{a,b,c,d}
\end{tikzpicture}}}
\end{minipage}
&
\begin{minipage}{0.16\textwidth}
\centering
\scalebox{1.0}{
{\begin{tikzpicture}[scale=1,rotate=45]
\GraphInit[vstyle=Simple]
\SetVertexSimple[MinSize=6pt]
\Vertices{circle}{a,b,c,d}
\Edges(c,d)
\end{tikzpicture}}}
\end{minipage}
&
\begin{minipage}{0.16\textwidth}
\centering
\scalebox{1.0}{
{\begin{tikzpicture}[scale=1,rotate=45]
\GraphInit[vstyle=Simple]
\SetVertexSimple[MinSize=6pt]
\Vertices{circle}{a,b,c,d}
\Edges(a,b)
\Edges(c,d)
\end{tikzpicture}}}
\end{minipage}
&
\begin{minipage}{0.16\textwidth}
\centering
\scalebox{1.0}{
{\begin{tikzpicture}[scale=1,rotate=45]
\GraphInit[vstyle=Simple]
\SetVertexSimple[MinSize=6pt]
\Vertices{circle}{a,b,c,d}
\Edges(a,b,c)
\end{tikzpicture}}}
\end{minipage}
&
\begin{minipage}{0.16\textwidth}
\centering
\scalebox{1.0}{
{\begin{tikzpicture}[scale=1,rotate=45]
\GraphInit[vstyle=Simple]
\SetVertexSimple[MinSize=6pt]
\Vertices{circle}{a,b,c,d}
\Edges(a,b,c,a)
\end{tikzpicture}}}
\end{minipage}\\
\\
$4P_1=$ & $2P_1+\nobreak P_2=$ & $2P_2=$ & $P_1+\nobreak P_3=$ & $K_3+\nobreak P_1=$\\
$\overline{K_4}$ & $\overline{\diamondgraph}$ & $\overline{C_4}$ & $\overline{\paw}$ & $\overline{\claw}$\\
\\
\end{tabular}
\scalebox{1.0}{
{\begin{tikzpicture}[scale=1,rotate=90]
\GraphInit[vstyle=Simple]
\SetVertexSimple[MinSize=6pt]
\Vertex[x=0,y=0]{a}
\Vertex[x=0,y=1.41421356237]{b}
\Vertex[x=0,y=1.41421356237*2]{c}
\Vertex[x=0,y=1.41421356237*3]{d}
\Edges(a,b,c,d)
\end{tikzpicture}}}
\\
$P_4=\overline{P_4}$
\end{center}
\caption{\label{fig:4-vert}The graphs on four vertices.}
\end{figure}
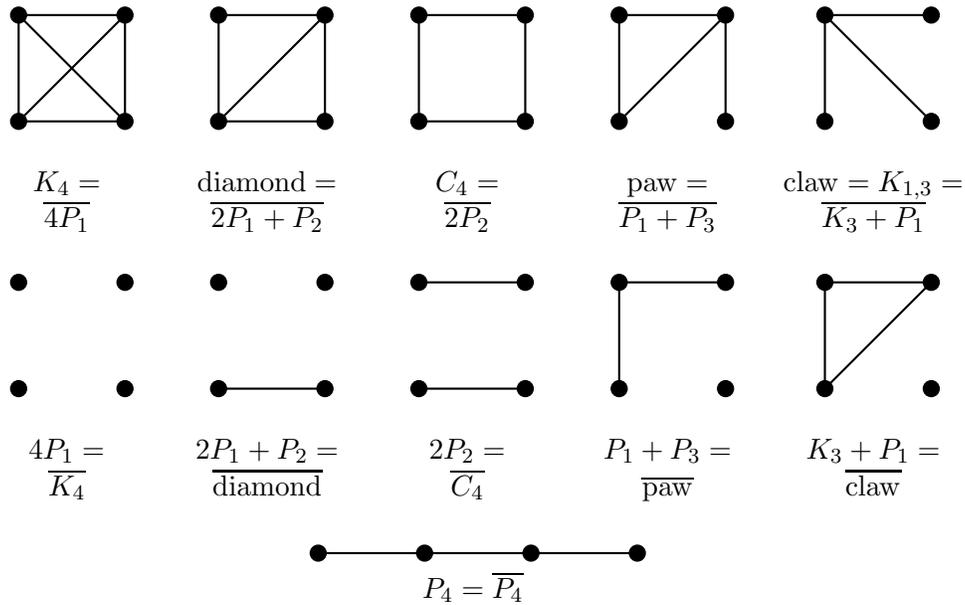

Brandst{\"a}dt, Engelfriet, Le and Lozin~\cite{BELL06} considered all sets~${\cal H}$ of graphs on at most four vertices and determined for which such sets~${\cal H}$ the class of ${\cal H}$-free graphs has bounded clique-width.
They proved the following dichotomy for sets ${\cal H}$ of \mbox{$4$-vertex} graphs and showed that all cases involving at least one graph with fewer than four vertices follow from known cases (see also Theorems~\ref{thm:single} and~\ref{thm:classification2}); the graphs in Theorem~\ref{t-bell} are displayed in \figurename~\ref{fig:4-vert}.

\begin{theorem}[\cite{BELL06}]\label{t-bell}
Let~${\cal H}$ be a set of $4$-vertex graphs.
The class of ${\cal H}$-free graphs has bounded clique-width if and only if~${\cal H}$ is \emph{not} a subset of any of the following sets:
\begin{enumerate}[(i)]
\item $\{C_4,2P_2\}$\\[-2em]
\item $\{K_4,2P_2\}$\\[-2em]
\item $\{C_4,4P_1\}$\\[-2em]
\item $\{K_4,\diamondgraph,C_4,\claw\}$\\[-2em]
\item $\{4P_1,2P_1+\nobreak P_2,2P_2,K_3+\nobreak P_1\}$\\[-2em]
\item $\{K_4,\diamondgraph,C_4,\paw,K_3+\nobreak P_1\}$, or\\[-2em]
\item $\{4P_1,2P_1+\nobreak P_2,2P_2,P_1+\nobreak P_3,\claw\}$.
\end{enumerate}
\end{theorem}

\subsection{Considering Hereditary Graph Classes Closed Under Complementation}\label{s-complement}

Recall that subgraph complementation preserves boundedness of clique-width by Fact~\ref{fact:comp}.
It is therefore natural to consider hereditary classes of graphs~${\cal G}$ that are closed under complementation.
In this section we survey the known results for these graph classes.
Recall that by Observation~\ref{o-comp} a hereditary graph class~${\cal G}$ is closed under complementation if and only if ${\cal H}={\cal F}_{\cal G}$ is closed under complementation.
We start by considering the cases where~$|{\cal H}|$ is small.

The only two non-empty self-complementary induced subgraphs of~$P_4$ are~$P_1$ and~$P_4$.
Hence, from Theorem~\ref{thm:single} it follows that the only self-complementary graphs~$H$ for which the class of $H$-free graphs has bounded clique-width are $H=P_1$ and $H=P_4$.
This result settles the $|{\cal H}|=1$ case and was generalized as follows.

\begin{theorem}[\cite{BDJLPZ17}]\label{thm:selfcomp-1}
For any set~${\cal H}$ of non-empty self-complementary graphs, the class of ${\cal H}$-free graphs has bounded clique-width if and only if either $P_1 \in {\cal H}$ or $P_4 \in {\cal H}$.
\end{theorem}

We now discuss the $|{\cal H}|=2$ case.
By Theorem~\ref{thm:selfcomp-1}, it remains to consider the case when ${\cal H} =\{H_1,H_2\}$ with $H_2=\overline{H_1}$ and~$H_1$ is not self-complementary.
This leads to the following classification, which also follows from Theorem~\ref{thm:classification2}.
The graphs in this classification are displayed in \figurename~\ref{fig:H-co-H-bdd-cw}.

\begin{theorem}[\cite{BDJLPZ17}]\label{thm:selfcomp-2}
For a graph~$H$, the class of $(H,\overline{H})$-free graphs has bounded clique-width if and only if~$H$ or~$\overline{H}$ is an induced subgraph of~$K_{1,3}$, $P_1+\nobreak P_4$, $2P_1+\nobreak P_3$ or~$sP_1$ for some $s\geq 1$.
\end{theorem}

\begin{figure}
\begin{center}
\begin{tabular}{cccc}
\begin{minipage}{0.2\textwidth}
\centering
\scalebox{1.0}{
{\begin{tikzpicture}[scale=1,rotate=90]
\GraphInit[vstyle=Simple]
\SetVertexSimple[MinSize=6pt]
\Vertices{circle}{a,b,c}
\Vertex[x=0,y=0]{d}
\Edges(a,d,b)
\Edges(c,d)
\end{tikzpicture}}}
\end{minipage}
&
\begin{minipage}{0.2\textwidth}
\centering
\scalebox{1.0}{
{\begin{tikzpicture}[scale=1,rotate=90]
\GraphInit[vstyle=Simple]
\SetVertexSimple[MinSize=6pt]
\Vertices{circle}{a,b,c}
\Vertex[x=0,y=0]{d}
\Edges(a,b,c,a)
\end{tikzpicture}}}
\end{minipage}
&
\begin{minipage}{0.2\textwidth}
\centering
\scalebox{1.0}{
{\begin{tikzpicture}[scale=1,rotate=90]
\GraphInit[vstyle=Simple]
\SetVertexSimple[MinSize=6pt]
\Vertices{circle}{a,b,c,d,e}
\Edges(b,c,d,e)
\end{tikzpicture}}}
\end{minipage}
&
\begin{minipage}{0.2\textwidth}
\centering
\scalebox{1.0}{
{\begin{tikzpicture}[scale=1,rotate=90]
\GraphInit[vstyle=Simple]
\SetVertexSimple[MinSize=6pt]
\Vertices{circle}{a,b,c,d,e}
\Edges(b,c,d,e)
\Edges(b,a,c)
\Edges(d,a,e)
\end{tikzpicture}}}
\end{minipage}\\
\\
$\claw=K_{1,3}$ &
$\overline{K_{1,3}}=K_3+\nobreak P_1$ &
$P_1+\nobreak P_4$ &
$\gem=\overline{P_1+P_4}$\\
\\
\begin{minipage}{0.2\textwidth}
\centering
\scalebox{1.0}{
{\begin{tikzpicture}[scale=1,rotate=90]
\GraphInit[vstyle=Simple]
\SetVertexSimple[MinSize=6pt]
\Vertices{circle}{a,b,c,d,e}
\Edges(e,a,b)
\end{tikzpicture}}}
\end{minipage}
&
\begin{minipage}{0.2\textwidth}
\centering
\scalebox{1.0}{
{\begin{tikzpicture}[scale=1,rotate=90]
\GraphInit[vstyle=Simple]
\SetVertexSimple[MinSize=6pt]
\Vertices{circle}{a,b,c,d,e}
\Edges(e,a,b)
\Edges(b,c,d,e,b)
\Edges(b,d)
\Edges(c,e)
\end{tikzpicture}}}
\end{minipage}
&
\begin{minipage}{0.2\textwidth}
\centering
\scalebox{1.0}{
{\begin{tikzpicture}[scale=1,rotate=90]
\GraphInit[vstyle=Simple]
\SetVertexSimple[MinSize=6pt]
\Vertices{circle}{a,b,c,d,e}
\end{tikzpicture}}}
\end{minipage}
&
\begin{minipage}{0.2\textwidth}
\centering
\scalebox{1.0}{
{\begin{tikzpicture}[scale=1,rotate=90]
\GraphInit[vstyle=Simple]
\SetVertexSimple[MinSize=6pt]
\Vertices{circle}{a,b,c,d,e}
\Edges(a,b,c,d,e,a,c,e,b,d,a)
\end{tikzpicture}}}
\end{minipage}\\
\\
$2P_1+\nobreak P_3$ &
$\overline{2P_1+\nobreak P_3}$ &
$sP_1$~($s=\nobreak 5$~shown)&
$\overline{sP_1}$~($s=\nobreak 5$~shown)\\
\end{tabular}
\end{center}
\caption{\label{fig:H-co-H-bdd-cw}Graphs~$H$ for which the clique-width of $(H,\overline{H})$-free graphs is bounded.}
\end{figure}
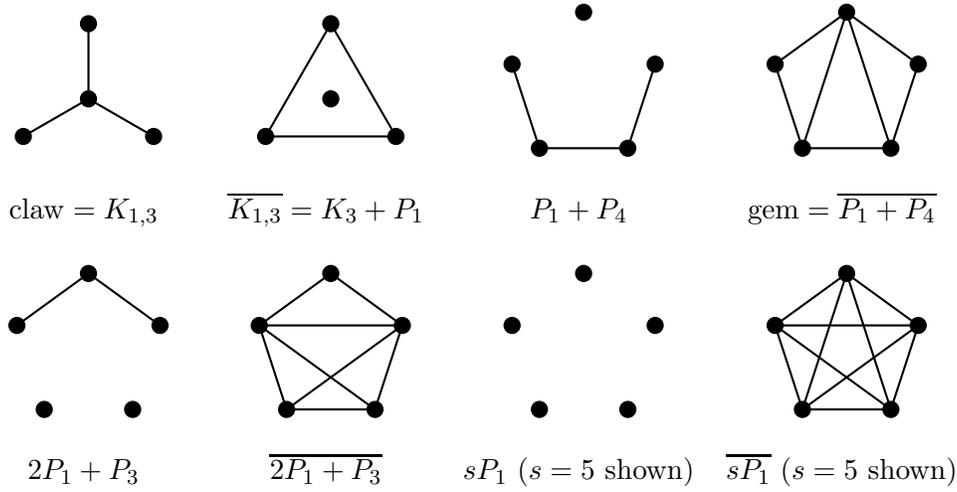

As we will see, the $|{\cal H}|=3$ case has not yet been fully settled.
Up to permutations of the graphs $H_1$, $H_2$,~$H_3$, a class of $(H_1,H_2,H_3)$-free graphs is closed under complementation if and only if~$H_i$ is self-complementary for all $i \in \{1,2,3\}$, or $H_1=\overline{H_2}$ and~$H_3$ is self-complementary (note that we may assume that~${\cal H}$ is minimal).
By Theorem~\ref{thm:selfcomp-1}, we only need to consider the second case.
By Theorem~\ref{thm:single}, we may exclude the case when $H_3=P_1$ or $H_3=P_4$.
The next two smallest self-complementary graphs~$H_3$ are the~$C_5$ and the bull.

Blanch\'e, Dabrowski, Johnson, Lozin, Paulusma and Zamaraev~\cite{BDJLPZ17} proved that the classification of boundedness of clique-width for $(H,\overline{H},C_5)$-free graphs coincides with the one of Theorem~\ref{thm:selfcomp-2}.
This raised the question of whether the same is true for other sets of self-complementary graphs ${\cal F}\neq \{C_5\}$.
However, the bull is self-complementary, and if~${\cal F}$ contains the bull, then the answer is negative, which can be seen as follows.
By Theorem~\ref{thm:selfcomp-2}, both the class of $(S_{1,1,2},\overline{S_{1,1,2}})$-free graphs and the class of $(2P_2,C_4)$-free graphs have unbounded clique-width.
In contrast, by Theorem~\ref{thm:1-vert-ext-of-P4}, both the class of $(S_{1,1,2},\overline{S_{1,1,2}},\mbox{bull})$-free graphs and even the class of $(P_5,\overline{P_5},\mbox{bull})$-free graphs have bounded clique-width.
However, as shown in the next theorem, the bull turned out to be the {\em only} exception if we exclude the ``trivial'' cases $H_3=P_1$ and $H_3=P_4$, which are the only non-empty self-complementary graphs on fewer than five vertices.

\begin{theorem}[\cite{BDJLPZ17}]\label{thm:selfcomp-3}
Let~${\cal F}$ be a set of self-complementary graphs on at least five vertices not equal to the bull.
For a graph~$H$, the class of $(\{H,\overline{H}\} \cup {\cal F})$-free graphs has bounded clique-width if and only if~$H$ or~$\overline{H}$ is an induced subgraph of~$K_{1,3}$, $P_1+\nobreak P_4$, $2P_1+\nobreak P_3$ or~$sP_1$ for some $s\geq 1$.
\end{theorem}

By Theorems~\ref{thm:selfcomp-1} and~\ref{thm:selfcomp-3} the case $|{\cal H}|=3$ is settled except when $H_1=\overline{H_2}$ and~$H_3$ is the bull; see also~\cite{BDJLPZ17}.

\begin{oproblem}\label{o-bull}
For which graphs~$H$ does the class of $(H,\overline{H},\bull)$-free graphs have bounded clique-width?
\end{oproblem}

In light of Theorem~\ref{thm:selfcomp-3}, Open Problem~\ref{o-bull} can also be extended to sets~${\cal F}$ of self-complementary graphs containing the bull.

\subsection{Forbidding with Respect to Other Graph Containment Relations}\label{s-containment}

In this section we survey results on (un)boundedness of clique-width for hereditary graph classes that can alternatively be characterized by some other graph containment relation.
In particular, when we forbid a finite collection of either subgraphs, minors or topological minors, it is possible to completely characterize those graph classes that have bounded clique-width.

\begin{theorem}[\cite{DP16,KLM09}]\label{t-finite}
Let $\{H_1,\ldots,H_p\}$ be a finite set of graphs.
Then the following statements hold:\\[-1.5em]
\begin{enumerate}[(i)]
\renewcommand{\theenumi}{(\roman{enumi})}
\renewcommand{\labelenumi}{(\roman{enumi})}
\item\label{t-finite-subgraph}The class of $(H_1,\ldots,H_p)$-subgraph-free graphs has bounded clique-width if and only if $H_i\in {\cal S}$ for some $i \in \{1,\ldots,p\}$.\\[-1.5em]
\item\label{t-finite-minor}The class of $(H_1,\ldots,H_p)$-minor-free graphs has bounded clique-width if and only if~$H_i$ is planar for some $i \in \{1,\ldots,p\}$.\\[-1.5em]
\item\label{t-finite-top minor}The class of $(H_1,\ldots,H_p)$-topological-minor-free graphs has bounded clique-width if and only if~$H_i$ is planar and has maximum degree at most~$3$ for some $i \in \{1,\ldots,p\}$.
\end{enumerate}
\end{theorem}

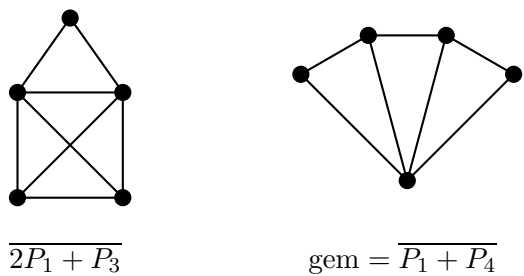
\begin{figure}
\centering
\begin{tabular}{cc}
\begin{minipage}{0.29\textwidth}
\centering
\scalebox{1}{
{\begin{tikzpicture}[scale=1.4]
\GraphInit[vstyle=Simple]
\SetVertexSimple[MinSize=6pt]
\Vertex[x=0.5,y=0.70710678118]{x}
\Vertex[x=0,y=0]{a}
\Vertex[x=1,y=0]{b}
\Vertex[x=0,y=-1]{c}
\Vertex[x=1,y=-1]{d}
\Edges(d,a,x,b,c)
\Edges(a,b,d,c,a)
\end{tikzpicture}}}
\end{minipage}
&
\begin{minipage}{0.29\textwidth}
\centering
\scalebox{1}{
{\begin{tikzpicture}[scale=2,rotate=90]
\GraphInit[vstyle=Simple]
\SetVertexSimple[MinSize=6pt]
\Vertex[x=0,y=0]{x}
\Vertex[a=-45,d=1]{a}
\Vertex[a=-15,d=1]{b}
\Vertex[a=15,d=1]{c}
\Vertex[a=45,d=1]{d}
\Edges(a,x,b)
\Edges(c,x,d)
\Edges(a,b,c,d)
\end{tikzpicture}}}
\end{minipage}\\\\
$\overline{2P_1+P_3}$ & $\gem=\overline{P_1+P_4}$
\end{tabular}
\caption{\label{fig:gem-crossed-house}The graphs~$H$ such that the class of $H$-induced-minor-free graphs has bounded clique-width.}
\end{figure}

The graph classes in Theorem~\ref{t-finite} have in common that the corresponding containment relation allows edge deletions.
If edge deletions are not permitted, then the situation becomes less clear, as we already saw for the induced subgraph relation.
This is also true for the induced minor relation, for which only the following (non-trivial) result is known.
We refer to \figurename~\ref{fig:gem-crossed-house} for a picture of the graphs $\overline{2P_1+P_3}$ and $\overline{P_1+P_4}$ (recall that the latter graph is also known as the $\gem$).

\begin{theorem}[\cite{BOS15}]\label{thm:cw-ind-min}
Let~$H$ be a graph.
The class of $H$-induced-minor-free graphs has bounded clique-width if and only if $H \ssi \overline{2P_1+P_3}$ or $H \ssi \overline{P_1+P_4}$.
\end{theorem}

With an eye on Theorem~\ref{thm:classification2}, Theorem~\ref{thm:cw-ind-min} leads to the following open problem.

\begin{oproblem}\label{o-im}
Determine for which pairs of graphs $(H_1,H_2)$ the class of $(H_1,H_2)$-induced-minor-free graphs has bounded clique-width.
\end{oproblem}

We end this section with two more open problems; we note that a class of $\mbox{$H$-contraction}$-free graphs need not be hereditary and that Open Problem~\ref{o-contraction} is trivial if we allow disconnected graphs, since edge contractions preserve the number of components in a graph.

\begin{sloppypar}
\begin{oproblem}\label{o-contraction}
Determine for which graphs~$H$ the class of connected $H$-contraction-free graphs has bounded clique-width.
\end{oproblem}
\end{sloppypar}

\begin{sloppypar}
\begin{oproblem}\label{o-itm}
Determine for which graphs~$H$ the class of $H$-induced-topological-minor-free graphs has bounded clique-width.
\end{oproblem}
\end{sloppypar}

\section{Algorithmic Consequences}\label{s-algo}

In this section we illustrate how bounding clique-width (or one of its equivalent parameters) can be used to find polynomial-time algorithms to solve problems on special graph classes, even when these problems are \NP-hard on general graphs.
In Section~\ref{s-meta-th} we discuss meta-theorems, and in Section~\ref{s-meta} we show how they can be used as part of a general strategy for solving problems.
In Section~\ref{s-atoms} we focus on atoms, which are often used as a specific ingredient for the general strategy.
Finally, in Sections~\ref{s-gc} and~\ref{s-gi} we look at two problems in particular, namely {\sc Colouring} and {\sc Graph Isomorphism}, respectively.
For other graph problems where boundedness of clique-width is used to classify their computational complexity on hereditary graph classes, see, for example,~\cite{BGM,CKL11}. 
We refer to~\cite{BGP13,BK17,FGLS10,FGLS14,GLSZ18} for parameterized complexity results on clique-width.

\subsection{Meta-Theorems}\label{s-meta-th}

We observed in Section~\ref{s-intro} that one of the advantages of showing that a graph class has bounded clique-width is that one can apply meta-theorems that say that any problem definable within certain constraints can be solved in polynomial time on the class.
We mentioned such a theorem concerning any problem that can be defined in MSO$_1$~\cite{CMR00}.
The result of~\cite{CMR00} has been extended by others to address graph problems that cannot be defined in MSO$_1$.
An important example of such a problem is the {\sc ${\cal F}$-Partition} problem, which asks, for a graph~$G$ and an integer~$k$, whether~$V(G)$ can be partitioned into (possibly empty) sets $V_1,\ldots,V_k$ such that every~$V_i$ induces a graph in~${\cal F}$.
In particular, if~${\cal F}$ consists of the edgeless graphs, then the {\sc ${\cal F}$-Partition} problem is equivalent to the {\sc Colouring} problem.

Espelage, Gurski and Wanke~\cite{EGW01} gave a general method to show that on graphs of bounded clique-width, {\sc ${\cal F}$-Partition} is polynomial-time solvable for a number of graph classes~${\cal F}$ including complete graphs, edgeless graphs, forests and triangles.
Their method can also be applied to other problems, such as {\sc Hamilton Cycle} (see also~\cite{Wa94} and see~\cite{BKK17} for a faster algorithm) and {\sc Cubic Subgraph}.
Later, Kobler and Rotics~\cite{KR03} proved that a variety of other \NP-complete graph partition problems (where either the set of vertices or the set of edges is partitioned) can be solved in polynomial time for graphs of bounded clique-width.
Again, their set of problems includes {\sc Colouring} 
(see~\cite{La18} for the fastest known algorithm, parameterized by clique-width, for finding a $k$-colouring if~$k$ is constant).
However, their work also captures other graph partition problems, such as {\sc List $k$-Colouring} and {\sc Edge-Dominating Set}.

\begin{sloppypar}
Gerber and Kobler~\cite{GK03} gave a framework of vertex partition problems with respect to a fixed interval degree constraint matrix.
They showed that these problems, which include {\sc Induced Bounded Degree Subgraph}, {\sc Induced $k$-Regular Subgraph}, $H$-{\sc Colouring} and $H$-{\sc Covering}, are all solvable in polynomial time on graphs of bounded clique-width.
In the same paper, they extended their framework to include more general problems, such as {\sc Satisfactory Graph Partitioning} and {\sc Majority Domination Number}.
Rao~\cite{Ra07} gave another family of vertex partitioning problems that can be solved in polynomial time for graphs of bounded clique-width.
Besides {\sc Colouring}, this family also includes {\sc Domatic Number}, {\sc Hamilton Cycle} and {\sc ${\cal F}$-Partition} where~${\cal F}$ consists of complete and edgeless graphs; perfect graphs; or $H$-free graphs, for an arbitrary fixed graph~$H$.
\end{sloppypar}

The algorithms in~\cite{CMR00,EGW01,GK03,GS15,KR03} all require a $c$-expression of the input graph~$G$ for some constant~$c$.
Recall that computing the clique-width of a graph is \NP-hard~\cite{FRRS09} (and that the complexity of deciding whether a graph has clique-width at most~$c$ is still open for every constant~$c\geq 4$).
Of course, this suggests we cannot hope to compute a $\cw(G)$-expression in polynomial time.
However, it is sufficient to use the algorithm of Seymour and Oum~\cite{OS06}, which returns a $c$-expression for some $c\leq 2^{3\cw(G)+2}-1$ in $O(n^9\log n)$ time, or the later improvements of Oum~\cite{Oum08} and Hlin\v{e}n\'y and Oum~\cite{HO08} that provide cubic-time algorithms which yield a $c$-expression for some $c\leq 8^{\cw(G)}-1$ and $c\leq 2^{\cw{G}+1}-1$, respectively.

We note that there exist problems that are polynomial-time solvable for graphs of clique-width~$c$, but \NP-complete for graphs of clique-width~$d$ for constants~$c$ and~$d$ with $c<d$.
For example, this holds for the {\sc Disjoint Paths} problem, which is linear-time solvable for graphs of clique-width at most~$2$, but \NP-complete for graphs of clique-width at most~$6$~\cite{GW06}.

\subsection{A General Strategy for Finding Algorithms}\label{s-meta}

Below we describe an approach that has often been used as a general strategy when we want to solve a problem~$\Pi$ on a graph class~${\cal G}$.
We suppose that there exists some meta-algorithm~{\tt A} that can be used to solve~$\Pi$ on classes of bounded clique-width.
We say that the graph class~${\cal G}$ is \defword{reducible} to some subclass ${\cal G'}\subseteq {\cal G}$ with respect to~$\Pi$ if the following holds: if~$\Pi$ can be solved in polynomial time on~${\cal G}'$, then~$\Pi$ can also be solved in polynomial time on~${\cal G}$.
We can now state the following general approach.

\medskip
\noindent \rule{\textwidth}{0.1mm}\\
{\tt Clique-Width Method}\\[-1.5em]
\begin{enumerate}[1a.]
\item [1.] Check if~${\cal G}$ has bounded clique-width (for instance, by using the {\tt BCW Method}).\\[-1.5em]
\item [2.] If so, then apply {\tt A}.
Otherwise choose between 3a and 3b.\\[-1.5em]
\item [3a.] Reduce~${\cal G}$ to some subclass~${\cal G}'$ of bounded clique-width and apply~{\tt A}.
\item [3b.] Partition~${\cal G}$ into two classes~${\cal G}_1$ and~${\cal G}_2$, such that~${\cal G}_1$ has bounded clique-width and is as large as possible.
Apply {\tt A} to solve~$\Pi$ on~${\cal G}_1$.
Use some problem-specific algorithm to solve~$\Pi$ on~${\cal G}_2$.\\[-2.5em]
\end{enumerate}
\noindent \rule{\textwidth}{0.1mm}

\medskip
\noindent
To give an example where Step~3a of this method is used, we can let~${\cal G}'$ be the class that consists of all atoms in~${\cal G}$.
Recall that a connected graph is an atom if it has no clique cut-set. 
Dirac~\cite{Di61} introduced the notion of a clique cut-set and proved that every chordal graph is either complete or has a clique cut-set.
As complete graphs have clique-width~$2$, this means that chordal graphs that are atoms have clique-width at most~$2$, whereas the class of chordal graphs has unbounded clique-width (see, for example, Theorem~\ref{thm:chordal-classification}). 
Over the years, decomposition into atoms has become a widely used tool for solving decision problems on hereditary graph classes.
For instance, a classical result of Tarjan~\cite{Tarjan85} implies that {\sc Colouring} and other problems, such as those of determining the size of a largest independent set ({\sc Independent Set}) or a largest clique ({\sc Clique}), are polynomial-time solvable on a hereditary graph class~${\cal G}$ if and only if they are polynomial-time solvable on the atoms of~${\cal G}$.
We will discuss atoms in more detail in Section~\ref{s-atoms}.

To give an example where Step~3b of this method is used, Fraser, Hamel, Ho\`ang, Holmes and LaMantia~\cite{FHHHL17} proved that {\sc Colouring} can be solved in polynomial time for $(C_4,C_5,4P_1)$-free graphs by proving that the non-perfect graphs from this class have bounded clique-width and by recalling that {\sc Colouring} can be solved in polynomial time on perfect graphs~\cite{GLS84}.

\subsection{Atoms}\label{s-atoms}
As mentioned, atoms are an important example for Step 3a in the {\tt Clique-Width Method}.
To determine new polynomial-time results for {\sc Colouring}, Gaspers, Huang and Paulusma~\cite{GHP18} investigated whether there exist graph classes of unbounded clique-width whose atoms have bounded clique-width.
They found that this is not the case for the classes of $H$-free graphs.
That is, the classification for $H$-free atoms coincides with the classification for $H$-free graphs in Theorem~\ref{thm:single}.

\begin{theorem}[\cite{GHP18}]\label{t-atoms}
Let~$H$ be a graph.
The class of $H$-free atoms has bounded clique-width if and only if~$H$ is an induced subgraph of~$P_4$.
\end{theorem}

As split graphs are chordal by Observation~\ref{o-inclusion}, it follows that split atoms (split graphs that are atoms) are complete graphs, and thus have clique-width at most~$2$, whereas the class of general split graphs has unbounded clique-width~\cite{MR99}.
As the class of split graphs coincides with the class of $(C_4,C_5,2P_2)$-free graphs~\cite{FH77}, Gaspers, Huang and Paulusma~\cite{GHP18} asked whether there exists a class of $(H_1,H_2)$-free graphs of unbounded clique-width whose atoms form a class of bounded clique-width.
They proved that this is indeed the case by showing a constant bound on the clique-width of atoms in the class of $(C_4,P_6)$-free graphs, which form a superclass of split graphs (they used this to prove that {\sc Colouring} is polynomial-time solvable for $(C_4,P_6)$-free graphs).

\begin{theorem}[\cite{GHP18}]\label{t-atoms2}
Every $(C_4,P_6)$-free atom has clique-width at most~$18$.
\end{theorem}

We are not aware of any other examples, which leads us to ask the following open problem (see also~\cite{GHP18}).

\begin{oproblem}\label{o-atoms}
Determine all pairs of graphs $H_1,H_2$ such that the class of $(H_1,H_2)$-free graphs has unbounded clique-width, but the class of $(H_1,H_2)$-free atoms has bounded clique-width.
\end{oproblem}

Recall from Open Problem~\ref{oprob:twographs} that there are still five non-equivalent pairs $H_1,H_2$ for which we do not know whether the clique-width of $(H_1,H_2)$-free graphs is bounded or unbounded.
Due to the algorithmic implications mentioned above, the following problem is therefore also of interest.

\begin{oproblem}\label{oprob:twographs2}
Does the class of $(H_1,H_2)$-free atoms have bounded clique-width when:
\begin{enumerate}[(i)]
\item \label{oprob:twographs:3P_1b} $H_1=K_3$ and $H_2 \in \{P_1+\nobreak S_{1,1,3},
\allowbreak S_{1,2,3}\}$\\[-1.5em]
\item\label{oprob:twographs:2P_1+P_2b} $H_1=\diamondgraph$ and $H_2 \in \{P_1+\nobreak P_2+\nobreak P_3,\allowbreak P_1+\nobreak P_5\}$\\[-1.5em]
\item \label{oprob:twographs:P_1+P_4b} $H_1=\gem$ and $H_2=P_2+\nobreak P_3$.
\end{enumerate}
\end{oproblem}

\subsection{Graph Colouring}\label{s-gc}

Kr\'al', Kratochv\'{\i}l, Tuza, and Woeginger completely classified the complexity of {\sc Colouring} for $H$-free graphs.

\begin{theorem}[\cite{KKTW01}]\label{t-dicho}
Let~$H$ be a graph.
If $H\ssi P_4$ or $H\ssi P_1+\nobreak P_3$, then {\sc Colouring} restricted to $H$-free graphs is polynomial-time solvable, otherwise it is \NP-complete.
\end{theorem}

For $(H_1,H_2)$-free graphs, the classification of {\sc Colouring} is open for many pairs of graphs~$H_1$,~$H_2$.
A summary of the known results can be found in~\cite{GJPS17}, but several other results have since appeared~\cite{BDJP16,CH,DP18,GHP18,KM18,KMP,ML17}; see~\cite{DP18} for further details.
In relation to boundedness of clique-width, the following is of importance.
There still exist ten classes of $(H_1,H_2)$-free graphs, for which {\sc Colouring} could potentially be solved in polynomial time by showing that their clique-width is bounded.
That is, for these classes, the complexity of {\sc Colouring} is not resolved, and it is not known whether the clique-width is bounded.
This list is obtained by updating the list of~\cite{DDP17}, which contains 13 cases, with the result of~\cite{BDJP16} for $(H_1,H_2)=(2P_1+\nobreak P_3,\overline{2P_1+P_3})$ and the results of~\cite{BDJP18} for $(H_1,H_2)=(\gem,P_1+\nobreak 2P_2)$ and $(H_1,H_2)=(P_1+\nobreak P_4,\overline{P_1+2P_2})$.

\begin{oproblem}\label{oprob:twographs3}
Can the {\sc Colouring} problem be solved in polynomial time on $(H_1,H_2)$-free graphs when:
\begin{enumerate}[(i)]
\item $H_1\in \{K_3,\paw\}$ and $H_2\in \{P_1+\nobreak S_{1,1,3},S_{1,2,3}\}$\\[-1.5em]
\item $H_1=2P_1+\nobreak P_2$ and $H_2 \in \{\overline{P_1+\nobreak P_2+\nobreak P_3}, \allowbreak \overline{P_1+\nobreak P_5}\}$\\[-1.5em]
\item $H_1=\diamondgraph$ and $H_2 \in \{P_1+\nobreak P_2+\nobreak P_3,\allowbreak P_1+\nobreak P_5\}$\\[-1.5em]
\item $H_1=P_1+\nobreak P_4$ and $H_2=\overline{P_2+P_3}$\\[-1.5em]
\item $H_1=\gem$ and $H_2=P_2+\nobreak P_3$.
\end{enumerate}
\end{oproblem}

\subsection{Graph Isomorphism}\label{s-gi}

Grohe and Schweitzer~\cite{GS15} proved that {\sc Graph Isomorphism} is polynomial-time solvable for graphs of bounded clique-width.
Hence, identifying graph classes of bounded clique-width is of importance for the {\sc Graph Isomorphism} problem.

The classification for the computational complexity of {\sc Graph Isomorphism} for $H$-free graphs can be found in a technical report of Booth and Colbourn~\cite{BC79}, who credited the result to an unpublished manuscript of Colbourn and Colbourn.
Another proof of this result appears in a paper of Kratsch and Schweitzer~\cite{KS12}.

\begin{theorem}[\cite{BC79}]\label{t-gip4}
Let~$H$ be a graph. 
If $H\ssi P_4$, then {\sc Graph Isomorphism} for $H$-free graphs can be solved in polynomial time, otherwise it is \GI-complete.
\end{theorem}

Note that {\sc Graph Isomorphism} is polynomial-time solvable even for the class of permutation graphs~\cite{Co81}, which contains the class of $P_4$-free graphs.

Schweitzer~\cite{Sc17} observed great similarities between the techniques used for classifying boundedness of clique-width and classifying the complexity of {\sc Graph Isomorphism} for hereditary graph classes.
He proved that {\sc Graph Isomorphism} is \GI-complete for any graph class~${\cal G}$ that allows a so-called simple path encoding and also showed that every such graph class~${\cal G}$ has unbounded clique-width.
Indeed, the {\tt UCW Method} relies on some clique-width-boundedness-preserving transformations of an arbitrary graph from some known graph class~${\cal G'}$ of unbounded clique-width, such as the class of walls, to a graph of the unknown class~${\cal G}$.
One way to do this is to show that the graphs in~${\cal G}$ contain a simple path encoding of graphs from~${\cal G'}$.

Kratsch and Schweitzer~\cite{KS12} initiated a complexity classification for {\sc Graph Isomorphism} for $(H_1, H_2)$-free graphs.
Schweitzer~\cite{Sc17} extended the results of~\cite{KS12} and proved that the number of unknown cases is finite, but did not explicitly list what these cases were.
As mentioned earlier, {\sc Graph Isomorphism} is polynomial-time solvable for graphs of bounded clique-width~\cite{GS15}.
Bonamy, Dabrowski, Johnson and Paulusma~\cite{BDJP18} therefore combined the known results for boundedness of clique-width for bigenic classes (Theorem~\ref{thm:classification2}) with the results of~\cite{KS12} and~\cite{Sc17} to obtain an explicit list of only~14 cases, for which the complexity of {\sc Graph Isomorphism} was unknown.
In the same paper they reduced this number to~$7$ and gave the following state-of-the-art summary; recall that $K_{1,t}^+$ and~$K_{1,t}^{++}$ are the graphs obtained from~$K_{1,t}$ by subdividing one edge once or twice, respectively.

\begin{theorem}[\cite{BDJP18}]\label{thm:gi-classification2}
For a class~${\cal G}$ of graphs defined by two forbidden induced subgraphs, the following holds:
\begin{enumerate}
\item {\sc Graph Isomorphism} is solvable in polynomial time on~${\cal G}$ if~${\cal G}$ is equivalent\footnote{Equivalence is defined in the same way as for clique-width (see Footnote~\ref{footnote:equivalence}).
If two classes are equivalent, then the complexity of {\sc Graph Isomorphism} is the same on both of them.~\cite{BDJP18}.} to a class of $(H_1,H_2)$-free graphs such that one of the following holds:
\begin{enumerate}[(i)]
\renewcommand{\theenumii}{(\roman{enumii})}
\renewcommand{\labelenumii}{(\roman{enumii})}
\item \label{thm-old:gi-classification2:poly:P4}$H_1$ or $H_2 \ssi P_4$\\[-1.5em]
\item \label{thm-old:gi-classification2:poly:K1t+P1-co-K1t+P1}$\overline{H_1}$ and $H_2 \ssi K_{1,t}+\nobreak P_1$ for some~$t \geq 1$\\[-1.5em]
\item \label{thm-old:gi-classification2:poly:P3+tP1-coP3+tP1}$\overline{H_1}$ and $H_2 \ssi tP_1+\nobreak P_3$ for some~$t \geq 1$\\[-1.5em]
\item \label{thm-old:gi-classification2:poly:Kt}$H_1 \ssi K_t$ and $H_2 \ssi 2K_{1,t}, K_{1,t}^+$ or~$P_5$ for some $t \geq 1$\\[-1.5em]
\item \label{thm-old:gi-classification2:poly:paw}$H_1 \ssi \paw$ and $H_2 \ssi P_2+\nobreak P_4, P_6, S_{1,2,2}$ or~$K_{1,t}^{++}+\nobreak P_1$ for some~$t \geq 1$\\[-1.5em]
\item \label{thm-old:gi-classification2:poly:diamond}$H_1 \ssi \diamondgraph$ and $H_2 \ssi P_1+\nobreak 2P_2$\\[-1.5em]
\item \label{thm-old:gi-classification2:poly:gem}$H_1 \ssi \gem$ and $H_2 \ssi P_1+\nobreak P_4$ or~$P_5$ or\\[-1.5em]
\item \label{thm-old:gi-classification2:poly:crossed-house}$H_1 \ssi \overline{2P_1+P_3}$ and $H_2 \ssi P_2+\nobreak P_3$.
\end{enumerate}
\item {\sc Graph Isomorphism} is \GI-complete on~${\cal G}$ if~${\cal G}$ is equivalent to a class of $(H_1,H_2)$-free graphs such that one of the following holds:
\begin{enumerate}[(i)]
\renewcommand{\theenumii}{(\roman{enumii})}
\renewcommand{\labelenumii}{(\roman{enumii})}
\item \label{thm-old:gi-classification2:hard:not-path-star-forest}neither~$H_1$ nor~$H_2$ is a path star forest\\[-1.5em]
\item \label{thm-old:gi-classification2:hard:not-co-path-star-forest}neither~$\overline{H_1}$ nor~$\overline{H_2}$ is a path star forest\\[-1.5em]
\item \label{thm-old:gi-classification2:hard:K3}$H_1\si K_3$ and $H_2\si 2P_1+\nobreak 2P_2,P_1+\nobreak 2P_3,2P_1+\nobreak P_4$ or~$3P_2$\\[-1.5em]
\item \label{thm-old:gi-classification2:hard:K4}$H_1\si K_4$ and $H_2\si K_{1,4}^{++}, P_1+\nobreak 2P_2$ or~$P_1+\nobreak P_4$\\[-1.5em]
\item \label{thm-old:gi-classification2:hard:K5}$H_1\si K_5$ and $H_2\si K_{1,3}^{++}$\\[-1.5em]
\item \label{thm-old:gi-classification2:hard:2P2}$H_1 \si C_4$ and $H_2 \si K_{1,3}, 3P_1+\nobreak P_2$ or~$2P_2$\\[-1.5em]
\item \label{thm-old:gi-classification2:hard:2P1+P2}$H_1\si \diamondgraph$ and $H_2 \si K_{1,3}, P_2+\nobreak P_4, 2P_3$ or~$P_6$ or\\[-1.5em]
\item \label{thm-old:gi-classification2:hard:gem}$H_1\si \overline{P_1+P_4}$ and $H_2\si P_1+\nobreak 2P_2$.
\end{enumerate}
\end{enumerate}
\end{theorem}

As shown in~\cite{BDJP18}, Theorem~\ref{thm:gi-classification2} leads to the following open problem.

\newpage
\begin{oproblem}\label{oprob:gi}
What is the complexity of {\sc Graph Isomorphism} on $(H_1,H_2)$-free graphs in the following seven cases?
\begin{enumerate}[(i)]
\item $H_1=K_3$ and $H_2 \in \{P_7,S_{1,2,3}\}$\\[-1.5em]
\item $H_1=K_4$ and $H_2=S_{1,1,3}$\\[-1.5em]
\item $H_1=\diamondgraph$ and $H_2 \in \{P_1+\nobreak P_2+\nobreak P_3,P_1+\nobreak P_5\}$\\[-1.5em]
\item $H_1=\gem$ and $H_2=P_2+\nobreak P_3$\\[-1.5em]
\item $H_1=\overline{2P_1+P_3}$ and $H_2=P_5$.
\end{enumerate}
\end{oproblem}

\begin{sloppypar}
For $H$-induced-minor-free graphs the classification for the complexity of {\sc Graph Isomorphism} is given in Theorem~\ref{thm:gi}.
Note that the second and third tractable cases follow from Theorem~\ref{thm:cw-ind-min} and the fact that {\sc Graph Isomorphism} is polynomial-time solvable on graphs of bounded clique-width~\cite{GS15}.
We refer to \figurename~\ref{fig:gem-crossed-house} for a picture of the graphs $\overline{2P_1+P_3}$ and $\overline{P_1+P_4}$.
\end{sloppypar}

\begin{theorem}[\cite{BOS15}] \label{thm:gi}
Let~$H$ be a graph.
The {\sc Graph Isomorphism} problem on $H$-induced-minor-free graphs is polynomial-time solvable if:
\begin{enumerate}[(i)]
\item $H$ is a complete graph, \\[-1.5em]
\item $H \ssi \overline{2P_1+P_3}$ or \\[-1.5em]
\item $H \ssi \overline{P_1+P_4}$
\end{enumerate}
and \GI-complete otherwise.
\end{theorem}

\section{Well-Quasi-Orderability}\label{s-wqo}

We recall that the Robertson-Seymour Theorem~\cite{RS04-Wagner} states that the set of all finite graphs is well-quasi-ordered by the minor relation.
This result, combined with the cubic-time algorithm of~\cite{RS95} for testing if a graph~$G$ contains some fixed graph~$H$ as a minor, gives a cubic-time algorithm for testing whether a graph belongs to some minor-closed graph class.
Other known results on well-quasi-orderability include a result of Ding~\cite{Di92}, which implies that every class of graphs with bounded vertex cover number is well-quasi-ordered by the induced subgraph relation and a result of Mader~\cite{Ma72}, who showed that every class of graphs with bounded feedback vertex number is well-quasi-ordered by the topological minor relation.
Fellows, Hermelin and Rosamund~\cite{FHR12} simplified the proofs of Ding and Mader.
They also showed that every class of graphs of bounded circumference is well-quasi-ordered by the induced minor relation.
As applications they gave linear-time algorithms for recognizing graphs from any topological-minor-closed graph class with bounded feedback vertex number; any induced-minor-closed graph class of bounded circumference; and any induced-subgraph-closed graph class with bounded vertex cover number.

The Robertson-Seymour Theorem also implies that there exist graph classes of unbounded clique-width that are well-quasi-ordered by the minor relation.
For hereditary graph classes, the notion of well-quasi-orderability by the induced subgraph relation is closely related to boundedness of clique-width, but the exact relationship between the two notions is not yet fully understood.
In this section we survey results on well-quasi-orderability by the induced subgraph relation for hereditary classes, together with some more results for other containment relations.

In Section~\ref{s-intro}, we noted that Daligault, Rao and Thomass{\'e~\cite{DRT10} asked if every hereditary graph class that is well-quasi-ordered by the induced subgraph relation has bounded clique-width.
Lozin, Razgon and Zamaraev~\cite{LRZ18} gave a negative answer to this question.
That is, they found an example of a hereditary graph class that is well-quasi-ordered by the induced subgraph relation but has unbounded clique-width.
As the hereditary graph class in their example is not finitely defined (that is, this graph class is defined by infinitely many forbidden induced subgraphs), they conjectured the following.

\begin{conj}[\cite{LRZ18}]\label{c-f}
If a finitely defined hereditary class of graphs~${\cal G}$ is well-quasi-ordered by the induced subgraph relation, then~${\cal G}$ has bounded clique-width.
\end{conj}

We note that the reverse implication of the statement in Conjecture~\ref{c-f} is not true.
We can take the (hereditary) class of graphs of maximum degree at most~$2$, which have clique-width at most~$4$ by Proposition~\ref{p-atmost2}.
However, the class of graphs of maximum degree at most~$2$ contains all cycles, which form an infinite anti-chain.
Furthermore, the class of graphs of maximum degree at most~$2$, is finitely defined: it is the class of $(\claw,\paw,\diamondgraph,K_4)$-free graphs.

\subsection{Well-Quasi-Orderability Preserving Operations}\label{s-labeled}

In order to prove that some class of graphs is well-quasi-ordered by the induced subgraph relation or not, we would like to use similar facts to those used to prove boundedness or unboundedness of clique-width.
This is not straightforward, as there is no analogue of Facts~\ref{fact:del-vert}--\ref{fact:subdiv} for well-quasi-orderability by the induced subgraph relation.
We show this in the three examples below, but first we recall that these facts concern, respectively, vertex deletion (Fact~\ref{fact:del-vert}), subgraph complementation (Fact~\ref{fact:comp}), bipartite complementation (Fact~\ref{fact:bip}), being prime (Fact~\ref{fact:prime}), being $2$-connected (Fact~\ref{fact:2-conn}), and edge subdivision for graphs of bounded maximum degree (Fact~\ref{fact:subdiv}).

\begin{example}
A counterexample for analogues of Facts~\ref{fact:del-vert}--\ref{fact:bip} is formed by the class of cycles~\cite{DLP18}: deleting a vertex of a cycle, complementing the subgraph induced by two adjacent vertices, or applying a bipartite complementation between two adjacent vertices yields a path.
The set of cycles is an infinite anti-chain with respect to the induced subgraph relation, but the set of paths is well-quasi-ordered.
\end{example}

\begin{example}\label{e-prime2con}
A counterexample for analogues of Facts~\ref{fact:prime}--\ref{fact:2-conn} is formed by the following class of graphs.
For $i\geq 1$, take a path of length~$i$ with end-vertices~$u$ and~$v$ and add vertices $u'$, $u''$, $v'$,~$v''$ with edges $uu'$, $uu''$, $vv'$ and~$vv''$.
Call the resulting graph~$H_i$ (see also \figurename~\ref{fig:H_i-graphs}) and let~${\cal H}$ be the class of graphs~$H_i$ (and their induced subgraphs).
If $i\neq j$, then~$H_i$ is not an induced subgraph of~$H_j$, which implies that~${\cal H}$ is not well-quasi-ordered by the induced subgraph relation.
However, the prime graphs of~${\cal H}$ are paths, which are well-quasi-ordered by the induced subgraph relation.
This shows that the analogue to Fact~\ref{fact:prime} does not hold for well-quasi-orderability by the induced subgraph relation.
The analogue to Fact~\ref{fact:2-conn} does not hold either, as~${\cal H}$ contains no $2$-connected graphs.
\end{example}

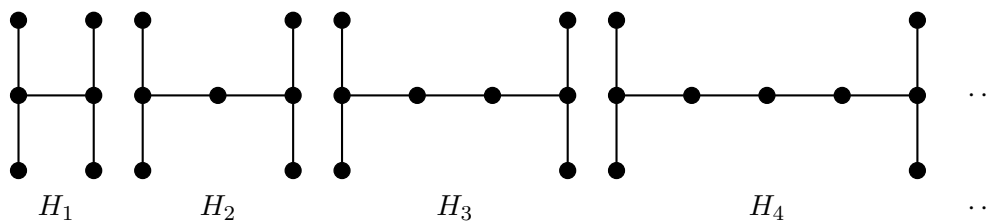
\begin{figure}
\begin{center}
\begin{tabular}{ccccc}
\begin{tikzpicture}[scale=1]
\GraphInit[vstyle=Simple]
\SetVertexSimple[MinSize=6pt]
\Vertex[x=0,y=-1]{a1}
\Vertex[x=0,y=0]{a2}
\Vertex[x=0,y=1]{a3}
\Vertex[x=1,y=-1]{c1}
\Vertex[x=1,y=0]{c2}
\Vertex[x=1,y=1]{c3}
\Edges(a1,a2,a3)
\Edges(c1,c2,c3)
\Edges(a2,c2)
\end{tikzpicture}
&
\begin{tikzpicture}[scale=1]
\GraphInit[vstyle=Simple]
\SetVertexSimple[MinSize=6pt]
\Vertex[x=0,y=-1]{a1}
\Vertex[x=0,y=0]{a2}
\Vertex[x=0,y=1]{a3}
\Vertex[x=1,y=0]{b1}
\Vertex[x=2,y=-1]{c1}
\Vertex[x=2,y=0]{c2}
\Vertex[x=2,y=1]{c3}
\Edges(a1,a2,a3)
\Edges(c1,c2,c3)
\Edges(a2,b1,c2)
\end{tikzpicture}
&
\begin{tikzpicture}[scale=1]
\GraphInit[vstyle=Simple]
\SetVertexSimple[MinSize=6pt]
\Vertex[x=0,y=-1]{a1}
\Vertex[x=0,y=0]{a2}
\Vertex[x=0,y=1]{a3}
\Vertex[x=1,y=0]{b1}
\Vertex[x=2,y=0]{b2}
\Vertex[x=3,y=-1]{c1}
\Vertex[x=3,y=0]{c2}
\Vertex[x=3,y=1]{c3}
\Edges(a1,a2,a3)
\Edges(c1,c2,c3)
\Edges(a2,b1,b2,c2)
\end{tikzpicture}
&
\begin{tikzpicture}[scale=1]
\GraphInit[vstyle=Simple]
\SetVertexSimple[MinSize=6pt]
\Vertex[x=0,y=-1]{a1}
\Vertex[x=0,y=0]{a2}
\Vertex[x=0,y=1]{a3}
\Vertex[x=1,y=0]{b1}
\Vertex[x=2,y=0]{b2}
\Vertex[x=3,y=0]{b3}
\Vertex[x=4,y=-1]{c1}
\Vertex[x=4,y=0]{c2}
\Vertex[x=4,y=1]{c3}
\Edges(a1,a2,a3)
\Edges(c1,c2,c3)
\Edges(a2,b1,b2,b3,c2)
\end{tikzpicture}
&
\begin{tikzpicture}[scale=1]
\node at (0,0) {$\cdots$};
\GraphInit[vstyle=Simple]
\SetVertexSimple[MinSize=6pt,LineColor=white,FillColor=white]
\Vertex[x=0,y=1]{blank1}
\Vertex[x=0,y=-1]{blank2}
\end{tikzpicture}\\
$H_1$ & $H_2$ & $H_3$ & $H_4$ & $\cdots$
\end{tabular}
\end{center}
\caption{\label{fig:H_i-graphs}The graphs~$H_i$ from Example~\ref{e-prime2con}.}
\end{figure}

\begin{example}
To obtain a counterexample for the analogue of Fact~\ref{fact:subdiv} we consider the class of graphs~${\cal H}_1$ consisting of the graph~$H_1$ from Example~\ref{e-prime2con} only.
This class is well-quasi-ordered by the induced subgraph relation.
However, we can obtain the class~${\cal H}$ in Example~\ref{e-prime2con}, which is not well-quasi-ordered, from~${\cal H}_1$ via edge subdivisions.
That is, for $i\geq 1$, the graph~$H_{i+1}$ is obtained from~$H_i$ by the subdivision of an edge of the path of length~$i$.
\end{example}

As these examples suggest, we need a stronger variant of well-quasi-orderability by the induced subgraph relation.
To define this variant, consider an arbitrary quasi-order $(W,\leq)$.
Then a graph~$G$ is a \defword{labelled} graph if each vertex~$v$ of~$G$ is equipped with a \defword{label}~$l_G(v)\in W$.
A graph~$F$ with labelling~$l_F$ is a \defword{labelled induced subgraph} of~$G$ if~$F$ is isomorphic to an induced subgraph~$G'$ of~$G$ such that there is an isomorphism which maps each vertex~$v$ of~$F$ to a vertex~$w$ of~$G'$ with $l_F(v)\leq l_G(w)$.
If $(W,\leq)$ is a well-quasi-order, then it is not possible for a graph class~${\cal G}$ to contain an infinite sequence of labelled graphs that is strictly-decreasing with respect to the labelled induced subgraph relation.
We say that~${\cal G}$ is well-quasi-ordered by the \defword{labelled} induced subgraph relation if for \emph{every} well-quasi-order $(W,\leq)$ the class~${\cal G}$ contains no infinite anti-chains of labelled graphs.

\begin{observation}\label{o-lwqo}
Every graph class that is well-quasi-ordered by the labelled induced subgraph relation is well-quasi-ordered by the induced subgraph relation.
\end{observation}

Daligault, Rao and Thomass{\'e} proved the following result.

\begin{theorem}[\cite{DRT10}]\label{t-drt10}
Every hereditary class of graphs that is well-quasi-ordered by the labelled induced subgraph relation is finitely defined.
\end{theorem}

By Theorem~\ref{t-drt10} it is easy to prove that there exist hereditary graph classes that are well-quasi-ordered by the induced subgraph relation but not by the labelled induced subgraph relation.
Korpelainen, Lozin and Razgon~\cite{KLR13} gave the class of linear forests as an example (see also Example~\ref{e-paths} below).
The same authors conjectured that if a hereditary class of graphs~${\cal G}$ is defined by a finite set of forbidden induced subgraphs, then~${\cal G}$ is well-quasi-ordered by the induced subgraph relation if and only if it is well-quasi-ordered by the labelled induced subgraph relation.
However, Brignall, Engen and Vatter~\cite{BEV18} recently found a counterexample for this conjecture.

\begin{theorem}[\cite{BEV18}]\label{t-bev18}
There exists a graph class~${\cal G}^*$ with $|{\cal F}_{{\cal G}^*}|=14$ that is well-quasi-ordered by the induced subgraph relation but not by the labelled induced subgraph relation.
\end{theorem}

Theorem~\ref{t-bev18} leads to the following open problem.

\begin{oproblem}\label{o-13}
Does there exist a hereditary graph class~${\cal G}$ with $|{\cal F}_{\cal G}|\leq 13$ that is well-quasi-ordered by the induced subgraph relation but not by the labelled induced subgraph relation?
\end{oproblem}

We consider an approach similar to one used for boundedness of clique-width.
A graph operation~$\gamma$ \defword{preserves} well-quasi-orderability by the labelled induced subgraph relation if, for every finite constant~$k$ and every graph class~${\cal G}$, every graph class~${\cal G}'$ that is $(k,\gamma)$-obtained from~${\cal G}$ is well-quasi-ordered by this relation if and only if~${\cal G}$ is.
We also say that a graph property~$\pi$ \defword{preserves} well-quasi-orderability by the labelled induced subgraph relation if for every graph class~${\cal G}$, the subclass of~${\cal G}$ with property~$\pi$ is well-quasi-ordered by the labelled induced subgraph relation if and only if this is the case for~${\cal G}$.

\paragraph{Facts about well-quasi orderability:}

\begin{enumerate}[\bf F{a}ct 1.]
\item \label{fact:del-vert2}Vertex deletion preserves well-quasi-orderability by the labelled induced subgraph relation~\cite{DLP18}.\\[-1em]

\begin{sloppypar}
\item \label{fact:comp2}Subgraph complementation preserves well-quasi-orderability by the labelled induced subgraph relation~\cite{DLP18}.\\[-1em]
\end{sloppypar}

\item \label{fact:bip2}Bipartite complementation preserves well-quasi-orderability by the labelled induced subgraph relation~\cite{DLP18}.\\[-1em]

\item \label{fact:prime2}Being prime preserves well-quasi-orderability by the labelled induced subgraph relation for hereditary classes~\cite{AL15}.
\end{enumerate}

For labelled well-quasi-orders, there is no analogue to Fact~\ref{fact:2-conn} (on $2$-connectivity) and~Fact~\ref{fact:subdiv} (on edge subdivision) as illustrated by the following counterexample.

\begin{example}\label{e-paths}
Let~${\cal F}$ be the (hereditary) class of linear forests.
The class~${\cal F}$ contains the class~${\cal P}$ of all paths on at least two vertices.
If we label the end-vertices of every path in~${\cal P}$ with one label and all other vertices with a second label incomparable with the first, we obtain an infinite anti-chain with respect to the labelled induced subgraph relation (see also \figurename~\ref{fig:Pi-antichain}).
Hence~${\cal F}$ is not well-quasi-ordered by the labelled induced subgraph relation.
However, the restriction of~${\cal F}$ to $2$-connected graphs is the empty class, which is well-quasi-ordered by the labelled induced subgraph relation.
Moreover, every graph of~${\cal F}$ has maximum degree at most~$2$.
However, every path of~${\cal P}$ can be obtained by repeatedly subdividing~$P_2$, and the class~$\{P_2\}$ is well-quasi-ordered by the labelled induced subgraph relation.
We conclude that Facts~\ref{fact:2-conn} and~\ref{fact:subdiv} for clique-width do not have a counterpart for well-quasi-orderability by the labelled induced subgraph relation.
\end{example}

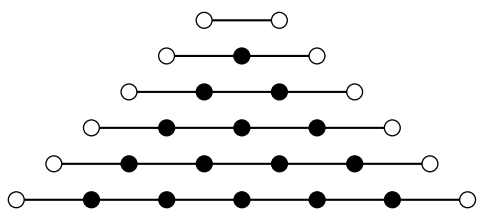
\begin{figure}
\begin{center}
\begin{tikzpicture}[scale=1]
\GraphInit[vstyle=Simple]
\SetVertexSimple[MinSize=6pt]
\SetVertexSimple[MinSize=6pt,FillColor=white]
\Vertex[x=0,y=0]{a0}
\Vertex[x=1,y=0]{az}
\Edges(a0,az)
\end{tikzpicture}

\begin{tikzpicture}[scale=1]
\GraphInit[vstyle=Simple]
\SetVertexSimple[MinSize=6pt]
\Vertex[x=1,y=0]{a1}
\SetVertexSimple[MinSize=6pt,FillColor=white]
\Vertex[x=0,y=0]{a0}
\Vertex[x=2,y=0]{az}
\Edges(a0,a1,az)
\end{tikzpicture}

\begin{tikzpicture}[scale=1]
\GraphInit[vstyle=Simple]
\SetVertexSimple[MinSize=6pt]
\Vertex[x=1,y=0]{a1}
\Vertex[x=2,y=0]{a2}
\SetVertexSimple[MinSize=6pt,FillColor=white]
\Vertex[x=0,y=0]{a0}
\Vertex[x=3,y=0]{az}
\Edges(a0,a1,a2,az)
\end{tikzpicture}

\begin{tikzpicture}[scale=1]
\GraphInit[vstyle=Simple]
\SetVertexSimple[MinSize=6pt]
\Vertex[x=1,y=0]{a1}
\Vertex[x=2,y=0]{a2}
\Vertex[x=3,y=0]{a3}
\SetVertexSimple[MinSize=6pt,FillColor=white]
\Vertex[x=0,y=0]{a0}
\Vertex[x=4,y=0]{az}
\Edges(a0,a1,a2,a3,az)
\end{tikzpicture}

\begin{tikzpicture}[scale=1]
\GraphInit[vstyle=Simple]
\SetVertexSimple[MinSize=6pt]
\Vertex[x=1,y=0]{a1}
\Vertex[x=2,y=0]{a2}
\Vertex[x=3,y=0]{a3}
\Vertex[x=4,y=0]{a4}
\SetVertexSimple[MinSize=6pt,FillColor=white]
\Vertex[x=0,y=0]{a0}
\Vertex[x=5,y=0]{az}
\Edges(a0,a1,a2,a3,a4,az)
\end{tikzpicture}

\begin{tikzpicture}[scale=1]
\GraphInit[vstyle=Simple]
\SetVertexSimple[MinSize=6pt]
\Vertex[x=1,y=0]{a1}
\Vertex[x=2,y=0]{a2}
\Vertex[x=3,y=0]{a3}
\Vertex[x=4,y=0]{a4}
\Vertex[x=5,y=0]{a5}
\SetVertexSimple[MinSize=6pt,FillColor=white]
\Vertex[x=0,y=0]{a0}
\Vertex[x=6,y=0]{az}
\Edges(a0,a1,a2,a3,a4,a5,az)
\end{tikzpicture}
\end{center}
\caption[An anti-chain of paths under the labelled induced subgraph relation.]{\label{fig:Pi-antichain}An anti-chain of paths under the labelled induced subgraph relation.
The two colours are incomparable.}
\end{figure}

As a final remark in this section, we note that it is easy to verify that graph classes of bounded neighbourhood diversity (introduced in~\cite{La12}) have bounded clique-width and are well-quasi-ordered by the labelled induced subgraph relation.
Moreover, the same property also holds for graph classes of bounded uniformicity (introduced in~\cite{KL11}) or bounded lettericity (introduced in~\cite{Petkovsek2002}); uniformicity and lettericity are more general than neighbourhood diversity.

\subsection{Results for Hereditary Graph Classes}\label{s-wqoh}

We now survey known results for well-quasi-orderability of hereditary graphs by the induced subgraph relation.
As we shall see, all known results agree with Conjecture~\ref{c-f}.
We start with a result of Damaschke.

\begin{theorem}[\cite{Da90}]\label{t-da90}
Let~$H$ be a graph.
The class of $H$-free graphs is well-quasi-ordered by the induced subgraph relation if and only if $H \ssi P_4$.
\end{theorem}

In fact, the same classification holds for the labelled induced subgraph relation~\cite{AL15}, which means that
if there is a hereditary class~${\cal G}$ which gives a positive answer to Open Problem~\ref{o-13}, then $|{\cal F}_{\cal G}|\geq 2$.
We also note that the classification of Theorem~\ref{t-da90} coincides with the one of Theorem~\ref{thm:single} for boundedness of clique-width.
In order to increase our understanding of well-quasi-orderability by the induced subgraph relation we can follow the same approaches as done in Section~\ref{s-hereditary} for clique-width.
However, considerably less work has been done on this subject.

Just as in Section~\ref{s-hereditary}, we can first restrict ourselves to $H$-free graphs contained in some other hereditary graph class.
In particular, results for $H$-free bipartite graphs, such as those in~\cite{KL11b}, have shown to be useful.
For instance, they have been used to prove results on well-quasi-orderability for $(H_1,H_2)$-free graphs~\cite{KL11}.
Combining the results for $H$-free bipartite and $H$-free triangle-free graphs of~\cite{KL11b,KL11} with the results of~\cite{AL15,Di92,DLP17} and Ramsey's Theorem for the case when $H=sP_1$ $(s\geq 1)$ yields the following two classifications (see~\cite{DLP17} for further explanation).

\begin{theorem}[\cite{DLP17}]\label{t-bipartite2}
Let~$H$ be a graph.
The class of $H$-free bipartite graphs is well-quasi-ordered by the induced subgraph relation if and only if $H=sP_1$ for some $s\geq 1$ or~$H$ is an induced subgraph of $P_1+\nobreak P_5$, $P_2+\nobreak P_4$ or~$P_6$.
\end{theorem}

\begin{theorem}[\cite{DLP17}]\label{t-triangle2}
Let~$H$ be a graph.
The class of $(K_3,H)$-free graphs is well-quasi-ordered by the induced subgraph relation if and only if $H=sP_1$ for some $s\geq 1$ or~$H$ is an induced subgraph of $P_1+\nobreak P_5$, $P_2+\nobreak P_4$, or~$P_6$.
\end{theorem}

We note that the classifications of Theorem~\ref{t-bipartite2} and~\ref{t-triangle2} coincide.
In contrast, we recall that it is not yet clear if the classifications for boundedness of clique-width on $H$-free bipartite graphs and $(K_3,H)$-free graphs also coincide; see Open Problem~\ref{oprob:twographs}.

\begin{sloppypar}
We now present the state-of-the-art summary for well-quasi-orderability for classes on $(H_1,H_2)$-free graphs, which is obtained by combining results from~\cite{AL15,Di92,DLP18,DLP17,KL11b,KL11}.
Note that Theorem~\ref{t-wqowqo} implies Theorem~\ref{t-triangle2}.
\end{sloppypar}

\begin{theorem}[\cite{DLP17}]\label{t-wqowqo}
Let~${\cal G}$ be a class of graphs defined by two forbidden induced subgraphs.
Then:
\begin{enumerate}
\item ${\cal G}$ is well-quasi-ordered by the \defword{labelled} induced subgraph relation if it is equivalent\footnote{Equivalence is defined in the same way as for clique-width (see Footnote~\ref{footnote:equivalence}).
If two classes are equivalent, then one of them is well-quasi-ordered by the induced subgraph relation if and only if the other one is~\cite{KL11}.} to a class of $(H_1,H_2)$-free graphs such that one of the following holds:
\begin{enumerate}[(i)]
\item $H_1$ or~$H_2 \ssi P_4$ \\[-1.7em]
\item $H_1=K_s$ and $H_2=tP_1$ for some $s,t\geq 1$ \\[-1.7em]
\item $H_1 \ssi \paw$ and $H_2 \ssi P_1+\nobreak P_5, P_2+\nobreak P_4$ or~$P_6$ \\[-1.7em]
\item $H_1 \ssi \diamondgraph$ and $H_2 \ssi P_2+\nobreak P_3$ or~$P_5$.\\[-9pt]
\end{enumerate}
\item ${\cal G}$ is not well-quasi-ordered by the induced subgraph relation if it is equivalent to a class of $(H_1,H_2)$-free graphs such that one of the following holds:
\begin{enumerate}[(i)]
\item neither~$H_1$ nor~$H_2$ is a linear forest \\[-1.7em]
\item $H_1 \si K_3$ and $H_2 \si 3P_1+\nobreak P_2, 3P_2$ or~$2P_3$ \\[-1.7em]
\item $H_1 \si C_4$ and $H_2 \si 4P_1$ or~$2P_2$ \\[-1.7em]
\item $H_1 \si \diamondgraph$ and $H_2 \si 4P_1, P_2+\nobreak P_4$ or~$P_6$ \\[-1.7em]
\item $H_1 \si \gem$ and $H_2 \si P_1+\nobreak 2P_2$.
\end{enumerate}
\end{enumerate}
\end{theorem}
Theorem~\ref{t-wqowqo} does not cover six cases, which are all still open (see also~\cite{DLP17}).

\begin{oproblem}\label{o-wqo}
Is the class of $(H_1,H_2)$-free graphs well-quasi-ordered by the induced subgraph relation when:\\[-1.5em]
\begin{enumerate}[(i)]
\item $H_1=\diamondgraph$ and $H_2 \in \{P_1+\nobreak 2P_2, P_1+\nobreak P_4\}$\\[-1.5em]
\item $H_1=\gem$ and $H_2 \in \{P_1+\nobreak P_4, 2P_2, P_2+\nobreak P_3, P_5\}$.
\end{enumerate}
\end{oproblem}

It follows from Theorems~\ref{thm:classification2} and~\ref{t-wqowqo} that the class of $(\overline{P_1+P_4},P_2+\nobreak P_3)$-free graphs is the only class of $(H_1,H_2)$-free graphs left for which Conjecture~\ref{c-f} still needs to be verified (see also~\cite{DLP18}).

\begin{oproblem}\label{o-con}
Is Conjecture~\ref{c-f} true for the class of $(H_1,H_2)$-free graphs when $H_1=\overline{P_1+P_4}$ and $H_2=P_2+\nobreak P_3$?
\end{oproblem}

Finally, instead of the induced subgraph relation or the minor relation, one can also consider other containment relations.
Ding~\cite{Di92} proved that for a graph~$H$, the class of $H$-subgraph-free graphs is well-quasi-ordered by the subgraph relation if and only if~$H$ is a linear forest.
This result can be readily generalized.

\begin{sloppypar}
\begin{theorem}\label{t-di92}
Let $\{H_1,\ldots,H_p\}$ be a finite set of graphs.
The class of $(H_1,\ldots,H_p)$-subgraph-free graphs is well-quasi-ordered by the subgraph relation if and only if~$H_i$ is a linear forest for some $i \in \{1,\ldots,p\}$.
\end{theorem}
\end{sloppypar}

\begin{proof}
If~$H_i$ is a linear forest for some $i \in \{1,\ldots,p\}$, then we can apply the result of Ding~\cite{Di92}.
Now suppose that every~$H_i$ either has a cycle or an induced claw. 
We let~$g$ be the maximum girth over all~$H_i$ that contain a cycle. 
Then the set of cycles of length at least $g+1$ is an infinite antichain of $(H_1,\ldots,H_p)$-free graphs with respect to the subgraph relation.
\end{proof}

Kami\'nski, Raymond and Trunck~\cite{KRT18} and B{\l}asiok, Kami\'nski, Raymond and Trunck~\cite{BKRT18} gave classifications for the contraction relation and induced minor relation, respectively. 
We note that the connectivity condition in Theorem~\ref{t-krt18} is natural, as the edgeless graphs form an antichain under the contraction relation.
We refer to \figurename~\ref{fig:gem-crossed-house} for pictures of the graphs $\overline{2P_1+P_3}$ and $\overline{P_1+P_4}$ in Theorem~\ref{t-b2}.

\begin{theorem}[\cite{KRT18}]\label{t-krt18}
Let~$H$ be a graph.
The class of connected $H$-contraction-free graphs is well-quasi-ordered by the contraction relation if and only if $H \in \{C_3,\diamondgraph,P_1,P_2,P_3\}$.
\end{theorem}

\begin{theorem}[\cite{BKRT18}]\label{t-b2}
Let~$H$ be a graph.
The class of $H$-induced-minor-free graphs is well-quasi-ordered by the induced-minor relation if and only if $H \ssi \overline{2P_1+P_3}$ or $H \ssi \overline{P_1+P_4}$.
\end{theorem}

We pose the following two open problems.

\begin{oproblem}\label{o-o2}
Determine for which pairs of graphs $(H_1,H_2)$ the class of connected $(H_1,H_2)$-contraction-free graphs is well-quasi-ordered by the contraction relation.
\end{oproblem}

\begin{oproblem}\label{o-o3}
Determine for which pairs of graphs $(H_1,H_2)$ the class of $(H_1,H_2)$-induced-minor-free graphs is well-quasi-ordered by the induced minor relation.
\end{oproblem}

For containment relations other than the induced subgraph relation we can ask the following question: does there exist a containment-closed graph class of unbounded clique-width that is well-quasi-ordered by the same containment relation?  
The Robertson-Seymour Theorem~\cite{RS04-Wagner} tells us that that the class of all (finite) graphs is well-quasi-ordered by the minor relation.
Hence, if we forbid minors, we can consider the class of all finite graphs, which has unbounded clique-width.
By Theorems~\ref{t-finite}.\ref{t-finite-subgraph} and~\ref{t-di92}, we would need to forbid an infinite set of graphs for the subgraph relation to find a positive answer to this question.
A \defword{clique-cactus graph}  is a graph in which each block is either a complete graph or a cycle (these graphs are also known as cactus block graphs).
The class of $\diamondgraph$-contraction-free graphs coincides with the class of clique-cactus graphs~\cite{KRT18}.
As complete graphs and cycles have clique-width at most~$2$ and~$4$, respectively, clique-cactus graphs have bounded clique-width due to Fact~\ref{fact:prime}.
Hence, by Theorem~\ref{t-krt18} we would need to forbid a set of at least two graphs for the contraction relation (when considering connected graphs).
The classification in Theorem~\ref{t-b2} coincides with the classification in Theorem~\ref{thm:cw-ind-min} for boundedness of the clique-width of $H$-induced-minor-free graphs.
Hence we would also need to forbid a set of at least two graphs for the induced minor relation.
We note that the hereditary graph class given in~\cite{LRZ18} (that is well-quasi-ordered by the induced subgraph relation, but has unbounded clique-width) is not closed under contractions, subgraphs or induced minors.
Hence, this class does not give a positive answer to the question for contractions, subgraphs or induced minors.

\section{Variants of Clique-Width}\label{s-con}

We have surveyed results and techniques for proving (un)boundedness of clique-width for various families of hereditary graph classes and stated a number of open problems.
We conclude our paper with a brief discussion of some other variants of clique-width.
Lozin and Rautenbach~\cite{LR07} introduced the notion of relative clique-width, whose definition is more consistent with the definition of treewidth. 
Computing relative clique-width is \NP-hard, as shown by M\"uller and Urner~\cite{MU10}, but the concept has not been studied for hereditary graph classes.
Courcelle~\cite{Co14} and F\"urer~\cite{Fu17} defined symmetric clique-width and multi-clique-width, respectively.
Both these width parameters are equivalent to clique-width~\cite{Co14,Fu17}.
As this survey focuses on boundedness of clique-width, we therefore do not discuss these parameters any further here.
Instead we focus on two other variants, namely linear clique-width (Section~\ref{s-lcw}) and power-bounded clique-width (Section~\ref{s-pcw}).

\subsection{Linear Clique-Width}\label{s-lcw}

\defword{Linear clique-width}~\cite{GW05,LR07}, also called \defword{sequential clique-width}, is defined in the same way as clique-width except that in Operation~\ref{operation:disjoint-union} (the disjoint union operation) of the definition of clique-width, at least one of the two graphs must consist of a single vertex.
Just as clique-width is equivalent to NLC-width and rank-width, linear clique-width is equivalent to linear NLC-width~\cite{GW05} and linear rank-width (see, for example,~\cite{Ou17}).\footnote{We note that the corresponding variants for directed graphs were recently introduced by Gurski and Rehs~\cite{GR18}.}
Moreover, just as is the case for clique-width, the notion of linear clique-width is also not well understood, and similar approaches to those for clique-width have been followed.
To illustrate this, the following analogous results to those for clique-width are known.
Computing linear clique-width is \NP-hard~\cite{FRRS09} for general graphs, but it is polynomial-time solvable for forests~\cite{AK15} and distance-hereditary graphs~\cite{AKK17}.
Moreover, graphs of linear clique-width at most~$3$ can be recognized in polynomial time~\cite{HMP11}, but the computational complexity of recognizing graphs of linear clique-width at most~$c$ is unknown for $c\geq 4$ (see~\cite{HMP12} for some partial results for $c=4$).
Another analogous result is due to Gurski and Wanke who proved the following theorem (compare to Theorem~\ref{t-gw1}).

\begin{theorem}[\cite{GW07}]\label{t-gw2}
A class of graphs~${\cal G}$ has bounded path-width if and only if the class of the line graphs of graphs in~${\cal G}$ has bounded linear clique-width.
\end{theorem}

By definition, every graph class of bounded linear clique-width has bounded clique-width, but the reverse implication does not hold.
For example, recall that every $P_4$-free graph has clique-width at most~$2$ by Theorem~\ref{thm:single} and that every tree has clique-width at most~$3$ by Proposition~\ref{p-forest}.
In contrast, Gurski and Wanke~\cite{GW05} proved that the class of $P_4$-free graphs and even the class of complete binary trees have unbounded linear clique-width.
This led Brignall, Korpelainen and Vatter to consider hereditary subclasses of $P_4$-free graphs.
They proved the following dichotomy result.

\begin{theorem}[\cite{BKV15}]\label{t-triviallyperfect}
A hereditary subclass of $P_4$-free graphs has bounded linear clique-width if and only if it contains neither the class of $(C_4,P_4)$-free graphs nor the class of $(2P_2,P_4)$-free graphs.
\end{theorem}

We note that $(C_4,P_4)$-free graphs are also known as the \defword{trivially perfect} or \defword{quasi-threshold} graphs.

To obtain an analogous result to Theorem~\ref{thm:single}, we state the following two results.
The first one is due to Gurski.
The second can be easily derived from known results.

\begin{theorem}[\cite{Gurski06}]\label{t-gu06}
A graph has linear clique-width at most~$2$ if and only if it is $(2P_2,\overline{2P_3},P_4)$-free.
\end{theorem}

\begin{theorem}\label{thm:single2}
Let~$H$ be a graph.
The class of $H$-free graphs has bounded linear clique-width if and only if~$H$ is an induced subgraph of~$P_1+\nobreak P_2$ or~$P_3$.
Furthermore, $(P_1+\nobreak P_2)$-free graphs and $P_3$-free graphs have linear clique-width at most~$2$ and~$3$, respectively.
\end{theorem}

\begin{proof}
By Theorem~\ref{thm:single} it suffices to consider the case when~$H$ is an induced subgraph of~$P_4$
and by Theorem~\ref{t-triviallyperfect} we may assume that $H \neq P_4$.
Let~$G$ be an $H$-free graph.
If $H \ssi P_3$, then every connected component of~$G$ is a complete graph.
Complete graphs are readily seen to have linear clique-width at most~$2$. 
Hence,~$G$ has linear clique-width at most~$3$ (after creating each connected component, we relabel all of its vertices to a third label).
The only remaining case is $H=P_1+\nobreak P_2$.
Since~$P_1+\nobreak P_2$ is an induced subgraph of $2P_2$, $\overline{2P_3}$ and~$P_4$, Theorem~\ref{t-gu06} implies that~$G$ has linear clique-width at most~$2$.
\end{proof}

Theorem~\ref{thm:single2} leads to the following open problem.

\begin{oproblem}\label{o-power}
Determine for which pairs of graphs~$(H_1,H_2)$ the class of $(H_1,H_2)$-free graphs has bounded linear clique-width.
\end{oproblem}

Just as is the case for clique-width, we expect that results on boundedness of linear clique-width for $H$-free bipartite graphs would be useful for solving Open Problem~\ref{o-power}.
We therefore also pose the following open problem.

\begin{oproblem}\label{o-power2}
Determine for which graphs~$H$ the class of $H$-free bipartite graphs has bounded linear clique-width.
\end{oproblem}

Finally, we can also prove an analogous result to Theorem~\ref{t-line}.

\begin{sloppypar}
\begin{theorem}\label{t-linelinear}
Let~$\{H_1,\ldots,H_p\}$ be a finite set of graphs.
Then the class of $(H_1,\ldots,H_p)$-free line graphs has bounded linear clique-width if and only if $H_i\in {\cal T}$ for some $i \in \{1,\ldots,p\}$.
\end{theorem}
\end{sloppypar}

\begin{proof}
First suppose that $H_i\in {\cal T}$ for some $i \in \{1,\ldots,p\}$.
By definition of~${\cal T}$, it follows that~$H_i$ is the line graph of some graph $F_i\in {\cal S}$.
We repeat the arguments of the proof of Theorem~\ref{t-line} to find that the class of $F_i$-subgraph-free graphs has bounded path-width.
Then, by Theorem~\ref{t-gw2}, the class of $H_i$-free graphs, and thus the class of $(H_1,\ldots,H_p)$-free graphs, has bounded linear clique-width.
Now suppose that $H_i\notin {\cal T}$ holds for every $i \in \{1,\ldots,p\}$.
By Theorem~\ref{t-line}, the class of $(H_1,\ldots,H_p)$-free line graphs has unbounded clique-width, and thus unbounded linear clique-width.
\end{proof}

\subsection{Power-Bounded Clique-Width}\label{s-pcw}

Recall that the $r$-th power of~$G$ $(r\geq 1)$ is the graph with vertex set~$V(G)$ and an edge between two vertices~$u$ and~$v$ if and only if~$u$ and~$v$ are at distance at most~$r$ from each other in~$G$.
Gurski and Wanke~\cite{GW09} proved that the clique-width of the $r$-th power of a tree is at most $r+2+\max\{\lfloor\frac{r}{2}\rfloor -1,0\}$ and that the $r$-th power of a graph~$G$ has clique-width at most $2(r+1)^{\tw(G)+1}-1$.

A graph class~${\cal G}$ has \defword{power-bounded clique-width} if there is a constant~$r$ such that the graph class consisting of all $r$-th powers of all graphs from~${\cal G}$ has bounded clique-width; otherwise~${\cal G}$ has \defword{power-unbounded clique-width}.
Hence, if a graph class has bounded clique-width, it has power-bounded clique-width (we can take $r=1$).
The reverse implication does not hold.
This follows, for example, from a comparison of Theorem~\ref{thm:single} with the following classification for $H$-free graphs of Bonomo, Grippo, Milani\v{c} and Safe.

\begin{theorem}[\cite{BGMS14}]\label{t-power1}
Let~$H$ be a graph.
Then the class of $H$-free graphs has power-bounded clique-width if and only if~$H$ is a linear forest.
\end{theorem}

Bonomo, Grippo, Milani\v{c} and Safe also proved the following classification for $(H_1,H_2)$-free graphs when both~$H_1$ and~$H_2$ are connected.

\begin{theorem}[\cite{BGMS14}]\label{t-power2}
Let~$H_1$ and~$H_2$ be two connected graphs.
Then the class of $(H_1,H_2)$-free graphs has power-bounded clique-width if and only if
\begin{enumerate}[(i)]
\item at least one of~$H_1$ and~$H_2$ is a path, or\\[-2em]
\item $H_1=S_{1,i,j}$ for some $i,j\geq 1$ and $H_2=T_{0,i',j'}$ for some $i',j'\geq 0$.
\end{enumerate}
\end{theorem}

The case when~$H_1$ or~$H_2$ is disconnected has not yet been settled.

\begin{oproblem}\label{o-power3}
Determine for which pairs of graphs~$(H_1,H_2)$ the class of $(H_1,H_2)$-free graphs has power-bounded clique-width.
\end{oproblem}

We note that analogous results to Theorems~\ref{t-unitinterval} and~\ref{t-bipper} exist for power-bounded clique-width; that is, the classes of bipartite permutation graphs and unit interval graphs have power-unbounded clique-width~\cite{BGMS14}.
For more open problems on power-bounded clique-width, we refer to~\cite{BGMS14}.

\thankyou{We thank Robert Brignall, Vincent Vatter and an anonymous reviewer for helpful comments.
The work was supported by the Leverhulme Trust Research Project Grant RPG-2016-258.}

\bibliography{mybib-short}
\myaddress
\end{document}